\documentclass[reqno]{amsart}
\usepackage{amsfonts}
\usepackage{amsmath}

\newtheorem{theorem}{Theorem}
\theoremstyle{plain}

\newtheorem{corollary}{Corollary}

\newtheorem{definition}{Definition}
\newtheorem{example}{Example}

\newtheorem{lemma}{Lemma}

\newtheorem{proposition}{Proposition}
\newtheorem{remark}{Remark}

\DeclareMathOperator{\Div}{div}
 \numberwithin{equation}{section}
\input{tcilatex}

\begin{document}
\title[Homogenization theory]{Homogenization in algebras with mean value}
\author{Jean Louis Woukeng}
\address{Department of Mathematics and Computer Science, University of
Dschang, P.O. Box 67, Dschang, Cameroon}
\email{jwoukeng@yahoo.fr}
\date{February, 2013}

\begin{abstract}
In several works, the theory of strongly continuous groups is used to build
a framework for solving stochastic homogenization problems. Following this
idea, we construct a detailed and comprehensive theory of homogenization.
This enables to solve homogenization problems in algebras with mean value,
regardless of whether they are ergodic or not, thereby responding
affirmatively to the question raised by Zhikov and Krivenko [V.V. Zhikov,
E.V. Krivenko, Homogenization of singularly perturbed elliptic operators.
Matem. Zametki, 33 (1983) 571-582 (english transl.: Math. Notes, 33 (1983)
294-300)] to know whether it is possible to homogenize problems in
nonergodic algebras. We also state and prove a compactness result for Young
measures in these algebras. As an important achievement we study the
homogenization problem associated with a stochastic Ladyzhenskaya model for
incompressible viscous flow, and we present and solve a few examples of
homogenization problems related to nonergodic algebras.
\end{abstract}

\maketitle

\section{Introduction and main results}

The theory of strongly continuous $N$-parameter groups of operators is a
very important tool in solving partial differential equations (PDEs). In 
\cite{Vo-Khac} (see also \cite{Blot}), it is used to solve PDEs in spaces of
almost periodic functions. In the stochastic homogenization theory, one
constructs through a dynamical system, a strongly continuous $N$-parameters
group. One then uses its infinitesimal generators to build up a framework
for solving random homogenization problems. We refer, e.g., to the works 
\cite{10, 20, Wright1} (see also \cite{AA}) for an exposition of this idea.

Given an algebra with mean value, the uniform continuity property of its
elements allows the construction [on the generalized Besicovitch spaces
associated to this algebra] of a strongly continuous $N$-parameters group.
We therefore rely on the properties of this group to construct as in the
stochastic case, a comprehensive and detailed framework for solving
deterministic homogenization problems as well as homogenization problems
related to stochastic partial differential equations (SPDEs). The results
obtained generalize the already existing ones, and provide more clarity and
conciseness to the latter.

Dealing now with deterministic homogenization theory, it has been so far
applied to solve only homogenization problems in ergodic algebras. In that
direction we refer, e.g., to the papers \cite{Casado, 26, NgWou1, NgWou2,
CMP, AIM, CMA, CPAA, ACAP, NoDEA}\ in which only ergodic algebras are
considered. In this paper we show how one can derive general homogenization
results in algebras with mean value through the theory of strongly
continuous groups of transformation. One very important achievement will be
to work out the homogenization problems related to a stochastic
Ladyzhenskaya model for incompressible non-Newtonian fluid, without help of
any ergodicity assumption. For the sake of clarity, let us give here below
some of our main results.

\begin{theorem}
\label{t1.1}Let $1<p<\infty $. Let $(u_{\varepsilon })_{\varepsilon \in E}$
be a bounded sequence in $W^{1,p}(Q)$. Then there exist a subsequence $%
E^{\prime }$ of $E$, and a couple $(u_{0},u_{1})\in
W^{1,p}(Q;I_{A}^{p})\times L^{p}(Q;\mathcal{B}_{\#A}^{1,p})$ such that, as $%
E^{\prime }\ni \varepsilon \rightarrow 0$, 
\begin{equation*}
u_{\varepsilon }\rightarrow u_{0}\ \text{in }L^{p}(Q)\text{-weak }\Sigma 
\text{;}
\end{equation*}%
\begin{equation*}
\frac{\partial u_{\varepsilon }}{\partial x_{i}}\rightarrow \frac{\partial
u_{0}}{\partial x_{i}}+\frac{\overline{\partial }u_{1}}{\partial y_{i}}\text{%
\ in }L^{p}(Q)\text{-weak }\Sigma \text{, }1\leq i\leq N\text{.}
\end{equation*}
\end{theorem}

The difference between the above result and other already existing results
is that no ergodicity assumption is required on the algebra with mean value $%
A$, in contrast to what had always been done so far. It should be noted that
the proof of the above result in the ergodic case relies heavily on the
ergodicity assumption made on $A$, see for instance the papers \cite{Casado,
26, CMP, NA}. In the general situation that we consider in this work, the
proof is based on a de Rham type result formulated as follows:\medskip 
\begin{equation*}
\left\{ 
\begin{array}{l}
\text{Let }\mathbf{v}\in (\mathcal{B}_{A}^{p})^{N}\text{ satisfying} \\ 
\int_{\Delta (A)}\widehat{\mathbf{v}}\cdot \widehat{\mathbf{g}}d\beta =0%
\text{\ for all }\mathbf{g}\in \mathcal{V}_{\Div}=\{\mathbf{f}\in (\mathcal{D%
}_{A}(\mathbb{R}^{N}))^{N}:\overline{\Div}_{y}\mathbf{f}=0\}\text{.} \\ 
\text{Then there exists }u\in \mathcal{B}_{\#A}^{1,p}\text{ such that }%
\mathbf{v}=\overline{D}_{y}u\text{.}%
\end{array}%
\right.
\end{equation*}

As a consequence of the preceding result we have the following important
result which is suitable for the homogenization of SPDEs in algebras with
mean value.

\begin{theorem}
\label{t1.2}Let $1<p<\infty $. Let $Q$ be an open subset in $\mathbb{R}^{N}$%
. Let $A=A_{y}\odot A_{\tau }$ be any product algebra with mean value on $%
\mathbb{R}^{N}\times \mathbb{R}$. Finally, let $(u_{\varepsilon
})_{\varepsilon \in E}$ be a sequence of $L^{p}(0,T;W^{1,p}(Q))$-values
random variables satisfying the following estimate: 
\begin{equation*}
\sup_{\varepsilon \in E}\mathbb{E}\left\Vert u_{\varepsilon }\right\Vert
_{L^{p}(0,T;W^{1,p}(Q))}^{p}<\infty .
\end{equation*}%
Then there exist a subsequence $E^{\prime }$ of $E$ and a couple of random
variables $(u_{0},u_{1})$ with $u_{0}\in L^{p}(\Omega
;L^{p}(0,T;W^{1,p}(Q;I_{A}^{p})))$ and $u_{1}\in L^{p}(\Omega ;L^{p}(Q_{T};%
\mathcal{B}_{A_{\tau }}^{p}(\mathbb{R}_{\tau };\mathcal{B}_{\#A_{y}}^{1,p})))
$ such that, as $E^{\prime }\ni \varepsilon \rightarrow 0$, 
\begin{equation*}
u_{\varepsilon }\rightarrow u_{0}\ \text{in }L^{p}(Q_{T}\times \Omega )\text{%
-weak }\Sigma 
\end{equation*}%
and%
\begin{equation*}
\frac{\partial u_{\varepsilon }}{\partial x_{i}}\rightarrow \frac{\partial
u_{0}}{\partial x_{i}}+\frac{\overline{\partial }u_{1}}{\partial y_{i}}\text{%
\ in }L^{p}(Q_{T}\times \Omega )\text{-weak }\Sigma \text{ (}1\leq i\leq N%
\text{).}
\end{equation*}
\end{theorem}

The next result is a compactness result related to Young measures in
algebras with mean value. It reads as follows.

\begin{theorem}
\label{t1.3}Let $Q$ be an open bounded subset of $\mathbb{R}^{N}$. Let $%
1\leq p<\infty $, and let $A$ be an algebra with mean value on $\mathbb{R}%
_{y}^{N}$. Finally let $(u_{\varepsilon })_{\varepsilon \in E}$ be a bounded
sequence in $L^{p}(Q;\mathbb{R}^{m})$. There exist a subsequence $E^{\prime }
$ from $E$ and a family $\nu =(\nu _{x,s})_{x\in Q,s\in \Delta (A)}\in
L^{\infty }(Q\times \Delta (A);\mathcal{P}(\mathbb{R}^{m}))$ such that, as $%
E^{\prime }\ni \varepsilon \rightarrow 0$, 
\begin{equation*}
\int_{Q}\Phi \left( x,\frac{x}{\varepsilon _{1}},u_{\varepsilon }(x)\right)
dx\rightarrow \int_{Q}\int_{\Delta (A)}\int_{\mathbb{R}^{m}}\widehat{\Phi }%
(x,s,\lambda )d\nu _{x,s}(\lambda )d\beta (s)dx
\end{equation*}%
for all $\Phi \in E_{p}$.
\end{theorem}

A special version of the above theorem can be found in \cite{Frid}. Our
approach is based on a result by Valadier \cite{Valadier2} regarding the
disintegration of measures. As expected the above results have a number of
applications. By way of illustrating, we give two of them in this work. The
most important one is the homogenization of a nonlinear SPDE. In this
regard, we prove the following

\begin{theorem}
\label{t1.4}Assume $p\geq 3$. For each $\varepsilon >0$ let $\mathbf{u}%
_{\varepsilon }$ be the unique solution of the following stochastic PDE:%
\begin{equation*}
\left\{ 
\begin{array}{l}
d\mathbf{u}_{\varepsilon }+(P^{\varepsilon }\mathbf{u}_{\varepsilon }+%
\mathcal{A}^{\varepsilon }\mathbf{u}_{\varepsilon }+B(\mathbf{u}%
_{\varepsilon }))dt=\mathbf{f}dt+g^{\varepsilon }(\mathbf{u}_{\varepsilon
})dW,\ 0<t<T \\ 
\mathbf{u}_{\varepsilon }(0)=\mathbf{u}^{0}.%
\end{array}%
\right.
\end{equation*}%
Under assumption \emph{(\ref{6.25})} (see Section \emph{6}), the sequence $(%
\mathbf{u}_{\varepsilon })_{\varepsilon >0}$ converges in probability to $%
\mathbf{u}_{0}$ in $L^{2}(Q_{T})$ where $\mathbf{u}_{0}$ is the unique
strong probabilistic solution of the following problem: 
\begin{equation*}
\left\{ 
\begin{array}{l}
d\mathbf{u}_{0}+(-\Div(\mathsf{m}D\mathbf{u}_{0})-\Div M(D\mathbf{u}_{0})+B(%
\mathbf{u}_{0}))dt=\mathbf{f}dt+\widetilde{g}(\mathbf{u}_{0})dW \\ 
\mathbf{u}_{0}(0)=\mathbf{u}^{0}.%
\end{array}%
\right.
\end{equation*}
\end{theorem}

In view of the above result, one might be tempted to believe that the
homogenization process for SPDEs is summarized in the homogenization of its
deterministic part, added to the average of its stochastic part. This is far
to be true in general. Indeed, one can obtain after passing to the limit, a
homogenized equation of a type completely different from that of the initial
problem; see e.g., \cite{Wang2}.

We can emphasize that our work is therefore the first one in which the
homogenization theory for PDEs and SPDEs is carried out beyond the ergodic
setting.

The paper is organized as follows. In Section 2 we give the key tools which
will be used in the following sections, namely we apply the semigroup theory
to the generalized Besicovitch spaces. Section 3 is devoted to the
systematic study of the concept of $\Sigma $-convergence. We prove there
some important compactness results. In Section 4 we prove Theorem \ref{t1.2}
and we give some of its important corollaries. Finally, Sections 5 and 6 are
devoted to the applications of the results of the earlier sections to
homogenization theory.

Unless otherwise stated, vector spaces throughout are assumed to be complex
vector spaces, and scalar functions are assumed to take complex values.

\section{The semigroup theory applied to the generalized Besicovitch spaces}

\subsection{Preliminaries}

Let $A$ be an algebra with mean value (algebra wmv, in short) on $\mathbb{R}%
^{N}$ \cite{20, Casado, NA, Zhikov4}, that is, $A$ is a closed subalgebra of
the $\mathcal{C}$*-algebra of bounded uniformly continuous functions $BUC(%
\mathbb{R}^{N})$ which contains the constants, is closed under complex
conjugation ($\overline{u}\in A$ whenever $u\in A$), is translation
invariant ($u(\cdot +a)\in A$ for any $u\in A$ and each $a\in \mathbb{R}^{N}$%
) and is such that each element possesses a mean value in the following
sense:

\begin{itemize}
\item[(\textit{MV})] For each $u\in A$, the sequence $(u^{\varepsilon
})_{\varepsilon >0}$ (where $u^{\varepsilon }(x)=u(x/\varepsilon _{1})$, $%
x\in \mathbb{R}^{N}$) weakly $\ast $-converges in $L^{\infty }(\mathbb{R}%
^{N})$ to some constant function $M(u)\in \mathbb{C}$ (the complex field) as 
$\varepsilon \rightarrow 0$, $\varepsilon _{1}=\varepsilon _{1}(\varepsilon
) $ being a positive function of $\varepsilon $ tending to zero with $%
\varepsilon $.
\end{itemize}

It is known that $A$ (endowed with the sup norm topology) is a commutative $%
\mathcal{C}$*-algebra with identity. We denote by $\Delta (A)$ the spectrum
of $A$ and by $\mathcal{G}$ the Gelfand transformation on $A$. We recall
that $\Delta (A)$ (a subset of the topological dual $A^{\prime }$ of $A$) is
the set of all nonzero multiplicative linear functionals on $A$, and $%
\mathcal{G}$ is the mapping of $A$ into $\mathcal{C}(\Delta (A))$ such that $%
\mathcal{G}(u)(s)=\left\langle s,u\right\rangle $ ($s\in \Delta (A)$), where 
$\left\langle ,\right\rangle $ denotes the duality pairing between $%
A^{\prime }$ and $A$. We endow $\Delta (A)$ with the relative weak$\ast $
topology on $A^{\prime }$. Then using the well-known theorem of Stone (see
e.g., either \cite{21} or more precisely \cite[Theorem IV.6.18, p. 274]{DS})
one can easily show that the spectrum $\Delta (A)$ is a compact topological
space, and the Gelfand transformation $\mathcal{G}$ is an isometric $\ast $%
-isomorphism identifying $A$ with $\mathcal{C}(\Delta (A))$ (the continuous
functions on $\Delta (A)$) as $\mathcal{C}$*-algebras. Next, since each
element of $A$ possesses a mean value, this yields an application $u\mapsto
M(u)$ (denoted by $M$ and called the mean value) which is a nonnegative
continuous linear functional on $A$ with $M(1)=1$, and so provides us with a
linear nonnegative functional $\psi \mapsto M_{1}(\psi )=M(\mathcal{G}%
^{-1}(\psi ))$ defined on $\mathcal{C}(\Delta (A))=\mathcal{G}(A)$, which is
clearly bounded. Therefore, by the Riesz-Markov theorem, $M_{1}(\psi )$ is
representable by integration with respect to some Radon measure $\beta $ (of
total mass $1$) in $\Delta (A)$, called the $M$\textit{-measure} for $A$ 
\cite{26}. It is evident that we have 
\begin{equation*}
M(u)=\int_{\Delta (A)}\mathcal{G}(u)d\beta \text{\ for }u\in A\text{.}
\end{equation*}%
The spectrum of a Banach algebra is an abstract concept. However, for some
special Banach algebras, it can be characterized as in the following

\begin{proposition}
\label{p2.0}Let $A$ be an algebra wmv. Assume $A$ separates the points of $%
\mathbb{R}^{N}$. Then $\Delta (A)$ is the Stone-\v{C}ech compactification of 
$\mathbb{R}^{N}$.
\end{proposition}

\begin{proof}
For each $y\in \mathbb{R}^{N}$ let us define an element $\phi _{y}$ of $%
\Delta (A)$ by setting $\phi _{y}(u)=u(y)$ for $u\in A$. Then the mapping $%
\phi :y\mapsto \phi _{y}$, from $\mathbb{R}^{N}$ into $\Delta (A)$, is
continuous and has dense range. In fact as the topology in $\Delta (A)$ is
the weak$\ast $ one and further the mappings $y\mapsto \phi _{y}(u)=u(y)$, $%
u\in A$, are continuous on $\mathbb{R}^{N}$, it follows that $\phi $ is
continuous. Now assuming that $\phi (\mathbb{R}^{N})$ is not dense in $%
\Delta (A)$ we derive the existence of a non empty open subset $U$ of $%
\Delta (A)$ such that $U\cap \phi (\mathbb{R}^{N})=\emptyset $. Then by
Urysohn's lemma there exists $v\in \mathcal{C}(\Delta (A))$ with $v\neq 0$
and $\left. v\right\vert _{\Delta (A))\backslash U}=0$ where $\left.
v\right\vert _{\Delta (A))\backslash U}$ denotes the restriction of $v$ to $%
\Delta (A))\backslash U$. By the Gelfand representation theorem, $v=\mathcal{%
G}(u)$ for some $u\in A$. But then 
\begin{equation*}
u(y)=\phi _{y}(u)=\mathcal{G}(u)(\phi _{y})=v(\phi _{y})=0
\end{equation*}%
for all $y\in \mathbb{R}^{N}$, contradicting $u\neq 0$. Thus $\phi (\mathbb{R%
}^{N})$ is dense in $\Delta (A)$.

Next, every $f$ in $A$ (viewed as element of $\mathcal{B}(\mathbb{R}^{N})$)
extends continuously to $\Delta (A)$ in the sense that there exists $%
\widehat{f}\in \mathcal{C}(\Delta (A))$ such that $\widehat{f}(\phi
(y))=f(y) $ for all $y\in \mathbb{R}^{N}$ (just take $\widehat{f}=\mathcal{G}%
(f)$). Finally assume that $A$ separates the points of $\mathbb{R}^{N}$.
Then the mapping $\phi :\mathbb{R}^{N}\rightarrow \phi (\mathbb{R}^{N})$ is
a homeomorphism. In fact, we only need to prove that $\phi $ is injective.
For that, let $y,z\in \mathbb{R}^{N}$ with $y\neq z$; since $A$ separates
the points of $\mathbb{R}^{N}$, there exists a function $u\in A$ such that $%
u(y)\neq u(z)$, hence $\phi _{y}\neq \phi _{z}$, and our claim is justified.
We therefore conclude that the couple $(\Delta (A),\phi )$ is the Stone-\v{C}%
ech compactification of $\mathbb{R}^{N}$.
\end{proof}

The following result is classically known (see e.g. \cite[Theorem 4.4]{Frid}%
).

\begin{proposition}
\label{p2.4}\emph{(1)} Assume $A=\mathcal{C}_{\text{\emph{per}}}(Y)$\ is the
algebra of $Y$-periodic continuous functions on $\mathbb{R}_{y}^{N}$\ ($Y=(-%
\frac{1}{2},\frac{1}{2})^{N}$). Then its spectrum is the $N$-dimensional
torus $\mathbb{T}^{N}=\mathbb{R}^{N}/\mathbb{Z}^{N}$.\emph{\ (2)} Assume $%
A=AP(\mathbb{R}_{y}^{N})$\ is the algebra of all almost periodic continuous
functions on $\mathbb{R}_{y}^{N}$ defined as the vector space consisting of
all functions defined on $\mathbb{R}_{y}^{N}$ that are uniformly
approximated by finite linear combinations of the functions in the set $%
\{\exp (2i\pi k\cdot y):k\in \mathbb{R}^{N}\}$. Then its spectrum $\Delta
(AP(\mathbb{R}_{y}^{N}))$\ is a compact topological group homeomorphic to
the Bohr compactification of $\mathbb{R}^{N}$.
\end{proposition}

Next, to any algebra with mean value $A$ are associated the following
subspaces: $A^{m}=\{\psi \in \mathcal{C}^{m}(\mathbb{R}^{N}):$ $%
D_{y}^{\alpha }\psi \in A$ for every $\alpha =(\alpha _{1},...,\alpha
_{N})\in \mathbb{N}^{N}$ with $\left\vert \alpha \right\vert \leq m\}$
(where $D_{y}^{\alpha }\psi =\partial ^{\left\vert \alpha \right\vert }\psi
/\partial y_{1}^{\alpha _{1}}\cdot \cdot \cdot \partial y_{N}^{\alpha _{N}}$%
). Endowed with the norm $\left\Vert \left\vert u\right\vert \right\Vert
_{m}=\sup_{\left\vert \alpha \right\vert \leq m}\left\Vert D_{y}^{\alpha
}\psi \right\Vert _{\infty }$, $A^{m}$ is a Banach space. We also define the
space $A^{\infty }$ as the space of $\psi \in \mathcal{C}^{\infty }(\mathbb{R%
}_{y}^{N})$ such that $D_{y}^{\alpha }\psi \in A$ for every $\alpha =(\alpha
_{1},...,\alpha _{N})\in \mathbb{N}^{N}$. Endowed with a suitable locally
convex topology defined by the family of norms $\left\Vert \left\vert \cdot
\right\vert \right\Vert _{m}$, $A^{\infty }$ is a Fr\'{e}chet space.

Let $B_{A}^{p}$ ($1\leq p<\infty $) denote the Besicovitch space associated
to $A$, that is the closure of $A$ with respect to the Besicovitch seminorm 
\begin{equation*}
\left\Vert u\right\Vert _{p}=\left( \underset{r\rightarrow +\infty }{\lim
\sup }\frac{1}{\left\vert B_{r}\right\vert }\int_{B_{r}}\left\vert
u(y)\right\vert ^{p}dy\right) ^{1/p}\text{.}
\end{equation*}%
It is known that $B_{A}^{p}$ is a complete seminormed vector space verifying 
$B_{A}^{q}\subset B_{A}^{p}$ for $1\leq p\leq q<\infty $. From this last
property one may naturally define the space $B_{A}^{\infty }$ as follows: 
\begin{equation*}
B_{A}^{\infty }=\{f\in \cap _{1\leq p<\infty }B_{A}^{p}:\sup_{1\leq p<\infty
}\left\Vert f\right\Vert _{p}<\infty \}\text{.}\;\;\;\;\;\;\;\;\;
\end{equation*}%
We endow $B_{A}^{\infty }$ with the seminorm $\left[ f\right] _{\infty
}=\sup_{1\leq p<\infty }\left\Vert f\right\Vert _{p}$, which makes it a
complete seminormed space. We recall that the spaces $B_{A}^{p}$ ($1\leq
p\leq \infty $) are not in general Fr\'{e}chet spaces since they are not
separated in general. The following properties are worth noticing (see e.g. 
\cite[Section 2]{CMP} or \cite[Section 2]{NA}):

\begin{itemize}
\item[(\textbf{1)}] The Gelfand transformation $\mathcal{G}:A\rightarrow 
\mathcal{C}(\Delta (A))$ extends by continuity to a unique continuous linear
mapping, still denoted by $\mathcal{G}$, of $B_{A}^{p}$ into $L^{p}(\Delta
(A))$, which in turn induces an isometric isomorphism $\mathcal{G}_{1}$, of $%
B_{A}^{p}/\mathcal{N}=\mathcal{B}_{A}^{p}$ onto $L^{p}(\Delta (A))$ (where $%
\mathcal{N}=\{u\in B_{A}^{p}:\mathcal{G}(u)=0\}$). Furthermore if $u\in
B_{A}^{p}\cap L^{\infty }(\mathbb{R}^{N})$ then $\mathcal{G}(u)\in L^{\infty
}(\Delta (A))$ and $\left\Vert \mathcal{G}(u)\right\Vert _{L^{\infty
}(\Delta (A))}\leq \left\Vert u\right\Vert _{L^{\infty }(\mathbb{R}^{N})}$.

\item[(\textbf{2)}] The mean value $M$ viewed as defined on $A$, extends by
continuity to a positive continuous linear form (still denoted by $M$) on $%
B_{A}^{p}$ satisfying $M(u)=\int_{\Delta (A)}\mathcal{G}(u)d\beta $ ($u\in
B_{A}^{p}$). Furthermore, $M(\tau _{a}u)=M(u)$ for each $u\in B_{A}^{p}$ and
all $a\in \mathbb{R}^{N}$, where $\tau _{a}u(z)=u(z+a)$ for almost all $z\in 
\mathbb{R}^{N}$. Moreover for $u\in B_{A}^{p}$ we have $\left\Vert
u\right\Vert _{p}=\left[ M(\left\vert u\right\vert ^{p})\right] ^{1/p}$.
\end{itemize}

\subsection{The semigroup theory}

Let $1\leq p\leq \infty $. We consider the $N$-parameter group of isometries 
$\{T(y):y\in \mathbb{R}^{N}\}$ defined by 
\begin{equation*}
T(y):\mathcal{B}_{A}^{p}\rightarrow \mathcal{B}_{A}^{p}\text{,\ }T(y)(u+%
\mathcal{N})=\tau _{y}u+\mathcal{N}\text{ for }u\in B_{A}^{p}\text{.}
\end{equation*}%
Since the elements of $A$ are uniformly continuous, $\{T(y):y\in \mathbb{R}%
^{N}\}$ is a strongly continuous group of operators in $\mathcal{L}(\mathcal{%
B}_{A}^{p},\mathcal{B}_{A}^{p})$ (the Banach space of continuous linear
functionals of $\mathcal{B}_{A}^{p}$ into $\mathcal{B}_{A}^{p}$) in the
sense of semigroups: $T(y)(u+\mathcal{N})\rightarrow u+\mathcal{N}$ in $%
\mathcal{B}_{A}^{p}$ as $\left\vert y\right\vert \rightarrow 0$. To $%
\{T(y):y\in \mathbb{R}^{N}\}$ is associated the following $N$-parameter
group $\{\overline{T}(y):y\in \mathbb{R}^{N}\}$ defined as 
\begin{equation*}
\begin{array}{l}
\overline{T}(y):L^{p}(\Delta (A))\rightarrow L^{p}(\Delta (A)) \\ 
\overline{T}(y)\mathcal{G}_{1}(u+\mathcal{N})=\mathcal{G}_{1}(T(y)(u+%
\mathcal{N}))=\mathcal{G}_{1}(\tau _{y}u+\mathcal{N})\text{\ for }u\in
B_{A}^{p}\text{.}%
\end{array}%
\end{equation*}%
The group $\{\overline{T}(y):y\in \mathbb{R}^{N}\}$ is also strongly
continuous. The infinitesimal generator of $T(y)$ (resp. $\overline{T}(y)$)
along the $i$th coordinate direction, denoted by $D_{i,p}$ (resp. $\partial
_{i,p}$), is defined as 
\begin{equation*}
D_{i,p}u=\lim_{t\rightarrow 0}t^{-1}\left( T(te_{i})u-u\right) \text{\ in }%
\mathcal{B}_{A}^{p}\text{ (resp. }\partial _{i,p}v=\lim_{t\rightarrow
0}t^{-1}\left( \overline{T}(te_{i})v-v\right) \text{\ in }L^{p}(\Delta (A))%
\text{)}
\end{equation*}%
where here we have used the same letter $u$ to denote the equivalence class
of an element $u\in B_{A}^{p}$ in $\mathcal{B}_{A}^{p}$, $e_{i}=(\delta
_{ij})_{1\leq j\leq N}$ ($\delta _{ij}$ being the Kronecker $\delta $). The
domain of $D_{i,p}$ (resp. $\partial _{i,p}$) in $\mathcal{B}_{A}^{p}$
(resp. $L^{p}(\Delta (A))$) is denoted by $\mathcal{D}_{i,p}$ (resp. $%
\mathcal{W}_{i,p}$). By using the general theory of semigroups \cite[Chap.
VIII, Section 1]{DS}, the following result holds.

\begin{proposition}
\label{p2.1}$\mathcal{D}_{i,p}$ (resp. $\mathcal{W}_{i,p}$) is a vector
subspace of $\mathcal{B}_{A}^{p}$ (resp. $L^{p}(\Delta (A))$), $D_{i,p}:%
\mathcal{D}_{i,p}\rightarrow \mathcal{B}_{A}^{p}$ (resp. $\partial _{i,p}:%
\mathcal{W}_{i,p}\rightarrow L^{p}(\Delta (A))$) is a linear operator, $%
\mathcal{D}_{i,p}$ (resp. $\mathcal{W}_{i,p}$) is dense in $\mathcal{B}%
_{A}^{p}$ (resp. $L^{p}(\Delta (A))$), and the graph of $D_{i,p}$ (resp. $%
\partial _{i,p}$) is closed in $\mathcal{B}_{A}^{p}\times \mathcal{B}%
_{A}^{p} $ (resp. $L^{p}(\Delta (A))\times L^{p}(\Delta (A))$).
\end{proposition}

In the sequel we denote by $\varrho $ the canonical mapping of $B_{A}^{p}$
onto $\mathcal{B}_{A}^{p}$, that is, $\varrho (u)=u+\mathcal{N}$ for $u\in
B_{A}^{p}$. The following result allows us to see $D_{i,p}$ as a
generalization of the usual partial derivative.

\begin{lemma}
\label{l2.1}Let $1\leq i\leq N$. If $u\in A^{1}$ then $\varrho (u)\in 
\mathcal{D}_{i,p}$ and 
\begin{equation}
D_{i,p}\varrho (u)=\varrho \left( \frac{\partial u}{\partial y_{i}}\right)
.\ \ \ \ \ \ \ \ \ \ \ \ \ \ \ \ \ \ \ \ \ \ \ \ \ \ \ \ \ \ \ \ \ 
\label{2.2}
\end{equation}
\end{lemma}

\begin{proof}
By the mean value's inequality we have 
\begin{equation*}
\left\vert u(y+r)-u(y)-D_{y}u(y)\cdot r\right\vert \leq \left\vert
r\right\vert \sup_{0\leq t\leq 1}\left\vert
D_{y}u(y+tr)-D_{y}u(y)\right\vert .
\end{equation*}%
Taking $r=te_{i}$ ($t>0$) in the above inequality, we get 
\begin{equation*}
\left\vert t^{-1}\left( u(y+te_{i})-u(y)\right) -\frac{\partial u}{\partial
y_{i}}(y)\right\vert \leq \sup_{0\leq \zeta \leq 1}\left\vert \frac{\partial
u}{\partial y_{i}}(y+\zeta te_{i})-\frac{\partial u}{\partial y_{i}}%
(y)\right\vert ,
\end{equation*}%
and hence 
\begin{equation*}
\left\Vert t^{-1}\left( \tau _{te_{i}}u-u\right) -\frac{\partial u}{\partial
y_{i}}\right\Vert _{\infty }\leq \sup_{0\leq \zeta \leq 1}\sup_{y\in \mathbb{%
R}^{N}}\left\vert \frac{\partial u}{\partial y_{i}}(y+\zeta te_{i})-\frac{%
\partial u}{\partial y_{i}}(y)\right\vert .
\end{equation*}%
Since $\partial u/\partial y_{i}\in A$, it is uniformly continuous, so that 
\begin{equation*}
\lim_{t\rightarrow 0}\left\Vert t^{-1}\left( \tau _{te_{i}}u-u\right) -\frac{%
\partial u}{\partial y_{i}}\right\Vert _{\infty }=0\text{,}
\end{equation*}%
and, since the uniform norm is greater than the $B_{A}^{p}$-norm, 
\begin{equation*}
\lim_{t\rightarrow 0}\left\Vert t^{-1}\left( \tau _{te_{i}}\varrho
(u)-\varrho (u)\right) -\varrho \left( \frac{\partial u}{\partial y_{i}}%
\right) \right\Vert _{p}=0,
\end{equation*}%
that is, $D_{i,p}\varrho (u)=\varrho (\partial u/\partial y_{i})$.
\end{proof}

\begin{remark}
\label{r2.1}\emph{From (\ref{2.2}) we deduce that }$D_{i,p}\circ \varrho
=\varrho \circ \partial /\partial y_{i}$\emph{, which means that }$D_{i,p}$%
\emph{\ is a generalization of the usual partial derivative.}
\end{remark}

One can naturally define higher order derivatives by setting $D_{p}^{\alpha
}=D_{1,p}^{\alpha _{1}}\circ \cdot \cdot \cdot \circ D_{N,p}^{\alpha _{N}}$
(resp. $\partial _{p}^{\alpha }=\partial _{1,p}^{\alpha _{1}}\circ \cdot
\cdot \cdot \circ \partial _{N,p}^{\alpha _{N}}$) for $\alpha =(\alpha
_{1},...,\alpha _{N})\in \mathbb{N}^{N}$ with $D_{i,p}^{\alpha
_{i}}=D_{i,p}\circ \cdot \cdot \cdot \circ D_{i,p}$, $\alpha _{i}$-times.
Now, let 
\begin{equation*}
\mathcal{B}_{A}^{1,p}=\cap _{i=1}^{N}\mathcal{D}_{i,p}=\{u\in \mathcal{B}%
_{A}^{p}:D_{i,p}u\in \mathcal{B}_{A}^{p}\ \forall 1\leq i\leq N\}
\end{equation*}%
and 
\begin{equation*}
\mathcal{D}_{A}(\mathbb{R}^{N})=\{u\in \mathcal{B}_{A}^{\infty }:D_{\infty
}^{\alpha }u\in \mathcal{B}_{A}^{\infty }\ \forall \alpha \in \mathbb{N}%
^{N}\}.
\end{equation*}%
It can be shown that $\mathcal{D}_{A}(\mathbb{R}^{N})$ is dense in $\mathcal{%
B}_{A}^{p}$, $1\leq p<\infty $. We also have that $\mathcal{B}_{A}^{1,p}$ is
a Banach space under the norm 
\begin{equation*}
\left\Vert u\right\Vert _{\mathcal{B}_{A}^{1,p}}=\left( \left\Vert
u\right\Vert _{p}^{p}+\sum_{i=1}^{N}\left\Vert D_{i,p}u\right\Vert
_{p}^{p}\right) ^{1/p}\ \ (u\in \mathcal{B}_{A}^{1,p});
\end{equation*}%
this comes from the fact that the graph of $D_{i,p}$ is closed.

The counter-part of the above properties also holds with 
\begin{equation*}
W^{1,p}(\Delta (A))=\cap _{i=1}^{N}\mathcal{W}_{i,p}\text{\ in place of }%
\mathcal{B}_{A}^{1,p}
\end{equation*}%
and 
\begin{equation*}
\mathcal{D}(\Delta (A))=\{u\in L^{\infty }(\Delta (A)):\partial _{\infty
}^{\alpha }u\in L^{\infty }(\Delta (A))\ \forall \alpha \in \mathbb{N}^{N}\}%
\text{\ in that of }\mathcal{D}_{A}(\mathbb{R}^{N})\text{.}
\end{equation*}%
We have the following relation between $D_{i,p}$ and $\partial _{i,p}$.

\begin{lemma}
\label{l2.2}Let $u\in \mathcal{D}_{i,p}$. Then $\mathcal{G}_{1}(u)\in 
\mathcal{W}_{i,p}$ and $\mathcal{G}_{1}(D_{i,p}u)=\partial _{i,p}\mathcal{G}%
_{1}(u)$.
\end{lemma}

\begin{proof}
We have 
\begin{eqnarray*}
\left\Vert t^{-1}(T(te_{i})u-u)-D_{i,p}u\right\Vert _{p} &=&\left\Vert 
\mathcal{G}_{1}(t^{-1}(T(te_{i})u-u))-\mathcal{G}_{1}(D_{i,p}u)\right\Vert
_{L^{p}(\Delta (A))} \\
&=&\left\Vert t^{-1}(\mathcal{G}_{1}(T(te_{i})u)-\mathcal{G}_{1}(u))-%
\mathcal{G}_{1}(D_{i,p}u)\right\Vert _{L^{p}(\Delta (A))} \\
&=&\left\Vert t^{-1}(\overline{T}(te_{i})\mathcal{G}_{1}(u)-\mathcal{G}%
_{1}(u))-\mathcal{G}_{1}(D_{i,p}u)\right\Vert _{L^{p}(\Delta (A))}.
\end{eqnarray*}%
Since $u\in \mathcal{D}_{i,p}$ we have $\left\Vert
t^{-1}(T(te_{i})u-u)-D_{i,p}u\right\Vert _{p}\rightarrow 0$ as $t\rightarrow
0$. Therefore 
\begin{equation*}
\left\Vert t^{-1}(\overline{T}(te_{i})\mathcal{G}_{1}(u)-\mathcal{G}_{1}(u))-%
\mathcal{G}_{1}(D_{i,p}u)\right\Vert _{L^{p}(\Delta (A))}\rightarrow 0\text{
as }t\rightarrow 0\text{,}
\end{equation*}%
so that $\mathcal{G}_{1}(u)\in \mathcal{W}_{i,p}$ with $\partial _{i,p}%
\mathcal{G}_{1}(u)=\mathcal{G}_{1}(D_{i,p}u)$.
\end{proof}

Now, let $u\in \mathcal{D}_{i,p}$ ($p\geq 1$, $1\leq i\leq N$). Then the
inequality 
\begin{equation*}
\left\Vert t^{-1}(T(te_{i})u-u)-D_{i,p}u\right\Vert _{1}\leq c\left\Vert
t^{-1}(T(te_{i})u-u)-D_{i,p}u\right\Vert _{p}
\end{equation*}%
for a positive constant $c$ independent of $u$ and $t$, yields $%
D_{i,1}u=D_{i,p}u$, so that $D_{i,p}$ is the restriction to $\mathcal{B}%
_{A}^{p}$ of $D_{i,1}$. Therefore, for all $u\in \mathcal{D}_{i,\infty }$ we
have $u\in \mathcal{D}_{i,p}$ ($p\geq 1$) and $D_{i,\infty }u=D_{i,p}u$\ $%
\forall 1\leq i\leq N$. The following simple result will be useful.

\begin{lemma}
\label{l2.3}We have $\mathcal{D}_{A}(\mathbb{R}^{N})=\varrho (A^{\infty })$.
\end{lemma}

\begin{proof}
From (\ref{2.2}) we have that, for $u\in \varrho (A^{\infty })$ and $\alpha
\in \mathbb{N}^{N}$, $D_{\infty }^{\alpha }u=\varrho (D_{y}^{\alpha }v)$
where $v\in A^{\infty }$ is such that $u=\varrho (v)$. This leads at once to 
$\varrho (A^{\infty })\subset \mathcal{D}_{A}(\mathbb{R}^{N})$. Conversely
if $u\in \mathcal{D}_{A}(\mathbb{R}^{N})$, then $u\in \mathcal{B}%
_{A}^{\infty }$ with $D_{\infty }^{\alpha }u\in \mathcal{B}_{A}^{\infty }$
for all $\alpha \in \mathbb{N}^{N}$, that is, $u=v+\mathcal{N}$ with $v\in
B_{A}^{\infty }$ being such that $D_{y}^{\alpha }v\in B_{A}^{\infty }$ for
all $\alpha \in \mathbb{N}^{N}$, i.e., $v\in A^{\infty }$ since, as $v$ is
in $L_{\text{loc}}^{p}(\mathbb{R}^{N})$ with all its distributional
derivatives, $v$ is of class $\mathcal{C}^{\infty }$. Hence $u=v+\mathcal{N}$
with $v\in A^{\infty }$, so that $u\in \varrho (A^{\infty })$.
\end{proof}

The following result holds.

\begin{proposition}
\label{p2.2}The following assertions hold.

\begin{itemize}
\item[(i)] $\int_{\Delta (A)}\partial _{\infty }^{\alpha }\widehat{u}d\beta
=0$ for all $u\in \mathcal{D}_{A}(\mathbb{R}^{N})$ and $\alpha \in \mathbb{N}%
^{N}$;

\item[(ii)] $\int_{\Delta (A)}\partial _{i,p}\widehat{u}d\beta =0$ for all $%
u\in \mathcal{D}_{i,p}$ and $1\leq i\leq N$;

\item[(iii)] $D_{i,p}(u\phi )=uD_{i,\infty }\phi +\phi D_{i,p}u$ for all $%
(\phi ,u)\in \mathcal{D}_{A}(\mathbb{R}^{N})\times \mathcal{D}_{i,p}$ and $%
1\leq i\leq N$.
\end{itemize}
\end{proposition}

\begin{proof}
(i) It suffices to check it for $\alpha =\alpha _{i}=(\delta _{ij})_{1\leq
j\leq N}$, that is, for $1\leq i\leq N$, 
\begin{equation*}
\int_{\Delta (A)}\partial _{i,\infty }\widehat{u}d\beta =0\text{\ for all }%
u\in \mathcal{D}_{A}(\mathbb{R}^{N}).
\end{equation*}%
But, as $u=\varrho (v)$ with $v\in A^{\infty }$ (see Lemma \ref{l2.2}), we
have 
\begin{eqnarray*}
\int_{\Delta (A)}\partial _{i,\infty }\widehat{u}d\beta &=&\int_{\Delta
(A)}\partial _{i,\infty }\mathcal{G}_{1}(\varrho (v))d\beta =\int_{\Delta
(A)}\mathcal{G}_{1}(D_{i,\infty }\varrho (v))d\beta \\
&=&\int_{\Delta (A)}\mathcal{G}_{1}\left( \varrho \left( \frac{\partial v}{%
\partial y_{i}}\right) \right) d\beta =\int_{\Delta (A)}\mathcal{G}\left( 
\frac{\partial v}{\partial y_{i}}\right) d\beta =0
\end{eqnarray*}%
where the last equality is justified as in \cite{26}. Whence (i).

(ii) Let $u\in \mathcal{D}_{i,p}$; there exists $u_{n}\in \mathcal{D}_{A}(%
\mathbb{R}^{N})$ such that $\left\Vert u-u_{n}\right\Vert _{\mathcal{D}%
_{i,p}}\rightarrow 0$ as $n\rightarrow \infty $. We have 
\begin{eqnarray*}
\left\vert \int_{\Delta (A)}\partial _{i,p}\widehat{u}d\beta \right\vert
&=&\left\Vert \partial _{i,p}\widehat{u}-\partial _{i,p}\widehat{u}%
_{n}\right\Vert _{L^{1}(\Delta (A))} \\
&\leq &c\left\Vert \partial _{i,p}\widehat{u}-\partial _{i,p}\widehat{u}%
_{n}\right\Vert _{L^{p}(\Delta (A))} \\
&=&c\left\Vert D_{i,p}u-D_{i,p}u_{n}\right\Vert _{p},
\end{eqnarray*}%
and the last term on the right-hand side of the above equality tends to zero
as $n\rightarrow \infty $, since $D_{i,p}$ is continuous (in fact it is
linear with a closed graph; see Proposition \ref{p2.1}). This shows (ii).
(iii) is easily verified.
\end{proof}

The formula (iii) in the above proposition leads to the equality 
\begin{equation*}
\int_{\Delta (A)}\widehat{\phi }\partial _{i,p}\widehat{u}d\beta
=-\int_{\Delta (A)}\widehat{u}\partial _{i,\infty }\widehat{\phi }d\beta \ \
\forall (u,\phi )\in \mathcal{D}_{i,p}\times \mathcal{D}_{A}(\mathbb{R}^{N}).
\end{equation*}%
This suggests us to define the concepts of distributions on $A$ and of a
weak derivative. Before we can do that, let us endow $\mathcal{D}_{A}(%
\mathbb{R}^{N})=\varrho (A^{\infty })$ with its natural topology defined by
the family of norms $N_{n}(u)=\sup_{\left\vert \alpha \right\vert \leq
n}\sup_{y\in \mathbb{R}^{N}}\left\vert D_{\infty }^{\alpha }u(y)\right\vert $%
, $n\in \mathbb{N}$. In this topology, $\mathcal{D}_{A}(\mathbb{R}^{N})$ is
a Fr\'{e}chet space. We denote by $\mathcal{D}_{A}^{\prime }(\mathbb{R}^{N})$
the topological dual of $\mathcal{D}_{A}(\mathbb{R}^{N})$. We endow it with
the strong dual topology. The elements of $\mathcal{D}_{A}^{\prime }(\mathbb{%
R}^{N})$ are called \textit{the distributions on }$A$. One can also define
the weak derivative of $f\in \mathcal{D}_{A}^{\prime }(\mathbb{R}^{N})$ as
follows: for any $\alpha \in \mathbb{N}^{N}$, $D^{\alpha }f$ stands for the
distribution defined by the formula 
\begin{equation*}
\left\langle D^{\alpha }f,\phi \right\rangle =(-1)^{\left\vert \alpha
\right\vert }\left\langle f,D_{\infty }^{\alpha }\phi \right\rangle \text{\
for all }\phi \in \mathcal{D}_{A}(\mathbb{R}^{N}).
\end{equation*}%
Since $\mathcal{D}_{A}(\mathbb{R}^{N})$ is dense in $\mathcal{B}_{A}^{p}$ ($%
1\leq p<\infty $), it is immediate that $\mathcal{B}_{A}^{p}\subset \mathcal{%
D}_{A}^{\prime }(\mathbb{R}^{N})$ with continuous embedding, so that one may
define the weak derivative of any $f\in \mathcal{B}_{A}^{p}$, and it
verifies the following functional equation: 
\begin{equation*}
\left\langle D^{\alpha }f,\phi \right\rangle =(-1)^{\left\vert \alpha
\right\vert }\int_{\Delta (A)}\widehat{f}\partial _{\infty }^{\alpha }%
\widehat{\phi }d\beta \text{\ for all }\phi \in \mathcal{D}_{A}(\mathbb{R}%
^{N}).
\end{equation*}%
In particular, for $f\in \mathcal{D}_{i,p}$ we have 
\begin{equation*}
-\int_{\Delta (A)}\widehat{f}\partial _{i,p}\widehat{\phi }d\beta
=\int_{\Delta (A)}\widehat{\phi }\partial _{i,p}\widehat{f}d\beta \ \
\forall \phi \in \mathcal{D}_{A}(\mathbb{R}^{N}),
\end{equation*}%
so that we may identify $D_{i,p}f$ with $D^{\alpha _{i}}f$, $\alpha
_{i}=(\delta _{ij})_{1\leq j\leq N}$. Conversely, if $f\in \mathcal{B}%
_{A}^{p}$ is such that there exists $f_{i}\in \mathcal{B}_{A}^{p}$ with $%
\left\langle D^{\alpha _{i}}f,\phi \right\rangle =-\int_{\Delta (A)}\widehat{%
f}_{i}\widehat{\phi }d\beta $ for all $\phi \in \mathcal{D}_{A}(\mathbb{R}%
^{N})$, then $f\in \mathcal{D}_{i,p}$ and $D_{i,p}f=f_{i}$. We are therefore
justified in saying that $\mathcal{B}_{A}^{1,p}$ is a Banach space under the
norm $\left\Vert \cdot \right\Vert _{\mathcal{B}_{A}^{1,p}}$. The same
result holds for $W^{1,p}(\Delta (A))$. Moreover it is a fact that $\mathcal{%
D}_{A}(\mathbb{R}^{N})$ (resp. $\mathcal{D}(\Delta (A))$) is a dense
subspace of $\mathcal{B}_{A}^{1,p}$ (resp. $W^{1,p}(\Delta (A))$).

Now, in order to deal with the homogenization theory, we need to define the
space of correctors. Before we can do this, however, we need some further
notions.

A function $f\in \mathcal{B}_{A}^{1}$ is said to be \textit{invariant} if
for any $y\in \mathbb{R}^{N}$, $T(y)f=f$. It is immediate that the above
notion of invariance is the well-known one relative to dynamical systems. An
algebra with mean value will therefore said to be \textit{ergodic} if every
invariant function $f$ is constant in $\mathcal{B}_{A}^{1}$. As in \cite[%
Lemma 2.3 (a)]{10} one can show that $f\in \mathcal{B}_{A}^{1}$ is invariant
if and only if $D_{i,1}f=0$ for all $1\leq i\leq N$. We denote by $I_{A}^{p}$
the set of $f\in \mathcal{B}_{A}^{1}$ that are invariant. The set $I_{A}^{p}$
is a closed vector subspace of $\mathcal{B}_{A}^{p}$ satisfying the
following property:%
\begin{equation}
f\in I_{A}^{p}\text{ if and only if }D_{i,p}f=0\text{ for all }1\leq i\leq N%
\text{.}  \label{2.5}
\end{equation}

\noindent The above property is due to the fact that $D_{i,p}$ is the
restriction to $\mathcal{B}_{A}^{p}$ of $D_{i,1}$. So the mapping 
\begin{equation*}
u\mapsto \left\Vert u\right\Vert _{\#,p}:=\left( \sum_{i=1}^{N}\left\Vert
D_{i,p}u\right\Vert _{p}^{p}\right) ^{1/p}
\end{equation*}%
considered as defined on $\mathcal{D}_{A}(\mathbb{R}^{N})$, is a norm on the
subspace $\mathcal{D}_{A}(\mathbb{R}^{N})/I_{A}^{p}$ of $\mathcal{D}_{A}(%
\mathbb{R}^{N})$ consisting of functions $u\in \mathcal{D}_{A}(\mathbb{R}%
^{N})$ that agree on $I_{A}^{p}$. Unfortunately, under this norm, $\mathcal{D%
}_{A}(\mathbb{R}^{N})/I_{A}^{p}$ is a normed vector space which is in
general not complete. We denote by $\mathcal{B}_{\#A}^{1,p}$ its completion
with respect to $\left\Vert \cdot \right\Vert _{\#,p}$. Moreover, as $%
\mathcal{D}_{A}(\mathbb{R}^{N})$ is dense in $\mathcal{B}_{A}^{1,p}$ and
further $\left\Vert u\right\Vert _{\#,p}=0$ if and only if $u\in I_{A}^{p}$,
we have that $\mathcal{B}_{\#A}^{1,p}$ is also the completion of $\mathcal{B}%
_{A}^{1,p}/I_{A}^{p}$ with respect to $\left\Vert \cdot \right\Vert _{\#,p}$%
. We denote by $J_{p}$ the canonical embedding of $\mathcal{B}%
_{A}^{1,p}/I_{A}^{p}$ into its completion $\mathcal{B}_{\#A}^{1,p}$ (which
allows us to viewed $\mathcal{B}_{A}^{1,p}/I_{A}^{p}$ as a subspace of $%
\mathcal{B}_{\#A}^{1,p}$). The following properties are due to the theory of
completion of uniform spaces (see \cite[Chap. II]{9}):

\begin{itemize}
\item[(P$_{1}$)] The gradient operator $D_{p}=(D_{1,p},...,D_{N,p}):\mathcal{%
D}_{A}(\mathbb{R}^{N})/I_{A}^{p}\rightarrow (\mathcal{B}_{A}^{p})^{N}$
extends by continuity to a unique mapping $\overline{D}_{p}:\mathcal{B}%
_{\#A}^{1,p}\rightarrow (\mathcal{B}_{A}^{p})^{N}$ with the properties 
\begin{equation*}
D_{i,p}=\overline{D}_{i,p}\circ J_{p}
\end{equation*}%
and 
\begin{equation*}
\left\Vert u\right\Vert _{\#,p}=\left( \sum_{i=1}^{N}\left\Vert \overline{D}%
_{i,p}u\right\Vert _{p}^{p}\right) ^{1/p}\ \ \text{for }u\in \mathcal{B}%
_{\#A}^{1,p}.
\end{equation*}

\item[(P$_{2}$)] The space $\mathcal{D}_{A}(\mathbb{R}^{N})/I_{A}^{p}$ (and
hence $\mathcal{B}_{A}^{1,p}/I_{A}^{p}$) is dense in $\mathcal{B}%
_{\#A}^{1,p} $: in fact by the embedding $J_{p}$, $\mathcal{D}_{A}(\mathbb{R}%
^{N})/I_{A}^{p}$ is viewed as a subspace of $\mathcal{B}_{\#A}^{1,p}$ (as
said above), and by the theory of completion, $J_{p}(\mathcal{D}_{A}(\mathbb{%
R}^{N})/I_{A}^{p})\equiv \mathcal{D}_{A}(\mathbb{R}^{N})/I_{A}^{p}$ is dense
in $\mathcal{B}_{\#A}^{1,p}$.
\end{itemize}

\noindent Moreover the mapping $\overline{D}_{p}$ is an isometric embedding
of $\mathcal{B}_{\#A}^{1,p}$ onto a closed subspace of $(\mathcal{B}%
_{A}^{p})^{N}$, so that $\mathcal{B}_{\#A}^{1,p}$ is a reflexive Banach
space. By duality we define the divergence operator $\Div_{p^{\prime }}:(%
\mathcal{B}_{A}^{p^{\prime }})^{N}\rightarrow (\mathcal{B}%
_{\#A}^{1,p})^{\prime }$ ($p^{\prime }=p/(p-1)$) by 
\begin{equation}
\left\langle \Div_{p^{\prime }}u,v\right\rangle =-\left\langle u,\overline{D}%
_{p}v\right\rangle \text{\ for }v\in \mathcal{B}_{\#A}^{1,p}\text{ and }%
u=(u_{i})\in (\mathcal{B}_{A}^{p^{\prime }})^{N}\text{,}  \label{2.6}
\end{equation}%
where $\left\langle u,\overline{D}_{p}v\right\rangle
=\sum_{i=1}^{N}\int_{\Delta (A)}\widehat{u}_{i}\partial _{i,p}\widehat{v}%
d\beta $. The operator $\Div_{p^{\prime }}$ just defined extends the natural
divergence operator defined in $\mathcal{D}_{A}(\mathbb{R}^{N})$ since $%
D_{i,p}f=\overline{D}_{i,p}(J_{p}f)$ for all $f\in \mathcal{D}_{A}(\mathbb{R}%
^{N})$ where here, we write $f$ under the form $f=(f-\overline{f})+\overline{%
f}$ with $\overline{f}\in I_{A}^{p}$ and we know that in that case, $%
D_{i,p}f=D_{i,p}(f-\overline{f})$, so that $\overline{D}_{i,p}(J_{p}f)=%
\overline{D}_{i,p}(J_{p}(f-\overline{f}))$.

Now if in (\ref{2.6}) we take $u=D_{p^{\prime }}w$ with $w\in \mathcal{B}%
_{A}^{p^{\prime }}$ being such that $D_{p^{\prime }}w\in (\mathcal{B}%
_{A}^{p^{\prime }})^{N}$ then this allows us to define the Laplacian
operator on $\mathcal{B}_{A}^{p^{\prime }}$, denoted here by $\Delta
_{p^{\prime }}$, as follows: 
\begin{equation*}
\left\langle \Delta _{p^{\prime }}w,v\right\rangle =\left\langle \Div%
_{p^{\prime }}(D_{p^{\prime }}w),v\right\rangle =-\left\langle D_{p^{\prime
}}w,\overline{D}_{p}v\right\rangle \text{\ for all }v\in \mathcal{B}%
_{\#A}^{1,p}.
\end{equation*}%
If in addition $v=J_{p}(\phi )$ with $\phi \in \mathcal{D}_{A}(\mathbb{R}%
^{N})/I_{A}^{p}$ then $\left\langle \Delta _{p^{\prime }}w,J_{p}(\phi
)\right\rangle =-\left\langle D_{p^{\prime }}w,D_{p}\phi \right\rangle $, so
that, for $p=2$, we get 
\begin{equation}
\left\langle \Delta _{2}w,J_{2}(\phi )\right\rangle =\left\langle w,\Delta
_{2}\phi \right\rangle \text{\ for all }w\in \mathcal{B}_{A}^{2}\text{ and }%
\phi \in \mathcal{D}_{A}(\mathbb{R}^{N})/I_{A}^{2}\text{.}  \label{2.7}
\end{equation}

The following result is also immediate.

\begin{proposition}
\label{p2.5}For $u\in A^{\infty }$ we have 
\begin{equation*}
\Delta _{p}\varrho (u)=\varrho (\Delta _{y}u)
\end{equation*}%
where $\Delta _{y}$ stands for the usual Laplacian operator on $\mathbb{R}%
_{y}^{N}$.
\end{proposition}

\begin{remark}
\label{r2.2}\emph{If the algebra }$A$\emph{\ is ergodic, then the space }$%
\mathcal{B}_{\#A}^{1,p}$\emph{\ is just the one defined in \cite{CMP, NA} as
the completion of }$\mathcal{B}_{A}^{1,p}/\mathbb{C}=\{u\in \mathcal{B}%
_{A}^{1,p}:M(u)=0\}$\emph{\ with respect to }$\left\Vert \cdot \right\Vert
_{\#,p}$\emph{. Indeed in that case the elements of }$I_{A}^{p}$\emph{\ are
constant functions.}
\end{remark}

The following result is a special version of a general result whose proof
can be found in \cite[Proposition 1]{AA} (see also \cite[Lemma 2.3 (b)]{10}%
). We will therefore just give the outlines of the proof without any detail.

\begin{proposition}
\label{p2.3}Let $\mathbf{v}\in (\mathcal{B}_{A}^{p})^{N}$ satisfying 
\begin{equation*}
\int_{\Delta (A)}\widehat{\mathbf{v}}\cdot \widehat{\mathbf{g}}d\beta =0%
\text{\ for all }\mathbf{g}\in \mathcal{V}_{\Div}=\{\mathbf{f}\in (\mathcal{D%
}_{A}(\mathbb{R}^{N}))^{N}:\Div_{p^{\prime }}\mathbf{f}=0\}.
\end{equation*}%
Then there exists $u\in \mathcal{B}_{\#A}^{1,p}$ such that $\mathbf{v}=%
\overline{D}_{p}u$.
\end{proposition}

\begin{proof}
The result follows from the equality $(\ker (\Div_{p^{\prime }}))^{\perp }=R(%
\overline{D}_{p})$ (where $R(\overline{D}_{p})$ stands for the range of $%
\overline{D}_{p}$) through the following assertions: (1) $\Div_{p^{\prime }}$
is closed; (2) $(\Div_{p^{\prime }})^{\ast }=-\overline{D}_{p}$ where $(\Div%
_{p^{\prime }})^{\ast }$ is the adjoint operator of $\Div_{p^{\prime }}$;
(3) $R(\overline{D}_{p})$ is closed in $(\mathcal{B}_{A}^{p})^{N}$, and
finally, (4) $\mathbf{v}$ is orthogonal to the kernel of $\Div_{p^{\prime }}$%
.
\end{proof}

We end this subsection with some notations. Let $f\in \mathcal{B}_{A}^{p}$.
We know that $D^{\alpha _{i}}f$ exists (in the sense of distributions) and
that $D^{\alpha _{i}}f=D_{i,p}f$ if $f\in \mathcal{D}_{i,p}$. So we can drop
the subscript $p$ and therefore denote $D_{i,p}$ (resp. $\partial _{i,p}$)
by $\overline{\partial }/\partial y_{i}$ (resp. $\partial _{i}$). Thus, $%
\overline{D}_{y}$ will stand for the gradient operator $(\overline{\partial }%
/\partial y_{i})_{1\leq i\leq N}$ and $\overline{\Div}_{y}$ for the
divergence operator $\Div_{p}$. We will also denote the operator $\overline{D%
}_{i,p}$ by $\overline{\partial }/\partial y_{i}$, and the canonical mapping 
$J_{p}$ will be merely denote by $J$, so that in property (P$_{1}$), we will
have 
\begin{equation}
\frac{\overline{\partial }}{\partial y_{i}}=\frac{\overline{\partial }}{%
\partial y_{i}}\circ J\text{\ \ (in the above notations).}  \label{2.4}
\end{equation}%
This will lead to the notation $\overline{D}_{p}=(\overline{\partial }%
/\partial y_{i})_{1\leq i\leq N}$. Finally, we will denote the Laplacian
operator on $\mathcal{B}_{A}^{p}$ by $\overline{\Delta }_{y}$.

\section{The $\Sigma $-convergence}

This section deals with two concepts of $\Sigma $-convergence: the usual one 
\cite{26} which is revisited, and its generalization to stochastic processes.

\subsection{The $\Sigma $-convergence revisited}

In all that follows $Q$ is an open subset of $\mathbb{\mathbb{R}}^{N}$
(integer $N\geq 1$) and $A$ is an algebra wmv on $\mathbb{\mathbb{R}}%
_{y}^{N} $. The notations are the one of the preceding section.

\begin{definition}
\label{d3.1}\emph{A sequence }$(u_{\varepsilon })_{\varepsilon >0}\subset
L^{p}(Q)$\emph{\ (}$1\leq p<\infty $\emph{) is said to }weakly $\Sigma $%
-converge\emph{\ in }$L^{p}(Q)$\emph{\ to some }$u_{0}\in L^{p}(Q;\mathcal{B}%
_{A}^{p})$\emph{\ if as }$\varepsilon \rightarrow 0$\emph{, we have } 
\begin{equation}
\int_{Q}u_{\varepsilon }(x)f\left( x,\frac{x}{\varepsilon _{1}}\right)
dx\rightarrow \iint_{Q\times \Delta (A)}\widehat{u}_{0}(x,s)\widehat{f}%
(x,s)dxd\beta \;\;\;\;\;\;\;\;  \label{3.1}
\end{equation}%
\emph{for every }$f\in L^{p^{\prime }}(Q;A)$\emph{\ (}$1/p^{\prime }=1-1/p$%
\emph{), where }$\widehat{u}_{0}=\mathcal{G}_{1}\circ u_{0}$\emph{\ and }$%
\widehat{f}=\mathcal{G}_{1}\circ (\varrho \circ f)=\mathcal{G}\circ f$\emph{%
. We express this by writing} $u_{\varepsilon }\rightarrow u_{0}$ in $%
L^{p}(Q)$-weak $\Sigma $.
\end{definition}

It is well-known that the weak $\Sigma $-convergence in $L^{p}$\ implies the
weak convergence in $L^{p}$. Throughout the paper the letter $E$ will denote
any ordinary sequence $E=(\varepsilon _{n})$ (integers $n\geq 0$) with $%
0<\varepsilon _{n}\leq 1$ and $\varepsilon _{n}\rightarrow 0$ as $%
n\rightarrow \infty $. Such a sequence will be termed a \textit{fundamental
sequence}. The following result is proved exactly as its homologue in \cite[%
Theorem 3.1]{CMP}.

\begin{theorem}
\label{t3.1}Any bounded sequence $(u_{\varepsilon })_{\varepsilon \in E}$ in 
$L^{p}(Q)$ (where $E$ is a fundamental sequence and $1<p<\infty $) admits a
subsequence which is weakly $\Sigma $-convergent in $L^{p}(Q)$.
\end{theorem}

The next result can be proven as in \cite[Theorem 4.10]{Casado}.

\begin{theorem}
\label{t3.2}Any uniformly integrable sequence $(u_{\varepsilon
})_{\varepsilon \in E}$ in $L^{1}(Q)$ admits a subsequence which is weakly $%
\Sigma $-convergent in $L^{1}(Q)$.
\end{theorem}

We recall that a sequence $(u_{\varepsilon })_{\varepsilon >0}$ in $L^{1}(Q)$
is said to be uniformly integrable if $(u_{\varepsilon })_{\varepsilon >0}$
is bounded in $L^{1}(Q)$ and further $\sup_{\varepsilon
>0}\int_{X}\left\vert u_{\varepsilon }\right\vert dx\rightarrow 0$ as $%
\left\vert X\right\vert \rightarrow 0$ ($X$ being an integrable set in $Q$
with $\left\vert X\right\vert $ denoting the Lebesgue measure of $X$).

In order to deal with the convergence of a product of sequences we need to
define the concept of strong $\Sigma $-convergence.

\begin{definition}
\label{d3.2}\emph{A sequence }$(u_{\varepsilon })_{\varepsilon >0}\subset
L^{p}(Q)$\emph{\ }$(1\leq p<\infty )$\emph{\ is said to }strongly $\Sigma $%
-converge\emph{\ in }$L^{p}(Q)$\emph{\ to some }$u_{0}\in L^{p}(Q;\mathcal{B}%
_{A}^{p})$\emph{\ if it is weakly }$\Sigma $\emph{-convergent towards }$%
u_{0} $\emph{\ and further satisfies the following condition: }%
\begin{equation}
\left\| u_{\varepsilon }\right\| _{L^{p}(Q)}\rightarrow \left\| \widehat{u}%
_{0}\right\| _{L^{p}(Q\times \Delta (A))}.  \label{3.12}
\end{equation}%
\emph{We denote this by }$u_{\varepsilon }\rightarrow u_{0}$\emph{\ in }$%
L^{p}(Q)$\emph{-strong }$\Sigma $\emph{.}
\end{definition}

\begin{remark}
\label{r3.1}\emph{(1) By the above definition, the uniqueness of the limit
of such a sequence is ensured. (2) By \cite{26} it is immediate that for any 
}$u\in L^{p}(Q;A)$\emph{, the sequence }$(u^{\varepsilon })_{\varepsilon >0}$%
\emph{\ is strongly }$\Sigma $\emph{-convergent to }$\varrho (u)$\emph{.}
\end{remark}

The next result will be of capital interest in the homogenization process.

\begin{theorem}
\label{t3.3}Let $1<p,q<\infty $ and $r\geq 1$ be such that $1/r=1/p+1/q\leq
1 $. Assume $(u_{\varepsilon })_{\varepsilon \in E}\subset L^{q}(Q)$ is
weakly $\Sigma $-convergent in $L^{q}(Q)$ to some $u_{0}\in L^{q}(Q;\mathcal{%
B}_{A}^{q})$, and $(v_{\varepsilon })_{\varepsilon \in E}\subset L^{p}(Q)$
is strongly $\Sigma $-convergent in $L^{p}(Q)$ to some $v_{0}\in L^{p}(Q;%
\mathcal{B}_{A}^{p})$. Then the sequence $(u_{\varepsilon }v_{\varepsilon
})_{\varepsilon \in E}$ is weakly $\Sigma $-convergent in $L^{r}(Q)$ to $%
u_{0}v_{0}$.
\end{theorem}

\begin{proof}
We assume without lost of generality that our sequences are real values.
This assumption is fully motivated by the fact that in general, almost only
linear problems are of complex coefficients, and so in that case, the
linearity permits to work with real coefficients. This being so, we will
deeply make use of the following simple inequalities proved in \cite[Proof
of Lemma 2.4]{Zhikov1}: 
\begin{equation}
\begin{array}{l}
0\leq \left\vert a+tb\right\vert ^{p}-\left\vert a\right\vert
^{p}-pt\left\vert a\right\vert ^{p-2}ab\leq c\left\vert t\right\vert
^{1+s}(\left\vert a\right\vert ^{p}+\left\vert b\right\vert ^{p}) \\ 
\text{for each }\left\vert t\right\vert \leq 1\text{ and for every }a,b\in 
\mathbb{R}\text{, where } \\ 
s=\min (p-1,1)>0\text{ and }c>0\text{ is independent of }a,b\text{.}%
\end{array}
\label{3.13}
\end{equation}%
We proceed in two steps.\medskip 

\noindent \textbf{Step 1.} Set $p^{\prime }=p/(p-1)$, and let us first show
that the sequence $z_{\varepsilon }=\left\vert v_{\varepsilon }\right\vert
^{p-2}v_{\varepsilon }$ is weakly stochastically $\Sigma $-convergent to $%
\left\vert v_{0}\right\vert ^{p-2}v_{0}$ in $L^{p^{\prime }}(Q)$. To this
end, let $z\in L^{p^{\prime }}(Q;\mathcal{B}_{A}^{p^{\prime }})$ denote the
weak stochastic $\Sigma $-limit of $(z_{\varepsilon })_{\varepsilon \in E}$
in $L^{p^{\prime }}(Q)$ (up to a subsequence if necessary; in fact it is
easily seen that this sequence is bounded in $L^{p^{\prime }}(Q)$). Let $%
\varphi \in L^{p}(Q;A)$ with $\left\Vert \varphi \right\Vert
_{L^{p}(Q;A)}\leq 1$. We have by the second inequality in (\ref{3.13}) that 
\begin{equation*}
\int_{Q}\left\vert v_{\varepsilon }+t\varphi ^{\varepsilon }\right\vert
^{p}dx\leq \int_{Q}\left\vert v_{\varepsilon }\right\vert
^{p}dx+pt\int_{Q}z_{\varepsilon }\varphi ^{\varepsilon }dx+c_{1}\left\vert
t\right\vert ^{1+s}
\end{equation*}%
for $\left\vert t\right\vert \leq 1$, $c_{1}$ being a positive constant
independent of $\varepsilon $ (since the sequence $(v_{\varepsilon
})_{\varepsilon \in E}$ is bounded in $L^{p}(Q)$). Taking the $\lim
\inf_{E\ni \varepsilon \rightarrow 0}$ in the above inequality we get, by
virtue of (\ref{3.12}) (in Definition \ref{d3.2}) and the lower
semicontinuity property (see \cite{AIM}) that 
\begin{eqnarray*}
\iint_{Q\times \Delta (A)}\left\vert \widehat{v}_{0}+t\widehat{\varphi }%
\right\vert ^{p}dxd\beta &\leq &\iint_{Q\times \Delta (A)}\left\vert 
\widehat{v}_{0}\right\vert ^{p}dxd\beta \\
&&+pt\iint_{Q\times \Delta (A)}\widehat{z}\widehat{\varphi }dxd\beta
+c_{1}\left\vert t\right\vert ^{1+s}.
\end{eqnarray*}%
On the other hand, the first inequality in (\ref{3.13}) yields 
\begin{equation*}
\iint_{Q\times \Delta (A)}\left\vert \widehat{v}_{0}+t\widehat{\varphi }%
\right\vert ^{p}dxd\beta \geq \iint_{Q\times \Delta (A)}\left\vert \widehat{v%
}_{0}\right\vert ^{p}dxd\beta +pt\iint_{Q\times \Delta (A)}\left\vert 
\widehat{v}_{0}\right\vert ^{p-2}\widehat{v}_{0}dxd\beta ,
\end{equation*}%
hence 
\begin{equation*}
pt\iint_{Q\times \Delta (A)}\left\vert \widehat{v}_{0}\right\vert ^{p-2}%
\widehat{v}_{0}dxd\beta \leq pt\iint_{Q\times \Delta (A)}\widehat{z}\widehat{%
\varphi }dxd\beta +c_{1}\left\vert t\right\vert ^{1+s}.
\end{equation*}%
Now, taking in the above inequality $\varphi =\psi /\left\Vert \psi
\right\Vert _{L^{p}(Q;A)}$ for any arbitrary $\psi \in L^{p}(Q;A)$ the same
inequality holds for any arbitrary $\psi $ in place of $\varphi $, which,
together with the arbitrariness of the real number $t$ in $\left\vert
t\right\vert \leq 1$, gives $z=\left\vert v_{0}\right\vert ^{p-2}v_{0}$%
.\medskip

\noindent \textbf{Step 2}. Now, let us establish the convergence result $%
u_{\varepsilon }v_{\varepsilon }\rightarrow u_{0}v_{0}$ in $L^{r}(Q)$-weak $%
\Sigma $. First of all the sequence $(u_{\varepsilon }v_{\varepsilon
})_{\varepsilon \in E}$ is bounded in $L^{r}(Q)$. Next, let $\varphi \in
L^{r^{\prime }}(Q;A)$ and set 
\begin{equation*}
\ell =\lim_{E\ni \varepsilon \rightarrow 0}\int_{Q}u_{\varepsilon
}v_{\varepsilon }\varphi ^{\varepsilon }dx\text{ (possibly up to a
subsequence).}
\end{equation*}%
We need to show that $\ell =\iint_{Q\times \Delta (A)}\widehat{u}_{0}%
\widehat{v}_{0}\widehat{\varphi }dxd\beta $. First and foremost we have $%
\varphi ^{\varepsilon }\in L^{r^{\prime }}(Q)$ and so, $u_{\varepsilon
}\varphi ^{\varepsilon }\in L^{p^{\prime }}(Q)$ since $1/r^{\prime
}+1/q=1/p^{\prime }$ and $u_{\varepsilon }\in L^{q}(Q)$. Thus, once again by
the second inequality in (\ref{3.13}) and keeping in mind the definition of $%
z_{\varepsilon }$ in Step 1, one has 
\begin{eqnarray*}
\int_{Q}\left\vert z_{\varepsilon }+tu_{\varepsilon }\varphi ^{\varepsilon
}\right\vert ^{p^{\prime }}dx &\leq &\int_{Q}\left\vert z_{\varepsilon
}\right\vert ^{p^{\prime }}dx+p^{\prime }t\int_{Q}\left\vert z_{\varepsilon
}\right\vert ^{p-2}z_{\varepsilon }u_{\varepsilon }\varphi ^{\varepsilon
}dx+c_{1}\left\vert t\right\vert ^{1+s} \\
&=&\int_{Q}\left\vert v_{\varepsilon }\right\vert ^{p}dx+p^{\prime
}t\int_{Q}v_{\varepsilon }u_{\varepsilon }\varphi ^{\varepsilon
}dx+c_{1}\left\vert t\right\vert ^{1+s}
\end{eqnarray*}%
since $v_{\varepsilon }=\left\vert z_{\varepsilon }\right\vert ^{p^{\prime
}-2}z_{\varepsilon }$ and $\left\vert z_{\varepsilon }\right\vert
^{p^{\prime }}=\left\vert v_{\varepsilon }\right\vert ^{p}$. On the other
hand, one easily sees that the sequence $(u_{\varepsilon }\varphi
^{\varepsilon })_{\varepsilon \in E}$ is weakly $\Sigma $-convergent to $%
u_{0}\varrho (\varphi )$ in $L^{p^{\prime }}(Q)$, so that, passing to the
limit in the above inequality, using the lower semicontinuity property, we
get 
\begin{eqnarray*}
\iint_{Q\times \Delta (A)}\left\vert \widehat{z}+t\widehat{u}_{0}\widehat{%
\varphi }\right\vert ^{p^{\prime }}dxd\beta &\leq &\iint_{Q\times \Delta
(A)}\left\vert \widehat{v}_{0}\right\vert ^{p}dxd\beta +p^{\prime }t\ell
+c_{1}\left\vert t\right\vert ^{1+s} \\
&=&\iint_{Q\times \Delta (A)}\left\vert \widehat{z}\right\vert ^{p^{\prime
}}dxd\beta +p^{\prime }t\ell +c_{1}\left\vert t\right\vert ^{1+s},
\end{eqnarray*}%
since $z=\left\vert v_{0}\right\vert ^{p-2}v_{0}$ (as shown in Step 1), and
therefore, $\left\vert v_{0}\right\vert ^{p}=\left\vert z\right\vert
^{p^{\prime }}$. Besides, we have by the first inequality in (\ref{3.13})
that 
\begin{eqnarray*}
\iint_{Q\times \Delta (A)}\left\vert \widehat{z}+t\widehat{u}_{0}\widehat{%
\varphi }\right\vert ^{p^{\prime }}dxd\beta &\geq &\iint_{Q\times \Delta
(A)}\left\vert \widehat{z}\right\vert ^{p^{\prime }}dxd\beta +p^{\prime
}t\iint_{Q\times \Delta (A)}\left\vert \widehat{z}\right\vert ^{p^{\prime
}-2}\widehat{z}\widehat{u}_{0}\widehat{\varphi }dxd\beta \\
&=&\iint_{Q\times \Delta (A)}\left\vert \widehat{z}\right\vert ^{p^{\prime
}}dxd\beta +p^{\prime }t\iint_{Q\times \Delta (A)}\widehat{v}_{0}\widehat{u}%
_{0}\widehat{\varphi }dxd\beta
\end{eqnarray*}%
since $v_{0}=\left\vert z\right\vert ^{p^{\prime }-2}z$. We are therefore
led to 
\begin{equation*}
p^{\prime }t\iint_{Q\times \Delta (A)}\widehat{v}_{0}\widehat{u}_{0}\widehat{%
\varphi }dxd\beta \leq p^{\prime }t\ell +c_{1}\left\vert t\right\vert
^{1+s}\;\forall \left\vert t\right\vert \leq 1,
\end{equation*}%
hence $\ell =\iint_{Q\times \Delta (A)}\widehat{v}_{0}\widehat{u}_{0}%
\widehat{\varphi }dxd\beta $.
\end{proof}

The following result will be of great interest in practise. It is a mere
consequence of the preceding theorem.

\begin{corollary}
\label{c3.4}Let $(u_{\varepsilon })_{\varepsilon \in E}\subset L^{p}(Q)$ and 
$(v_{\varepsilon })_{\varepsilon \in E}\subset L^{p^{\prime }}(Q)\cap
L^{\infty }(Q)$ ($1<p<\infty $ and $p^{\prime }=p/(p-1)$) be two sequences
such that:

\begin{itemize}
\item[(i)] $u_{\varepsilon }\rightarrow u_{0}$ in $L^{p}(Q)$-weak $\Sigma $;

\item[(ii)] $v_{\varepsilon }\rightarrow v_{0}$ in $L^{p^{\prime }}(Q)$%
-strong $\Sigma $;

\item[(iii)] $(v_{\varepsilon })_{\varepsilon \in E}$ is bounded in $%
L^{\infty }(Q)$.
\end{itemize}

\noindent Then $u_{\varepsilon }v_{\varepsilon }\rightarrow u_{0}v_{0}$ in $%
L^{p}(Q)$-weak $\Sigma $.
\end{corollary}

\begin{proof}
By Theorem \ref{t3.3}, the sequence $(u_{\varepsilon }v_{\varepsilon
})_{\varepsilon \in E}$ $\Sigma $-converges towards $u_{0}v_{0}$ in $%
L^{1}(Q) $. Besides the same sequence is bounded in $L^{p}(Q)$ so that by
the Theorem \ref{t3.1}, it weakly $\Sigma $-converges in $L^{p}(Q)$ towards
some $w_{0}\in L^{p}(Q;\mathcal{B}_{A}^{p})$. This gives as a result $%
w_{0}=u_{0}v_{0}$.
\end{proof}

The strong $\Sigma $-convergence is a generalization of the strong
convergence as one can easily see in the following result whose easy proof
is left to the reader.

\begin{proposition}
\label{p3.2}Let $(u_{\varepsilon })_{\varepsilon \in E}\subset L^{p}(Q)$ $%
(1\leq p<\infty )$ be a strongly convergent sequence in $L^{p}(Q)$ to some $%
u_{0}\in L^{p}(Q)$. Then $(u_{\varepsilon })_{\varepsilon \in E}$ strongly $%
\Sigma $-converges in $L^{p}(Q)$ towards $u_{0}$.
\end{proposition}

In the first step of the proof of Theorem \ref{t3.3} we have proven the
following assertion: If $v_{\varepsilon }\rightarrow v_{0}$ in $L^{p}(Q)$%
-strong $\Sigma $ then $\left\vert v_{\varepsilon }\right\vert
^{p-2}v_{\varepsilon }\rightarrow \left\vert v_{0}\right\vert ^{p-2}v_{0}$
in $L^{p^{\prime }}(Q)$-weak $\Sigma $. One can weaken the above strong
convergence condition and obtain, under an additional weak convergence
assumption, the following result: If $u_{\varepsilon }\rightarrow u_{0}$ in $%
L^{p}(Q)$-weak $\Sigma $ and $\left\vert u_{\varepsilon }\right\vert
^{p-2}u_{\varepsilon }\rightarrow v_{0}$ in $L^{p^{\prime }}(Q)$-weak $%
\Sigma $, then 
\begin{equation*}
\iint_{Q\times \Delta (A)}\widehat{u}_{0}\widehat{v}_{0}dxd\beta \leq 
\underset{\varepsilon \rightarrow 0}{\lim \inf }\int_{Q}\left\vert
u_{\varepsilon }\right\vert ^{p}dx\text{.}
\end{equation*}%
Moreover if the above inequality holds as an equality, then $%
v_{0}=\left\vert u_{0}\right\vert ^{p-2}u_{0}$.

The above result is a particular case of a general situation stated in the
following

\begin{theorem}
\label{t3.4}Let $1<p<\infty $. Let $(x,y,\lambda )\mapsto a(x,y,\lambda )$,
from $\overline{Q}\times \mathbb{R}^{N}\times \mathbb{R}^{m}$ to $\mathbb{R}%
^{m}$ be a vector-valued function which is of Carath\'{e}odory's type, i.e., 
\emph{(i)} and \emph{(ii)} below are satisfied:

\begin{itemize}
\item[(i)] $a(x,\cdot ,\lambda )$ is measurable for any $(x,\lambda )\in 
\overline{Q}\times \mathbb{R}^{m}$

\item[(ii)] $a(\cdot ,y,\cdot )$ is continuous for almost all $y\in \mathbb{R%
}^{N}$,
\end{itemize}

\noindent and further satisfies the following conditions:

\begin{itemize}
\item[(iii)] $\left\vert a(x,y,\lambda )\right\vert \leq c(\left\vert
\lambda \right\vert ^{p-1}+1)$

\item[(iv)] $\left( a(x,y,\lambda )-a(x,y,\lambda ^{\prime })\right) \cdot
(\lambda -\lambda ^{\prime })\geq 0$

\item[(v)] $a(x,\cdot ,\lambda )\in (B_{A}^{p^{\prime }})^{m}$
\end{itemize}

\noindent for all $x\in \overline{Q}$ and all $\lambda ,\lambda ^{\prime
}\in \mathbb{R}^{m}$, where $c$ is a positive constant independent of $%
(x,y,\lambda )$. Finally let $(v_{\varepsilon })_{\varepsilon \in E}\subset
L^{p}(Q)^{m}$ be a sequence which componentwise weakly $\Sigma $-converges
towards $v_{0}\in L^{p}(Q;(\mathcal{B}_{A}^{p})^{m})$ as $E\ni \varepsilon
\rightarrow 0$. Then the sequence $(a^{\varepsilon }(\cdot ,v_{\varepsilon
}))_{\varepsilon \in E}$ defined by $a^{\varepsilon }(\cdot ,v_{\varepsilon
})(x)=a(x,x/\varepsilon _{1},v_{\varepsilon }(x))$ for $x\in Q$, is weakly $%
\Sigma $-convergent in $L^{p^{\prime }}(Q)^{m}$ (up to a subsequence) to
some $z_{0}\in L^{p^{\prime }}(Q;(\mathcal{B}_{A}^{p^{\prime }})^{m})$ such
that 
\begin{equation}
\iint_{Q\times \Delta (A)}\widehat{z}_{0}\cdot \widehat{v}_{0}dxd\beta \leq 
\underset{E\ni \varepsilon \rightarrow 0}{\lim \inf }\int_{Q}a^{\varepsilon
}(\cdot ,v_{\varepsilon })\cdot v_{\varepsilon }dx.  \label{3.14}
\end{equation}%
Moreover if \emph{(\ref{3.14})} holds as an equality, then $%
z_{0}(x,y)=a(x,y,v_{0}(x,y))$.
\end{theorem}

We will make use of the following lemma.

\begin{lemma}
\label{l3.1}Let $F_{1}$ and $F_{2}$ be two Banach spaces, $(Y,\mathcal{M}%
,\mu )$ a measure space, $X$ a $\mu $-measurable subset of $Y$, and $%
g:X\times F_{1}\rightarrow F_{2}$ a Carath\'{e}odory mapping. For each
measurable function $u:X\rightarrow F_{1}$, let $G(u)$ be the measurable
function $x\mapsto g(x,u(x))$, from $X$ to $F_{2}$. If $G:u\mapsto G(u)$
maps $L^{p}(X;F_{1})$ into $L^{r}(X;F_{2})$ ($1\leq p,r<\infty $) then $G$
is continuous in the norm topology.
\end{lemma}

\begin{proof}
A look at the proof of \cite[Chap. IV, Proposition 1.1]{18} shows that one
may replace in that proof, the Borel subset $Q$ of $\mathbb{R}^{n}$ by the
measurable subset $X$ of $Y$, $E$ by $F_{1}$, $F$ by $F_{2}$, and get
readily our result.
\end{proof}

\begin{proof}[\textit{Proof of Theorem} \protect\ref{t3.4}]
By (iii) the sequence $(a^{\varepsilon }(\cdot ,v_{\varepsilon
}))_{\varepsilon \in E}$ is bounded in $L^{p^{\prime }}(Q)^{m}$, thus there
exists a subsequence $E^{\prime }$ from $E$ and a function $z_{0}\in
L^{p^{\prime }}(Q;(\mathcal{B}_{A}^{p^{\prime }})^{m})$ such that $%
a^{\varepsilon }(\cdot ,v_{\varepsilon })\rightarrow z_{0}$ in $L^{p^{\prime
}}(Q)^{m}$-weak $\Sigma $ as $E\ni \varepsilon \rightarrow 0$. Let us show (%
\ref{3.14}). For that purpose, let $\psi \in \lbrack \mathcal{C}_{0}^{\infty
}(Q)\otimes A]^{m}$ (which is dense in $L^{p}(Q;A)^{m}$); then the function $%
(x,y)\mapsto a(x,y,\psi (x,y))$ lies in $\mathcal{C}(\overline{Q}%
;B_{A}^{p^{\prime },\infty })^{m}$. Indeed, as a result of (ii) the function 
$a(\cdot ,y,\psi (\cdot ,y))$ is continuous. Moreover for each fixed $x\in 
\overline{Q}$, $a(x,\cdot ,\psi (x,\cdot ))\in (B_{A}^{p^{\prime }})^{m}$:
in fact $\psi (x,\cdot )\in (A)^{m}$, and it suffices to check that $%
a(x,\cdot ,\phi )\in (B_{A}^{p^{\prime }})^{m}$ for any $\phi \in (A)^{m}$.
But since the function $\phi $ is bounded, let $K\subset \mathbb{R}^{m}$ be
a compact set such that $\phi (y)\in K$ for all $y\in \mathbb{R}^{N}$.
Viewing $\lambda \mapsto a(x,\cdot ,\lambda )$ as a function defined on $K$,
we have that this function belongs to $\mathcal{C}(K;(B_{A}^{p^{\prime
}})^{m})$ (use also hypothesis (v)), so that by the classical
Stone-Weierstrass theorem one has $a(x,\cdot ,\phi )\in (B_{A}^{p^{\prime
}})^{m}$; see either \cite[Proposition 1]{AML} or \cite[Proposition 3.1]%
{ACAP}. As a result, we end up with the fact that the function $(x,y)\mapsto
a(x,y,\psi (x,y))$ belongs to $\mathcal{C}(\overline{Q};B_{A}^{p^{\prime
}})^{m}$, hence to $\mathcal{C}(\overline{Q};B_{A}^{p^{\prime },\infty
})^{m} $ where $B_{A}^{p^{\prime },\infty }=B_{A}^{p^{\prime }}\cap
L^{\infty }(\mathbb{R}_{y}^{N})$.

We now use (iv) to get 
\begin{equation*}
\int_{Q}\left( a^{\varepsilon }(\cdot ,v_{\varepsilon })-a^{\varepsilon
}(\cdot ,\psi ^{\varepsilon })\right) \cdot (v_{\varepsilon }-\psi
^{\varepsilon })dx\geq 0
\end{equation*}%
or equivalently, 
\begin{eqnarray*}
\int_{Q}a^{\varepsilon }(\cdot ,v_{\varepsilon })\cdot v_{\varepsilon }dx
&\geq &\int_{Q}a^{\varepsilon }(\cdot ,v_{\varepsilon })\cdot \psi
_{\varepsilon }dx+\int_{Q}a^{\varepsilon }(\cdot ,\psi _{\varepsilon })\cdot
v_{\varepsilon }dx \\
&&-\int_{Q}a^{\varepsilon }(\cdot ,\psi _{\varepsilon })\cdot \psi
_{\varepsilon }dx.
\end{eqnarray*}%
Taking the $\lim \inf_{E^{\prime }\ni \varepsilon \rightarrow 0}$ of the
both sides of the above inequality we get 
\begin{eqnarray}
\underset{E^{\prime }\ni \varepsilon \rightarrow 0}{\lim \inf }%
\int_{Q}a^{\varepsilon }(\cdot ,v_{\varepsilon })\cdot v_{\varepsilon }dx
&\geq &\iint_{Q\times \Delta (A)}\widehat{z}_{0}\cdot \widehat{\psi }%
dxd\beta +\iint_{Q\times \Delta (A)}\widehat{a}(\cdot ,\widehat{\psi })\cdot 
\widehat{v}_{0}dxd\beta   \label{3.15} \\
&&-\iint_{Q\times \Delta (A)}\widehat{a}(\cdot ,\widehat{\psi })\cdot 
\widehat{\psi }dxd\beta   \notag
\end{eqnarray}%
where: for the first integral on the right-hand side of (\ref{3.15}), we
have used the definition of the weak $\Sigma $-convergence for the sequence $%
(a^{\varepsilon }(\cdot ,v_{\varepsilon }))_{\varepsilon }$, for the second
integral, we have used the definition of the weak $\Sigma $-convergence of $%
(v_{\varepsilon })_{\varepsilon }$ associated to the fact that (\ref{3.1})
also holds for test functions in $\mathcal{C}(\overline{Q};B_{A}^{p^{\prime
},\infty })$ and so by taking $a(\cdot ,\psi )$ as a test function, and
finally for the last integral, we use the same argument as for the preceding
one. Therefore, subtracting $\iint_{Q\times \Delta (A)}\widehat{z}_{0}\cdot 
\widehat{v}_{0}dxd\beta $ from each member of (\ref{3.15}), we end up with 
\begin{equation}
\begin{array}{l}
\underset{E^{\prime }\ni \varepsilon \rightarrow 0}{\lim \inf }%
\int_{Q}a^{\varepsilon }(\cdot ,v_{\varepsilon })\cdot v_{\varepsilon
}dx-\iint_{Q\times \Delta (A)}\widehat{z}_{0}\cdot \widehat{v}_{0}dxd\beta
\geq  \\ 
-\iint_{Q\times \Delta (A)}\left( \widehat{z}_{0}-\widehat{a}(\cdot ,%
\widehat{\psi })\right) \cdot (\widehat{v}_{0}-\widehat{\psi })dxd\beta \ \
\forall \psi \in \lbrack \mathcal{C}_{0}^{\infty }(Q)\otimes A]^{m}\text{.}%
\end{array}
\label{3.16}
\end{equation}%
The right-hand side of (\ref{3.16}) is of the form $g(x,s,\widehat{\psi }%
(x,s))$ and, due to the fact that $\widehat{z}_{0}\in L^{p^{\prime
}}(Q\times \Delta (A))^{m}$, one easily deduces from (iv) that $g(x,s,%
\widehat{\psi })\in L^{1}(Q\times \Delta (A))$ for any $\widehat{\psi }\in
L^{p}(Q\times \Delta (A))^{m}$, so that the operator $G$ defined here as in
Lemma \ref{l3.1} (by taking there $X=Q\times \Delta (A)$, $%
F_{1}=L^{p}(Q\times \Delta (A))^{m}$, $F_{2}=L^{1}(Q\times \Delta (A))$)
maps $L^{p}(X;F_{1})$ into $L^{1}(X;F_{2})$. In view of Lemma \ref{l3.1}, $G$
is continuous under the norm topology. As a result, the inequality (\ref%
{3.16}) holds for any $\widehat{\psi }\in L^{p}(Q\times \Delta (A))^{m}$
(that is for any $\psi \in L^{p}(Q;\mathcal{B}_{A}^{p})^{m}$). Hence taking
in (\ref{3.16}) $\psi =v_{0}$ we readily get (\ref{3.14}).

For the last part of the theorem, assuming that (\ref{3.14}) is actually an
equality, we return to (\ref{3.16}) and take there $\psi =v_{0}+tw$, $w\in
L^{p}(Q;\mathcal{B}_{A}^{p})^{m}$ being arbitrarily fixed and $t>0$. Then we
get 
\begin{equation*}
\iint_{Q\times \Delta (A)}\left( \widehat{z}_{0}-\widehat{a}(\cdot ,\widehat{%
v}_{0}+t\widehat{w})\right) \cdot \widehat{w}dxd\beta \leq 0\ \ \forall w\in
L^{p}(Q;\mathcal{B}_{A}^{p})^{m}\text{.}
\end{equation*}%
Letting $t\rightarrow 0$, and then changing $w$ for $-w$, we end up with 
\begin{equation*}
\iint_{Q\times \Delta (A)}\left( \widehat{z}_{0}-\widehat{a}(\cdot ,\widehat{%
v}_{0})\right) \cdot \widehat{w}dxd\beta =0\ \ \forall w\in L^{p}(Q;\mathcal{%
B}_{A}^{p})^{m}\text{,}
\end{equation*}%
which implies $z_{0}=a(\cdot ,v_{0})$.
\end{proof}

As was said before the statement of Theorem \ref{t3.4}, if we take $%
a(x,y,\lambda )=\left\vert \lambda \right\vert ^{p-2}\lambda $ and $m=1$,
then we arrive at the claimed conclusion by the above theorem. The next
result gives the characterization of the $\Sigma $-limit of sequences
involving gradients.

\begin{theorem}
\label{t3.5}Let $1<p<\infty $. Let $(u_{\varepsilon })_{\varepsilon \in E}$
be a bounded sequence in $W^{1,p}(\Omega )$. Then there exist a subsequence $%
E^{\prime }$ of $E$, and a couple $(u_{0},u_{1})\in W^{1,p}(\Omega
;I_{A}^{p})\times L^{p}(\Omega ;\mathcal{B}_{\#A}^{1,p})$ such that, as $%
E^{\prime }\ni \varepsilon \rightarrow 0$, 

\begin{itemize}
\item[(i)] $u_{\varepsilon }\rightarrow u_{0}$ in $L^{p}(Q)$-weak $\Sigma $;

\item[(ii)] $\partial u_{\varepsilon }/\partial x_{i}\rightarrow \partial
u_{0}/\partial x_{i}+\overline{\partial }u_{1}/\partial y_{i}$ in $L^{p}(Q)$%
-weak $\Sigma $, $1\leq i\leq N$.
\end{itemize}
\end{theorem}

\begin{proof}
Since the sequences $(u_{\varepsilon })_{\varepsilon \in E}$ and $(\nabla
u_{\varepsilon })_{\varepsilon \in E}$ are bounded respectively in $L^{p}(Q)$
and in $L^{p}(Q)^{N}$, there exist (see Theorem \ref{t3.1}) a subsequence $%
E^{\prime }$ of $E$ and $u_{0}\in L^{p}(Q;\mathcal{B}_{A}^{p})$, $%
v=(v_{j})_{j}\in L^{p}(Q;\mathcal{B}_{A}^{p})^{N}$ such that $u_{\varepsilon
}\rightarrow u_{0}\ $in $L^{p}(Q)$-weak $\Sigma $ and $\frac{\partial
u_{\varepsilon }}{\partial x_{j}}\rightarrow v_{j}$ in $L^{p}(Q)$-weak $%
\Sigma $. For $\Phi \in (\mathcal{C}_{0}^{\infty }(Q)\otimes A^{\infty })^{N}
$ we have 
\begin{equation*}
\int_{Q}\varepsilon \nabla u_{\varepsilon }\cdot \Phi ^{\varepsilon
}dx=-\int_{Q}\left( u_{\varepsilon }(\Div_{y}\Phi )^{\varepsilon
}+\varepsilon u_{\varepsilon }(\Div\Phi )^{\varepsilon }\right) dx.
\end{equation*}%
Letting $E^{\prime }\ni \varepsilon \rightarrow 0$ we get%
\begin{equation*}
-\iint_{Q\times \Delta (A)}\widehat{u}_{0}\widehat{\Div}\widehat{\Phi }%
dxd\beta =0.
\end{equation*}%
This shows that $\overline{D}_{y}u_{0}=0$, which means that $u_{0}(x,\cdot
)\in I_{A}^{p}$ (see (\ref{2.5})), that is, $u_{0}\in L^{p}(Q;I_{A}^{p})$.
Next let $\Phi _{\varepsilon }(x)=\varphi (x)\Psi (x/\varepsilon )$ ($x\in Q$%
) with $\varphi \in \mathcal{C}_{0}^{\infty }(Q)$ and $\Psi =(\psi
_{j})_{1\leq j\leq N}\in (A^{\infty })^{N}$ with ${\Div}_{y}\Psi =0$. Clearly%
\begin{equation*}
\sum_{j=1}^{N}\int_{Q}\frac{\partial u_{\varepsilon }}{\partial x_{j}}%
\varphi \psi _{j}^{\varepsilon }dx=-\sum_{j=1}^{N}\int_{Q}u_{\varepsilon
}\psi _{j}^{\varepsilon }\frac{\partial \varphi }{\partial x_{j}}dx
\end{equation*}%
where $\psi _{j}^{\varepsilon }(x)=\psi _{j}(x/\varepsilon )$. Passing to
the limit in the above equation when $E^{\prime }\ni \varepsilon \rightarrow
0$ we get 
\begin{equation}
\sum_{j=1}^{N}\iint_{Q\times \Delta (A)}\widehat{v}_{j}\varphi \widehat{\psi 
}_{j}dxd\beta =-\sum_{j=1}^{N}\iint_{Q\times \Delta (A)}\widehat{u}_{0}%
\widehat{\psi }_{j}\frac{\partial \varphi }{\partial x_{j}}dxd\beta .
\label{2.3}
\end{equation}%
First, taking $\Phi =(\varphi \delta _{ij})_{1\leq i\leq N}$ with $\varphi
\in \mathcal{C}_{0}^{\infty }(Q)$ (for each fixed $1\leq j\leq N$) in (\ref%
{2.3}) we obtain 
\begin{equation}
\int_{Q}M(v_{j})\varphi dx=-\int_{Q}M(u_{0})\frac{\partial \varphi }{%
\partial x_{j}}dx  \label{2.4'}
\end{equation}%
and reminding that $M(v_{j})\in L^{p}(Q)$ we have by (\ref{2.4'}) that $%
\frac{\partial u_{0}}{\partial x_{j}}\in L^{p}(Q;I_{A}^{p})$, where $\frac{%
\partial u_{0}}{\partial x_{j}}$ is the distributional derivative of $u_{0}$
with respect to $x_{j}$. We deduce that $u_{0}\in W^{1,p}(Q;I_{A}^{p})$.
Coming back to (\ref{2.3}) we get 
\begin{equation*}
\iint_{Q\times \Delta (A)}\left( \widehat{\mathbf{v}}-\nabla \widehat{u}%
_{0}\right) \cdot \widehat{\Psi }\varphi dxd\beta =0\text{,}
\end{equation*}%
and so, as $\varphi $ is arbitrarily fixed, 
\begin{equation*}
\int_{\Delta (A)}\left( \widehat{\mathbf{v}}(x,s)-\nabla \widehat{u}%
_{0}(x,s)\right) \cdot \widehat{\Psi }(s)d\beta =0\text{, all }\Psi \text{
and a.e. }x.
\end{equation*}%
Therefore, the Proposition \ref{p2.3} provides us with a function $%
u_{1}(x,\cdot )\in \mathcal{B}_{\#A}^{1,p}$ such that 
\begin{equation*}
\mathbf{v}(x,\cdot )-Du_{0}(x)=\overline{D}_{y}u_{1}(x,\cdot )\text{ for
a.e. }x\in Q\text{.}
\end{equation*}%
This yields the existence of a function $u_{1}:x\mapsto u_{1}(x,\cdot )$
lying in $L^{p}(Q;\mathcal{B}_{\#A}^{1,p})$ such that $\mathbf{v}=Du_{0}+%
\overline{D}_{y}u_{1}$. This completes the proof.
\end{proof}

\begin{remark}
\label{r3.4}\emph{The conclusion of Theorem \ref{t3.5} generalizes all the
results so far existing in the framework of deterministic homogenization
theory. Indeed, if we assume the algebra }$A$\emph{\ to be ergodic, then }$%
I_{A}^{p}$\emph{\ consists of constant functions, and the function }$u_{0}$%
\emph{\ in Theorem \ref{t3.5} does not depend on }$y$\emph{, that is, }$%
u_{0}\in W^{1,p}(Q)$\emph{.}
\end{remark}

The $\Sigma $-convergence method also applies to time dependent functions.
To see this, let $A_{y}$ and $A_{\tau }$ be two algebras wmv on $\mathbb{R}%
_{y}^{N}$ and $\mathbb{R}_{\tau }$ respectively. Let $A=A_{y}\odot A_{\tau }$
be their product \cite{26, CMP, CPAA}. We know that $A$ is the closure in $%
BUC(\mathbb{R}_{y,\tau }^{N+1})$ of the tensor product $A_{y}\otimes A_{\tau
}$. Points in $\Delta (A_{y})$ and $\Delta (A_{\tau })$ are denoted
respectively by $s$ and $s_{0}$. We know that $\Delta (A)=\Delta
(A_{y})\times \Delta (A_{\tau })$ (Cartesian product) and further if $\beta $
(resp. $\beta _{y}$, $\beta _{\tau }$) is the $M$-measure for $A$ (resp. $%
A_{y}$, $A_{\tau }$) then $\beta =\beta _{y}\otimes \beta _{\tau }$; the
last equality follows in an obvious way by the density of $A_{y}\otimes
A_{\tau }$ in $A$ and by the Fubini's theorem.

Now, let $0<T<\infty $ and let $Q$ be an open subset in $\mathbb{R}^{N}$.
Set $Q_{T}=Q\times (0,T)$. We have the following well-known time-dependent
version of the $\Sigma $-convergence.

\begin{definition}
\label{d3.3}\emph{A sequence }$(u_{\varepsilon })_{\varepsilon >0}\subset
L^{p}(Q_{T})$\emph{\ (}$1\leq p<\infty $\emph{) is said to }weakly $\Sigma $%
-converge\emph{\ in }$L^{p}(Q_{T})$\emph{\ to some }$u_{0}\in L^{p}(Q_{T};%
\mathcal{B}_{A}^{p})$\emph{\ if as }$\varepsilon \rightarrow 0$\emph{, we
have }%
\begin{equation*}
\int_{Q_{T}}u_{\varepsilon }(x,t)f\left( x,t,\frac{x}{\varepsilon _{1}},%
\frac{t}{\varepsilon _{2}}\right) dxdt\rightarrow \iint_{Q_{T}\times \Delta
(A)}\widehat{u}_{0}(x,t,s,s_{0})\widehat{f}(x,t,s,s_{0})dxdtd\beta
\end{equation*}%
\emph{for every }$f\in L^{p^{\prime }}(Q_{T};A)$\emph{\ (}$1/p^{\prime
}=1-1/p$\emph{).}
\end{definition}

In Definition \ref{d3.3} we assume $\varepsilon _{1}$ and $\varepsilon _{2}$
to be any sequences of positive real numbers such that $\varepsilon
_{1},\varepsilon _{2}\rightarrow 0$ when $\varepsilon \rightarrow 0$. We
recall here that, as before, $\widehat{u}_{0}=\mathcal{G}_{1}\circ u_{0}$
and $\widehat{f}=\mathcal{G}\circ f$, $\mathcal{G}_{1}$ being the isometric
isomorphism sending $\mathcal{B}_{A}^{p}$ onto $L^{p}(\Delta (A))$ and $%
\mathcal{G}$, the Gelfand transformation on $A$.

The following important result whose proof is copied on that of Theorem \ref%
{t3.5}, and is therefore omitted.

\begin{theorem}
\label{t3.6}Let $1<p<\infty $. Let $Q$ be an open subset in $\mathbb{\mathbb{%
R}}^{N}$. Let $A=A_{y}\odot A_{\tau }$ be as above. Let $(u_{\varepsilon
})_{\varepsilon \in E}$ be a bounded sequence in $L^{p}(0,T;W^{1,p}(Q))$.
There exist a subsequence $E^{\prime }$ from $E$ and a couple $\mathbf{u}%
=(u_{0},u_{1})\in L^{p}(0,T;W^{1,p}(Q;I_{A}^{p}))\times L^{p}(Q_{T};\mathcal{%
B}_{A_{\tau }}^{p}(\mathbb{\mathbb{R}}_{\tau };\mathcal{B}_{\#A_{y}}^{1,p}))$
such that, as $E^{\prime }\ni \varepsilon \rightarrow 0$, 
\begin{equation*}
u_{\varepsilon }\rightarrow u_{0}\text{\ in }L^{p}(Q)\text{-weak }\Sigma 
\text{\ \ \ \ \ \ \ \ \ }
\end{equation*}%
and 
\begin{equation*}
\frac{\partial u_{\varepsilon }}{\partial x_{j}}\rightarrow \frac{\partial
u_{0}}{\partial x_{j}}+\frac{\overline{\partial }u_{1}}{\partial y_{j}}\text{%
\ in }L^{p}(Q_{T})\text{-weak }\Sigma \;(1\leq j\leq N).
\end{equation*}
\end{theorem}

\subsection{The $\Sigma $-convergence for stochastic processes}

In order to deal with homogenization problems related to \textit{stochastic}
PDEs we need to give a suitable notion of $\Sigma $-convergence adapted to
stochastic processes. In all that follows, $Q$ and $T$ are as above.\medskip

Let $(\Omega ,\mathcal{F},\mathbb{P})$ be a probability space. The
expectation on $(\Omega ,\mathcal{F},\mathbb{P})$ will throughout be denoted
by $\mathbb{E}$. Let us first recall the definition of the Banach space of
bounded $\mathcal{F}$-measurable functions. Denoting by $F(\Omega )$ the
Banach space of all bounded functions $f:\Omega \rightarrow \mathbb{R}$
(with the sup norm), we define $B(\Omega )$ as the closure in $F(\Omega )$
of the vector space $H(\Omega )$ consisting of all finite linear
combinations of the characteristic functions $1_{X}$ of sets $X\in \mathcal{F%
}$. Since $\mathcal{F}$ is an $\sigma $-algebra, $B(\Omega )$ is the Banach
space of all bounded $\mathcal{F}$-measurable functions. Likewise we define
the space $B(\Omega ;Z)$ of all bounded $(\mathcal{F},B_{Z})$-measurable
functions $f:\Omega \rightarrow Z$, where $Z$ is a Banach space endowed with
the $\sigma $-algebra of Borelians $B_{Z}$. The tensor product $B(\Omega
)\otimes Z$ is a dense subspace of $B(\Omega ;Z)$: this follows from the
obvious fact that $B(\Omega )$ can be viewed as a space of continuous
functions over the \textit{gamma-compactification} \cite{Zhdanok} of the
measurable space $(\Omega ,\mathcal{F})$, which is a compact topological
space. Next, for $X$ a Banach space, we denote by $L^{p}(\Omega ,\mathcal{F},%
\mathbb{P};X)$ the space of $X$-valued random variables $u$ such that $%
\left\Vert u\right\Vert _{X}$ is $L^{p}(\Omega ,\mathcal{F},\mathbb{P})$%
-integrable.

This being so, we still recall some preliminary as in the preceding
subsection. Let $A_{y}$ and $A_{\tau }$ be two algebras wmv on $\mathbb{R}%
_{y}^{N}$ and $\mathbb{R}_{\tau }$ respectively, and let $A=A_{y}\odot
A_{\tau }$ be their product. We denote by $\Delta (A_{y})$ (resp. $\Delta
(A_{\tau })$, $\Delta (A)$) the spectrum of $A_{y}$ (resp. $A_{\tau }$, $A$%
). The same letter $\mathcal{G}$ will denote the Gelfand transformation on $%
A_{y}$, $A_{\tau }$ and $A$, as well. Points in $\Delta (A_{y})$ (resp. $%
\Delta (A_{\tau })$) are denoted by $s$ (resp. $s_{0}$). The $M$-measure on
the compact space $\Delta (A_{y})$ (resp. $\Delta (A_{\tau })$) is denoted
by $\beta _{y}$ (resp. $\beta _{\tau }$). We know that $\Delta (A)=\Delta
(A_{y})\times \Delta (A_{\tau })$ and the $M$-measure on $\Delta (A)$ is the
product measure $\beta =\beta _{y}\otimes \beta _{\tau }$. Points in $\Omega 
$ are as usual denoted by $\omega $. Unless otherwise stated, random
variables will always be considered on the probability space $(\Omega ,%
\mathcal{F},\mathbb{P})$. The other notations are as before this subsection.

\begin{definition}
\label{d3.4}\emph{A sequence of random variables }$(u_{\varepsilon
})_{\varepsilon >0}\subset L^{p}(\Omega ,\mathcal{F},\mathbb{P}%
;L^{p}(Q_{T})) $\emph{\ (}$1\leq p<\infty $\emph{) is said to }weakly $%
\Sigma $-converge\emph{\ in }$L^{p}(Q_{T}\times \Omega )$\emph{\ to some
random variable }$u_{0}\in L^{p}(\Omega ,\mathcal{F},\mathbb{P};L^{p}(Q_{T};%
\mathcal{B}_{A}^{p}))$\emph{\ if as }$\varepsilon \rightarrow 0$\emph{, we
have } 
\begin{equation}
\begin{array}{l}
\int_{Q_{T}\times \Omega }u_{\varepsilon }(x,t,\omega )f\left( x,t,\frac{x}{%
\varepsilon },\frac{t}{\varepsilon },\omega \right) dxdtd\mathbb{P} \\ 
\ \ \ \ \ \ \ \ \ \ \ \ \ \ \rightarrow \iint_{Q_{T}\times \Omega \times
\Delta (A)}\widehat{u}_{0}(x,t,s,s_{0},\omega )\widehat{f}%
(x,t,s,s_{0},\omega )dxdtd\mathbb{P}d\beta%
\end{array}
\label{3.17}
\end{equation}%
\emph{for every }$f\in L^{p^{\prime }}(\Omega ,\mathcal{F},\mathbb{P}%
;L^{p^{\prime }}(Q_{T};A))$\emph{\ (}$1/p^{\prime }=1-1/p$\emph{), where }$%
\widehat{u}_{0}=\mathcal{G}_{1}\circ u_{0}$\emph{\ and }$\widehat{f}=%
\mathcal{G}_{1}\circ (\varrho \circ f)=\mathcal{G}\circ f$\emph{. We express
this by writing} $u_{\varepsilon }\rightarrow u_{0}$ in $L^{p}(Q_{T}\times
\Omega )$-weak $\Sigma $.
\end{definition}

\begin{remark}
\label{r3.0}\emph{The above weak }$\Sigma $\emph{-convergence in }$%
L^{p}(Q_{T}\times \Omega )$\emph{\ implies the weak convergence in }$%
L^{p}(Q_{T}\times \Omega )$\emph{. One can show as in the usual setting of }$%
\Sigma $\emph{-convergence method that each }$f\in L^{p}(\Omega ,\mathcal{F},%
\mathbb{P};L^{p}(Q_{T};A))$\emph{\ weakly }$\Sigma $\emph{-converges to }$%
\varrho \circ f$\emph{. Definition \ref{d3.4} generalizes Definitions \ref%
{d3.1} and \ref{d3.3} straightforwardly.}
\end{remark}

In order to simplify the notation, we will henceforth denote $L^{p}(\Omega ,%
\mathcal{F},\mathbb{P};X)$ merely by $L^{p}(\Omega ;X)$ if it is understood
from the context and there is no fear of confusion. Assume for a while that $%
p=2$. Then the definition above can be formally motivated by the following
fact: using the chaos expansion of $u_{\varepsilon }$ and $f$ we get $%
u_{\varepsilon }(x,t,\omega )=\sum_{j=1}^{\infty }u_{\varepsilon
,j}(x,t)\Phi _{j}(\omega )$ and $f(x,t,y,\tau ,\omega )=\sum_{k=1}^{\infty
}f_{k}(x,t,y,\tau )\Phi _{k}(\omega )$ where $u_{\varepsilon ,j}\in
L^{2}(Q_{T})$ and $f_{k}\in L^{2}(Q_{T};A)$, so that 
\begin{equation*}
\int_{Q_{T}\times \Omega }u_{\varepsilon }(x,t,\omega )f\left( x,t,\frac{x}{%
\varepsilon },\frac{t}{\varepsilon },\omega \right) dxdtd\mathbb{P}
\end{equation*}%
can be formally written as 
\begin{equation*}
\sum_{j,k=1}^{\infty }\int_{\Omega }\Phi _{j}(\omega )\Phi _{k}(\omega )d%
\mathbb{P}\int_{Q_{T}}u_{\varepsilon ,j}(x,t)f_{k}\left( x,t,\frac{x}{%
\varepsilon },\frac{t}{\varepsilon }\right) dxdt,
\end{equation*}%
and by the usual $\Sigma $-convergence method (see Definition \ref{d3.3}),
as $\varepsilon \rightarrow 0$, 
\begin{equation*}
\int_{Q_{T}}u_{\varepsilon ,j}(x,t)f_{k}\left( x,t,\frac{x}{\varepsilon },%
\frac{t}{\varepsilon }\right) dxdt\rightarrow \iint_{Q_{T}\times \Delta (A)}%
\widehat{u}_{0,j}(x,t,s,s_{0})\widehat{f}_{k}\left( x,t,s,s_{0}\right)
dxdtd\beta \text{.}
\end{equation*}%
Hence, by setting 
\begin{equation*}
\widehat{u}_{0}(x,t,s,s_{0},\omega )=\sum_{j=1}^{\infty }\widehat{u}%
_{0,j}(x,t,s,s_{0})\Phi _{j}(\omega );\;\widehat{f}\left( x,t,s,s_{0},\omega
\right) =\sum_{k=1}^{\infty }\widehat{f}_{k}\left( x,t,s,s_{0}\right) \Phi
_{k}(\omega )
\end{equation*}%
we get (\ref{3.17}). This is what can formally motivate our definition.

As in the preceding subsection, all the results together with their proofs
can be carried in the present context, mutatis mutandis. We just state a few.

\begin{theorem}
\label{t3.7}Let $1<p<\infty $. Let $(u_{\varepsilon })_{\varepsilon \in
E}\subset L^{p}(\Omega ;L^{p}(Q_{T}))$ be a sequence of random variables
verifying the following boundedness condition: 
\begin{equation*}
\sup_{\varepsilon \in E}\mathbb{E}\left\Vert u_{\varepsilon }\right\Vert
_{L^{p}(Q_{T})}^{p}<\infty .
\end{equation*}%
Then there exists a subsequence $E^{\prime }$ from $E$ such that the
sequence $(u_{\varepsilon })_{\varepsilon \in E^{\prime }}$ is weakly $%
\Sigma $-convergent in $L^{p}(Q_{T}\times \Omega )$.
\end{theorem}

\begin{theorem}
\label{t3.8}Let $1<p<\infty $. Let $(u_{\varepsilon })_{\varepsilon \in
E}\subset L^{p}(\Omega ;L^{p}(0,T;W^{1,p}(Q)))$ be a sequence of random
variables which satisfies the following estimate: 
\begin{equation*}
\sup_{\varepsilon \in E}\mathbb{E}\left\Vert u_{\varepsilon }\right\Vert
_{L^{p}(0,T;W^{1,p}(Q))}^{p}<\infty .
\end{equation*}%
Then there exist a subsequence $E^{\prime }$ of $E$ and a couple of random
variables $(u_{0},u_{1})$ with $u_{0}\in L^{p}(\Omega
;L^{p}(0,T;W^{1,p}(Q;I_{A}^{p})))$ and $u_{1}\in L^{p}(\Omega ;L^{p}(Q_{T};%
\mathcal{B}_{A_{\tau }}^{p}(\mathbb{R}_{\tau };\mathcal{B}_{\#A_{y}}^{1,p})))
$ such that, as $E^{\prime }\ni \varepsilon \rightarrow 0$, 
\begin{equation*}
u_{\varepsilon }\rightarrow u_{0}\ \text{in }L^{p}(Q_{T}\times \Omega )\text{%
-weak }\Sigma 
\end{equation*}%
and%
\begin{equation*}
\frac{\partial u_{\varepsilon }}{\partial x_{i}}\rightarrow \frac{\partial
u_{0}}{\partial x_{i}}+\frac{\overline{\partial }u_{1}}{\partial y_{i}}\text{%
\ in }L^{p}(Q_{T}\times \Omega )\text{-weak }\Sigma \text{, }1\leq i\leq N%
\text{.}
\end{equation*}
\end{theorem}

Theorem \ref{t3.8} will be very useful in the last section of this work.

\section{Young measures generated by an algebra with mean value}

In this section we assume that the algebra $A$ is separable. This assumption
is not fundamental, but it is made just to simplify the presentation of the
foregoing section.

Let $E_{p}$ ($1\leq p<\infty $) denote the space of continuous functions $%
\Phi :\overline{Q}\times \mathbb{R}^{N}\times \mathbb{R}^{m}\rightarrow 
\mathbb{R}$ satisfying the following conditions:

\begin{itemize}
\item[(C1)] $\Phi \in \mathcal{C}(\overline{Q}\times \mathbb{R}^{m};A)$

\item[(C2)] $\lim_{\left\vert \lambda \right\vert \rightarrow \infty }\frac{%
\Phi (x,y,\lambda )}{1+\left\vert \lambda \right\vert ^{p}}$ exists
uniformly in $(x,y)\in \overline{Q}\times \mathbb{R}^{N}$.
\end{itemize}

Let $K=\widehat{\mathbb{R}^{m}}$ be the Alexandroff one point
compactification of $\mathbb{R}^{m}$: $\widehat{\mathbb{R}^{m}}=\mathbb{R}%
^{m}\cup \{\infty \}$. Each element $\Phi $ of $E_{p}$ extends to a unique
element $\Psi $ of $\mathcal{C}(\overline{Q}\times K;A)$ as follows: 
\begin{equation*}
\Psi (x,y,\lambda )=\left\{ 
\begin{array}{l}
\frac{\Phi (x,y,\lambda )}{1+\left| \lambda \right| ^{p}}\text{\ if }%
(x,y,\lambda )\in \overline{Q}\times \mathbb{R}^{N}\times \mathbb{R}^{m} \\ 
\lim_{\left| \lambda \right| \rightarrow \infty }\frac{\Phi (x,y,\lambda )}{%
1+\left| \lambda \right| ^{p}}\text{\ if }\lambda =\infty .%
\end{array}%
\right.
\end{equation*}%
Besides $\Psi $ verifies the property that there exists $c>0$ (depending on $%
\Phi $) such that 
\begin{equation*}
\left| \Psi (x,y,\lambda )\right| \leq c(1+\left| \lambda \right| ^{p})\ \
\forall (x,y,\lambda )\in \overline{Q}\times \mathbb{R}^{N}\times K.
\end{equation*}%
On the other hand, the Gelfand transformation $\mathcal{G}$ being an
isometric isomorphism of $A$ onto $\mathcal{C}(\Delta (A))$, we construct an
isometric isomorphism of $\mathcal{C}(\overline{Q}\times K;A)$ onto $%
\mathcal{C}(\overline{Q}\times K;\mathcal{C}(\Delta (A)))=\mathcal{C}(%
\overline{Q}\times \Delta (A)\times K)$ (which is separable), so that $E_{p}$
is separable.

Equipped with the norm 
\begin{equation*}
\left\Vert \Phi \right\Vert =\sup_{x\in \overline{Q},y\in \mathbb{R}%
^{N},\lambda \in \mathbb{R}^{m}}\frac{\left\vert \Phi (x,y,\lambda
)\right\vert }{1+\left\vert \lambda \right\vert ^{p}},
\end{equation*}%
$E_{p}$ is a Banach space. We denote by $\mathcal{P}(\mathbb{R}^{m})$ the
space of probability measures on $\mathbb{R}^{m}$. This being so, we have
the following result.

\begin{theorem}
\label{t4.1}Let $Q$ be an open bounded subset of $\mathbb{R}^{N}$. Let $%
1\leq p<\infty $, and let $A$ be an algebra wmv on $\mathbb{R}_{y}^{N}$.
Finally let $(u_{\varepsilon })_{\varepsilon \in E}$ ($E$ being a
fundamental sequence) be a bounded sequence in $L^{p}(Q;\mathbb{R}^{m})$.
There exist a subsequence $E^{\prime }$ from $E$ and a family $\nu =(\nu
_{x,s})_{x\in Q,s\in \Delta (A)}\in L^{\infty }(Q\times \Delta (A);\mathcal{P%
}(\mathbb{R}^{m}))$ such that, as $E^{\prime }\ni \varepsilon \rightarrow 0$%
, 
\begin{equation}
\int_{Q}\Phi \left( x,\frac{x}{\varepsilon _{1}},u_{\varepsilon }(x)\right)
dx\rightarrow \int_{Q}\int_{\Delta (A)}\int_{\mathbb{R}^{m}}\widehat{\Phi }%
(x,s,\lambda )d\nu _{x,s}(\lambda )d\beta (s)dx  \label{4.1}
\end{equation}%
for all $\Phi \in E_{p}$.
\end{theorem}

\begin{proof}
For fixed $\Psi _{0}\in \mathcal{C}(\overline{Q}\times K;A)$, let us define $%
\mu _{\varepsilon }$ ($\varepsilon \in E$) as follows: 
\begin{equation*}
\left\langle \mu _{\varepsilon },\widehat{\Psi }_{0}\right\rangle
=\int_{Q}\Psi _{0}\left( x,\frac{x}{\varepsilon _{1}},u_{\varepsilon
}(x)\right) dx.
\end{equation*}%
We have 
\begin{equation*}
\left\vert \left\langle \mu _{\varepsilon },\widehat{\Psi }_{0}\right\rangle
\right\vert \leq \int_{Q}\sup_{y,\lambda }\left\vert \Psi _{0}\left(
x,y,\lambda \right) \right\vert dx.
\end{equation*}%
$\mathcal{G}$ being an isometric isomorphism of $A$ onto $\mathcal{C}(\Delta
(A))$ and as $\Psi _{0}\left( x,\cdot ,\lambda \right) \in A$, we have $%
\sup_{y\in \mathbb{R}^{N}}\left\vert \Psi _{0}\left( x,y,\lambda \right)
\right\vert =\sup_{s\in \Delta (A)}\left\vert \widehat{\Psi }_{0}\left(
x,s,\lambda \right) \right\vert $, hence 
\begin{equation*}
\left\vert \left\langle \mu _{\varepsilon },\widehat{\Psi }_{0}\right\rangle
\right\vert \leq \int_{Q}\sup_{s,\lambda }\left\vert \widehat{\Psi }%
_{0}\left( x,s,\lambda \right) \right\vert dx=\left\Vert \widehat{\Psi }%
_{0}\right\Vert _{L^{1}(Q;\mathcal{C}(\Delta (A)\times K))}.
\end{equation*}%
Thus $\mu _{\varepsilon }$ (continuous linear functional on the subspace $\{t%
\widehat{\Psi }_{0}:t\in \mathbb{R}\}$ of $L^{1}(Q;\mathcal{C}(\Delta
(A)\times K))$) extends (in a non unique way) to a continuous linear form on 
$L^{1}(Q;\mathcal{C}(\Delta (A)\times K))$, denoted by $\widetilde{\mu }%
_{\varepsilon }$, and satisfying 
\begin{equation*}
\left\Vert \widetilde{\mu }_{\varepsilon }\right\Vert _{L^{\infty }(Q;%
\mathcal{M}(\Delta (A)\times K))}\leq 1\ \ (\varepsilon \in E)
\end{equation*}%
where $\mathcal{M}(\Delta (A)\times K)$ (the dual space of $\mathcal{C}%
(\Delta (A)\times K)$) is the space of Radon measures defined on the compact
space $\Delta (A)\times K$. Because of the Banach-Alaoglu theorem, there
exist a subsequence $E^{\prime }(\Psi _{0})$ of $E$ and some $\mu \in
L^{\infty }(Q;\mathcal{M}(\Delta (A)\times K))$ such that, as $E^{\prime
}(\Psi _{0})\ni \varepsilon \rightarrow 0$, 
\begin{equation*}
\widetilde{\mu }_{\varepsilon }\rightarrow \mu \text{ in }L^{\infty }(Q;%
\mathcal{M}(\Delta (A)\times K))\text{-weak}\ast .
\end{equation*}%
In particular we have, as $E^{\prime }(\Psi _{0})\ni \varepsilon \rightarrow
0$, 
\begin{equation}
\int_{Q}\Psi _{0}\left( x,\frac{x}{\varepsilon _{1}},u_{\varepsilon
}(x)\right) dx\rightarrow \int_{Q}\int_{\Delta (A)\times K}\widehat{\Psi }%
_{0}(x,s,\lambda )d\mu _{x}(s,\lambda )dx.  \label{4.2}
\end{equation}%
$E_{p}$ being separable, let $\{\Psi _{k}:k\in \mathbb{N}\}$ be a countable
dense subset of $E_{p}$. For simplification we put $E=(\varepsilon
_{n})_{n\in \mathbb{N}}$. We have in hand a family $\{E^{\prime }(\Psi
_{k}):k\in \mathbb{N}\}$ of subsequences of $E$ obtained by repeating the
argument used to get (\ref{4.2}) and satisfying the following relation: $%
E^{\prime }(\Psi _{k+1})\subset E^{\prime }(\Psi _{k})$ for each $k\in 
\mathbb{N}$. By the well-known diagonal process we construct a subsequence $%
E^{\prime }$ from the family $\{E^{\prime }(\Psi _{k}):k\in \mathbb{N}\}$
satisfying, as $E^{\prime }\ni \varepsilon \rightarrow 0$, 
\begin{equation}
\int_{Q}\Psi _{k}\left( x,\frac{x}{\varepsilon _{1}},u_{\varepsilon
}(x)\right) dx\rightarrow \int_{Q}\int_{\Delta (A)\times K}\widehat{\Psi }%
_{k}(x,s,\lambda )d\mu _{x}(s,\lambda )dx\ \forall k\in \mathbb{N}.
\label{4.3}
\end{equation}%
By a mere routine we get (\ref{4.3}) by replacing $\Psi _{k}$ with any $\Psi 
$ in $E_{p}$. It is evident that, for a.e. $x\in Q$, $\mu _{x}$ is a
probability measure: in fact, taking in (\ref{4.3}) $\Psi \equiv 1$, we are
led (by the uniqueness of the limit) to $\int_{\Delta (A)\times K}d\mu
_{x}(s,\lambda )=1$.

Next, using the same argument as the one used in the proof of \cite[Theorem 7%
]{Valadier1} we see that the boundedness of $(u_{\varepsilon })_{\varepsilon
\in E}$ in $L^{p}(Q;\mathbb{R}^{m})$ implies that $\mu $ (thus constructed)
is supported by $Q\times \Delta (A)\times \mathbb{R}^{m}$, so that 
\begin{equation*}
\int_{Q}\int_{\Delta (A)\times K}\widehat{\Psi }(x,s,\lambda )d\mu
_{x}(s,\lambda )dx=\int_{Q}\int_{\Delta (A)\times \mathbb{R}^{m}}\widehat{%
\Psi }(x,s,\lambda )d\mu _{x}(s,\lambda )dx
\end{equation*}%
for all $\Psi \in \mathcal{C}(\overline{Q};\mathcal{C}(K;A))$. Thus, by \cite%
[Theorem 3]{Valadier1}, $\mu $ is the weak $\ast $-limit of $(\widetilde{\mu 
}_{\varepsilon })_{\varepsilon \in E^{\prime }}$ in $L^{\infty }(Q;\mathcal{M%
}(\Delta (A)\times \mathbb{R}^{m}))$ and thereby, defined a family of
probability measures $(\mu _{x})_{x\in Q}$ with $\mu _{x}\in \mathcal{M}%
(\Delta (A)\times \mathbb{R}^{m})$. Let $\nu _{x}$ denote the projection of $%
\mu _{x}$ onto $\Delta (A)$. Let us show that $\nu _{x}=\beta $. For that,
let $p_{1}$ denote the projection of $\Delta (A)\times \mathbb{R}^{m}$ onto $%
\Delta (A)$: $p_{1}(s,\lambda )=s$, $(s,\lambda )\in \Delta (A)\times 
\mathbb{R}^{m}$. We have the obvious equality 
\begin{equation}
\int_{\Delta (A)\times \mathbb{R}^{m}}(g\circ p_{1})(s,\lambda )d\mu
_{x}(s,\lambda )=\int_{\Delta (A)}g(s)d\nu _{x}(s)\ \ (g\in \mathcal{C}%
(\Delta (A))).  \label{4.5}
\end{equation}%
This being so, let $h\in A$, and let $\varphi \in \mathcal{K}(Q)$ (the space
continuous functions on $\mathbb{R}^{N}$ with compact support contained in $%
Q $). Set 
\begin{equation*}
\Phi (x,y,\lambda )=\varphi (x)h(y)\ \ (x\in Q,y\in \mathbb{R}^{N},\lambda
\in \mathbb{R}^{m}).
\end{equation*}%
Then $\Phi \in \mathcal{C}(\overline{Q};\mathcal{C}(K;A))$ and so, as $%
E^{\prime }\ni \varepsilon \rightarrow 0$, 
\begin{equation*}
\int_{Q}\varphi (x)h\left( \frac{x}{\varepsilon _{1}}\right) dx\rightarrow
\int_{Q}\int_{\Delta (A)\times \mathbb{R}^{m}}\varphi (x)\widehat{h}(s)d\mu
_{x}(s,\lambda ).
\end{equation*}%
On the other hand, we have, as $E^{\prime }\ni \varepsilon \rightarrow 0$,%
\begin{equation*}
\int_{Q}\varphi (x)h\left( \frac{x}{\varepsilon _{1}}\right) dx\rightarrow
\int_{Q}\int_{\Delta (A)}\varphi (x)\widehat{h}(s)d\beta (s)dx.
\end{equation*}%
We deduce that 
\begin{equation*}
\int_{\Delta (A)\times \mathbb{R}^{m}}\widehat{h}(s)d\mu _{x}(s,\lambda
)=\int_{\Delta (A)}\widehat{h}(s)d\beta (s)\text{ \ a.e. }x\in Q,
\end{equation*}%
or, taking into account (\ref{4.5}), 
\begin{equation*}
\int_{\Delta (A)}\widehat{h}(s)d\nu _{x}(s)=\int_{\Delta (A)}\widehat{h}%
(s)d\beta (s).
\end{equation*}%
Since the above equality holds for every $h\in A$, we deduce that $\nu
_{x}=\beta $ a.e. $x\in Q$ (hence $\nu _{x}$ is homogeneous, i.e. is
independent of $x$). Thus, using the Valadier's result on disintegration of
measures \cite[Theorem 2]{Valadier2}, there exists a probability measure $%
\nu _{x,s}$ ($s\in \Delta (A)$) on $\mathbb{R}^{m}$ such that 
\begin{equation*}
\mu _{x}=\nu _{x,s}\otimes \beta .
\end{equation*}%
We are therefore led to (\ref{4.1}) for all $\Phi \in E_{p}$. This completes
the proof of the theorem.
\end{proof}

Theorem \ref{t4.1} yields the following

\begin{definition}
\label{d4.1}\emph{The family of probability measures }$\{\nu _{x,s}\}_{x\in
Q,s\in \Delta (A)}$\emph{\ is called the }Young measure associated with $%
(u_{\varepsilon })_{\varepsilon \in E}$ at length scale $\varepsilon _{1}$.
\end{definition}

The concept of Young measures is weaker than the one of weak $\Sigma $-limit
as shown by the following result.

\begin{corollary}
\label{c4.1}Let $1<p<\infty $. The function $u_{0}\in L^{p}(Q;(\mathcal{B}%
_{A}^{p})^{m})$ defined by 
\begin{equation*}
\mathcal{G}_{1}(u_{0})(x,s)=\int_{\mathbb{R}^{m}}\lambda d\nu _{x,s}(\lambda
)\ \ \ ((x,s)\in Q\times \Delta (A))
\end{equation*}%
is the weak $\Sigma $-limit of $(u_{\varepsilon })_{\varepsilon \in
E^{\prime }}$.
\end{corollary}

\begin{proof}
Let $g\in \mathcal{K}(Q;A)$. Set 
\begin{equation*}
\Phi ^{i}(x,y,\lambda )=g(x,y)\lambda _{i}\ \ ((x,y,\lambda )\in Q\times 
\mathbb{R}^{N}\times \mathbb{R}^{m})\ (1\leq i\leq m),
\end{equation*}%
where $\lambda =(\lambda _{i})_{1\leq i\leq m}$. Then $\Phi ^{i}$ is
continuous on $Q\times \mathbb{R}^{N}\times \mathbb{R}^{m}$ (so is of Carath%
\'{e}odory's type on $Q\times \mathbb{R}^{N}\times \mathbb{R}^{m}$).
Besides, as $\left\vert \lambda _{i}\right\vert \leq 1+\left\vert \lambda
\right\vert ^{p}$ for all $\lambda \in \mathbb{R}^{m}$, we have 
\begin{equation*}
\left\vert \Phi ^{i}\left( x,\frac{x}{\varepsilon _{1}},u_{\varepsilon
}(x)\right) \right\vert \leq c(1+\left\vert u_{\varepsilon }(x)\right\vert
^{p})
\end{equation*}%
where $c=\sup_{x\in Q,y\in \mathbb{R}^{N}}\left\vert g(x,y)\right\vert
<\infty $. The sequence $(\Phi ^{i}(\cdot ,\cdot /\varepsilon
_{1},u_{\varepsilon }))_{\varepsilon \in E^{\prime }}$ is therefore
uniformly integrable since $p>1$. We deduce from \cite[Theorem 17]{Valadier3}
that, as $E^{\prime }\ni \varepsilon \rightarrow 0$, 
\begin{equation*}
\int_{Q}g\left( x,\frac{x}{\varepsilon _{1}}\right) u_{\varepsilon
}^{i}(x)dx\rightarrow \int_{Q}\int_{\Delta (A)}\int_{\mathbb{R}^{m}}\widehat{%
g}(x,s)\lambda _{i}d\nu _{x,s}(\lambda )d\beta (s)dx
\end{equation*}%
where $u_{\varepsilon }^{i}$ is the $i$th component of $u_{\varepsilon }$, $%
1\leq i\leq m$. But by the definition of the weak $\Sigma $-limit, if $%
u_{0}^{i}\in L^{p}(Q;\mathcal{B}_{A}^{p})$ is the weak $\Sigma $-limit of $%
(u_{\varepsilon }^{i})_{\varepsilon \in E^{\prime }}$, then, as $E^{\prime
}\ni \varepsilon \rightarrow 0$, 
\begin{equation*}
\int_{Q}g\left( x,\frac{x}{\varepsilon _{1}}\right) u_{\varepsilon
}^{i}(x)dx\rightarrow \int_{Q}\int_{\Delta (A)}\widehat{g}(x,s)\widehat{u}%
_{0}^{i}(x,s)d\beta (s)dx,
\end{equation*}%
hence, comparing the above convergence results, we deduce that 
\begin{equation*}
\widehat{u}_{0}^{i}(x,s)=\int_{\mathbb{R}^{m}}\lambda _{i}d\nu
_{x,s}(\lambda ),
\end{equation*}%
which completes the proof.
\end{proof}

\begin{remark}
\label{r4.1}\emph{Due to the preceding corollary, it is now clear that
Theorem \ref{t3.7} generalizes Theorem \ref{t3.1} of the preceding section.}
\end{remark}

The next result is very useful in practise.

\begin{theorem}
\label{t4.2}Let $(u_{\varepsilon })_{\varepsilon \in E}$ be a bounded
sequence in $L^{p}(Q;\mathbb{R}^{m})$ $(1\leq p<\infty )$ with associated
Young measure $(\nu _{x,s})_{x\in Q,s\in \Delta (A)}$. The following
properties hold:

\begin{itemize}
\item[(i)] Let $\Phi :Q\times \mathbb{R}^{N}\times \mathbb{R}^{m}\rightarrow
\lbrack 0,+\infty )$ be a Carath\'{e}odory integrand, i.e., $\Phi (\cdot
,y,\cdot )$ is continuous for all $y\in \mathbb{R}^{N}$ and $\Phi (x,\cdot
,\lambda )$ is measurable for all $(x,\lambda )\in Q\times \mathbb{R}^{m}$.
Assume further that 
\begin{equation*}
\Phi (x,\cdot ,\lambda )\in B_{A}^{1}\ \ \forall (x,\lambda )\in Q\times 
\mathbb{R}^{m}.
\end{equation*}%
Then 
\begin{equation*}
\int_{Q}\int_{\Delta (A)}\int_{\mathbb{R}^{m}}\widehat{\Phi }(x,s,\lambda
)d\nu _{x,s}(\lambda )d\beta (s)dx\leq ~\underset{E\ni \varepsilon
\rightarrow 0}{\lim \inf }\int_{Q}\Phi \left( x,\frac{x}{\varepsilon _{1}}%
,u_{\varepsilon }(x)\right) dx.
\end{equation*}

\item[(ii)] If in addition to \emph{(i)}, the sequence $(\Phi (\cdot ,\cdot
/\varepsilon _{1},u_{\varepsilon }))_{\varepsilon \in E}$ is uniformly
integrable, then $\widehat{\Phi }(x,s,\cdot )$ is $\nu _{x,s}$-integrable
for a.e. $(x,s)\in Q\times \Delta (A)$. Besides there exists $\chi \in
L^{1}(Q;\mathcal{B}_{A}^{1})$ such that 
\begin{equation}
\mathcal{G}_{1}(\chi )(x,s)=\int_{\mathbb{R}^{m}}\widehat{\Phi }(x,s,\lambda
)d\nu _{x,s}(\lambda )\text{\ \ a.e. }(x,s)\in Q\times \Delta (A)
\label{4.6}
\end{equation}%
and 
\begin{equation}
\Phi (\cdot ,\cdot /\varepsilon _{1},u_{\varepsilon })\rightarrow \chi \text{
in }L^{1}(Q)\text{-weak }\Sigma .  \label{4.7}
\end{equation}

\item[(iii)] The barycenter $(x,s)\mapsto \int_{\mathbb{R}^{m}}\lambda d\nu
_{x,s}(\lambda )$ belongs to $L^{p}(Q\times \Delta (A);\mathbb{R}^{m})$.
\end{itemize}
\end{theorem}

\begin{proof}
Thanks to the density of $A$ in $B_{A}^{1}$, (i) is a direct consequence of 
\cite[Theorem 16]{Valadier3}. As for (ii), the integrability of $\widehat{%
\Phi }(x,s,\cdot )$ is a consequence of \cite[Theorem 17]{Valadier3} (see
also \cite{Valadier1, Ball, Balder}). Let us check (\ref{4.6}) and (\ref{4.7}%
). Let $g\in \mathcal{K}(Q;A)$; define $\Psi :Q\times \mathbb{R}^{N}\times 
\mathbb{R}^{m}\rightarrow \mathbb{R}$ by $\Psi (x,y,\lambda )=g(x,y)\Phi
(x,y,\lambda )$. Then $\Psi $ is of Carath\'{e}odory's type on $Q\times 
\mathbb{R}^{N}\times \mathbb{R}^{m}$ and $\Psi (\cdot ,\cdot /\varepsilon
_{1},u_{\varepsilon })$ is uniformly integrable. In view of \cite[Theorem 17]%
{Valadier3} we have, as $E\ni \varepsilon \rightarrow 0$, 
\begin{eqnarray*}
\int_{Q}\Psi \left( x,\frac{x}{\varepsilon _{1}},u_{\varepsilon }(x)\right)
dx &\rightarrow &\int_{Q}\int_{\Delta (A)}\int_{\mathbb{R}^{m}}\widehat{\Psi 
}(x,s,\lambda )d\nu _{x,s}(\lambda )d\beta (s)dx \\
&=&\int_{Q}\int_{\Delta (A)}\int_{\mathbb{R}^{m}}\widehat{g}(x,s)\widehat{%
\Phi }(x,s,\lambda )d\nu _{x,s}(\lambda )d\beta (s)dx.
\end{eqnarray*}%
But, $(\Phi (\cdot ,\cdot /\varepsilon _{1},u_{\varepsilon }))_{\varepsilon
\in E}$ is uniformly integrable, so because of the Theorem \ref{t3.2}, there
exist a subsequence $E^{\prime }$ from $E$ and a function $\chi \in L^{1}(Q;%
\mathcal{B}_{A}^{1})$ such that 
\begin{equation*}
\Phi \left( \cdot ,\frac{\cdot }{\varepsilon _{1}},u_{\varepsilon }\right)
\rightarrow \chi \text{ in }L^{1}(Q)\text{- weak }\Sigma \text{ as }%
E^{\prime }\ni \varepsilon \rightarrow 0.
\end{equation*}%
Thus, for $g$, we have, when $E^{\prime }\ni \varepsilon \rightarrow 0$, 
\begin{equation*}
\int_{Q}g\left( x,\frac{x}{\varepsilon _{1}}\right) \Phi \left( x,\frac{x}{%
\varepsilon _{1}},u_{\varepsilon }(x)\right) dx\rightarrow
\int_{Q}\int_{\Delta (A)}\widehat{g}(x,s)\widehat{\chi }(x,s)d\beta (s)dx,
\end{equation*}%
hence, comparing the above convergence results yields 
\begin{equation*}
\widehat{\chi }(x,s)=\int_{\mathbb{R}^{m}}\widehat{\Phi }(x,s,\lambda )d\nu
_{x,s}(\lambda ),
\end{equation*}%
form which, (\ref{4.6}) and (\ref{4.7}). Finally, let us check (iii). If in
(ii) we choose $\Phi (x,y,\lambda )=\left\vert \lambda \right\vert ^{p}$
then by using Jensen's inequality we are led to 
\begin{eqnarray*}
\int_{Q}\int_{\Delta (A)}\left\vert \int_{\mathbb{R}^{m}}\lambda d\nu
_{x,s}(\lambda )\right\vert ^{p}d\beta dx &\leq &\int_{Q}\int_{\Delta
(A)}\int_{\mathbb{R}^{m}}\left\vert \lambda \right\vert ^{p}d\nu
_{x,s}(\lambda )d\beta (s)dx \\
&\leq &\underset{E\ni \varepsilon \rightarrow 0}{\lim \inf }%
\int_{Q}\left\vert u_{\varepsilon }(x)\right\vert ^{p}dx<+\infty ,
\end{eqnarray*}%
hence (iii).
\end{proof}

The following result characterizes the strong convergence.

\begin{proposition}
\label{p4.2}Let $(u_{\varepsilon })_{\varepsilon \in E}$ be a bounded
sequence in $L^{p}(Q;\mathbb{R}^{m})$ $(1\leq p<\infty )$, and let $(\nu
_{x,s})_{x\in Q,s\in \Delta (A)}$ be the associated Young measure. Assume
that either $v\in L^{p}(Q;(A)^{m})$ or $v\in \mathcal{C}(\overline{Q}%
;(B_{A}^{p})^{m})$. Then 
\begin{equation*}
\nu _{x,s}=\delta _{\widehat{v}(x,s)}\text{ if and only if }\lim_{E\ni
\varepsilon \rightarrow 0}\left\Vert u_{\varepsilon }-v^{\varepsilon
}\right\Vert _{L^{1}(Q)^{m}}=0\text{,}
\end{equation*}%
where $v^{\varepsilon }(x)=v(x,x/\varepsilon _{1})$, $x\in Q$.
\end{proposition}

\begin{proof}
Assume that $\nu _{x,s}=\delta _{\widehat{v}(x,s)}$ where $v$ is either in $%
L^{p}(Q;(A)^{m})$ or in $\mathcal{C}(\overline{Q};(B_{A}^{p})^{m})$. Set $%
\Phi (x,y,\lambda )=\left\vert \lambda -v(x,y)\right\vert $; then $\Phi $ is
a Carath\'{e}odory function and further, $\Phi \geq 0$. Moreover the
sequence $(\Phi (\cdot ,\cdot /\varepsilon _{1},u_{\varepsilon
}))_{\varepsilon \in E}$ is uniformly integrable since by setting $\varphi
(t)=t^{p}$ ($t\geq 0$), $\varphi $ is inf-compact and 
\begin{eqnarray*}
\int_{Q}\varphi \left( \left\vert \Phi \left( \cdot ,\frac{\cdot }{%
\varepsilon _{1}},u_{\varepsilon }\right) \right\vert \right) dx
&=&\int_{Q}\left\vert u_{\varepsilon }(x)-v\left( x,\frac{x}{\varepsilon _{1}%
}\right) \right\vert ^{p}dx \\
&\leq &2^{p}\left( \int_{Q}\left\vert u_{\varepsilon }(x)\right\vert
^{p}dx+\int_{Q}\left\vert v\left( x,\frac{x}{\varepsilon _{1}}\right)
\right\vert ^{p}dx\right) \\
&\leq &M
\end{eqnarray*}%
where $M$ is a positive constant independent of $\varepsilon $. Applying
[part (ii) of] Theorem \ref{t4.2}, we get, as $E\ni \varepsilon \rightarrow
0 $, 
\begin{equation*}
\int_{Q}\Phi \left( x,\frac{x}{\varepsilon _{1}},u_{\varepsilon }(x)\right)
dx\rightarrow \int_{Q}\int_{\Delta (A)}\left\langle \nu _{x,s},\widehat{\Phi 
}(x,s,\cdot )\right\rangle d\beta (s)dx.
\end{equation*}%
But 
\begin{eqnarray*}
\left\langle \nu _{x,s},\widehat{\Phi }(x,s,\cdot )\right\rangle
&=&\left\langle \delta _{\widehat{v}(x,s)},\widehat{\Phi }(x,s,\cdot
)\right\rangle =\widehat{\Phi }(x,s,\widehat{v}(x,s)) \\
&=&\left\vert \widehat{v}(x,s)-\widehat{v}(x,s)\right\vert =0\text{ a.e. }%
(x,s)\in Q\times \Delta (A).
\end{eqnarray*}%
On the other hand 
\begin{equation*}
\int_{Q}\Phi \left( x,\frac{x}{\varepsilon _{1}},u_{\varepsilon }(x)\right)
dx=\int_{Q}\left\vert u_{\varepsilon }(x)-v\left( x,\frac{x}{\varepsilon _{1}%
}\right) \right\vert dx=\left\Vert u_{\varepsilon }-v^{\varepsilon
}\right\Vert _{L^{1}(Q)}.
\end{equation*}%
Now assume that $\left\Vert u_{\varepsilon }-v^{\varepsilon }\right\Vert
_{L^{1}(Q)}\rightarrow 0$ as $E\ni \varepsilon \rightarrow 0$. Let $g\in 
\mathcal{K}(Q;A)$; Applying once more [part (ii) of] Theorem \ref{t4.2} with 
$\Phi (x,y,\lambda )=g(x,y)\left\vert \lambda -v(x,y)\right\vert $ we get,
when $E\ni \varepsilon \rightarrow 0$ 
\begin{equation*}
\int_{Q}\Phi \left( x,\frac{x}{\varepsilon _{1}},u_{\varepsilon }(x)\right)
dx\rightarrow \int_{Q}\int_{\Delta (A)}\int_{\mathbb{R}^{m}}\widehat{g}%
(x,s)\left\vert \lambda -\widehat{v}(x,s)\right\vert d\nu _{x,s}(\lambda
)d\beta (s)dx\text{.}
\end{equation*}%
But 
\begin{equation*}
\left\vert \int_{Q}\Phi \left( x,\frac{x}{\varepsilon _{1}},u_{\varepsilon
}(x)\right) dx\right\vert \leq \left\Vert g\right\Vert _{\infty }\left\Vert
u_{\varepsilon }-v^{\varepsilon }\right\Vert _{L^{1}(Q)}\rightarrow 0\text{
when }E\ni \varepsilon \rightarrow 0.
\end{equation*}%
Thus 
\begin{equation*}
\int_{Q}\int_{\Delta (A)}\widehat{g}(x,s)\left\langle \nu _{x,s},\left\vert
\lambda -\widehat{v}(x,s)\right\vert \right\rangle d\beta dx=0\ \forall g\in 
\mathcal{K}(Q;A),
\end{equation*}%
hence $\left\langle \nu _{x,s},\left\vert \lambda -\widehat{v}%
(x,s)\right\vert \right\rangle =0$ a.e. $(x,s)\in Q\times \Delta (A)$, which
leads to $\nu _{x,s}=\delta _{\widehat{v}(x,s)}$.
\end{proof}

\section{Homogenization of a convex integral functional}

All function spaces and scalar functions are real-valued in this section.

\subsection{Setting of the problem}

Our main concern here is the study of the asymptotic behavior (as $%
\varepsilon \rightarrow 0$) of the sequence of solutions of the problems 
\begin{equation*}
\min \left\{ F_{\varepsilon }(v):v\in W_{0}^{1,p}(Q;\mathbb{R}^{n})\right\}
\;\;\;\;\;\;\;\;\;\;\;\;\;\;\;\;\;\;\;\;\;\;\;\;
\end{equation*}%
where the functional $F_{\varepsilon }$ is defined on $L^{p}(Q;\mathbb{R}%
^{n})$ by 
\begin{equation}
F_{\varepsilon }(v)=\left\{ 
\begin{array}{l}
\int_{Q}f\left( x,\frac{x}{\varepsilon },Dv(x)\right) dx,\ v\in
W_{0}^{1,p}(Q;\mathbb{R}^{n})\;\;\;\;\;\;\;\;\;\; \\ 
+\infty \ \ \text{elsewhere,}%
\end{array}%
\right.  \label{5.1}
\end{equation}%
$Q$ being a bounded open set in $\mathbb{R}^{N}$ and $f:\overline{Q}\times 
\mathbb{R}^{N}\times \mathbb{R}^{nN}\rightarrow \lbrack 0,+\infty )$ a Carath%
\'{e}odory function (i.e., $f(x,\cdot ,\lambda )$ is measurable and $f(\cdot
,y,\cdot )$ is continuous) satisfying the following conditions:

\begin{itemize}
\item[(H$_{1}$)] $f(x,y,\cdot )$ is strictly convex for almost all $y\in 
\mathbb{R}^{N}$ and for all $x\in \overline{Q}$,

\item[(H$_{2}$)] There exist three constants $p>1$ and $c_{1}$, $c_{2}>0$
such that 
\begin{equation}
c_{1}\left\vert \lambda \right\vert ^{p}\leq f(x,y,\lambda )\leq
c_{2}(1+\left\vert \lambda \right\vert ^{p})\;\;\;\;\;\;\;\;\;\;\;\;\;
\label{5.2}
\end{equation}%
for all $(x,\lambda )\in \mathbb{R}^{N}\times \mathbb{R}^{nN}$ and for
almost all $y\in \mathbb{R}^{N}$.
\end{itemize}

From the above hypotheses, for any fixed $\varepsilon >0$ and for $v\in
L^{p}(Q;\mathbb{R}^{nN})$, the function $x\mapsto f\left( x,x/\varepsilon
,v(x)\right) $ of $Q$ into $\mathbb{R}_{+}$ (denoted by $f^{\varepsilon
}(\cdot ,\cdot ,v)$), is well defined and lies in $L^{1}(Q)$, with 
\begin{equation*}
c_{1}\left\Vert v\right\Vert _{L^{p}(Q)^{nN}}^{p}\leq \left\Vert
f^{\varepsilon }(\cdot ,\cdot ,v)\right\Vert _{L^{1}(Q)}\leq c_{2}^{\prime
}\left( 1+\left\Vert v\right\Vert _{L^{p}(Q)^{nN}}^{p}\right)
\;\;\;\;\;\;\;\;\;\;\;\;
\end{equation*}%
where $c_{2}^{\prime }=c_{2}\max (1,\left\vert Q\right\vert )$ with $%
\left\vert Q\right\vert =\int_{Q}dx$. Hence (\ref{5.1}) makes sense and, by
classical arguments, there exists \cite{CV, 18} (for each fixed $\varepsilon
>0$) a unique $u_{\varepsilon }\in W_{0}^{1,p}(Q;\mathbb{R}^{n})$ that
realizes the infimum of $F_{\varepsilon }$ on $L^{p}(Q;\mathbb{R}^{nN})$,
i.e., 
\begin{equation}
F_{\varepsilon }(u_{\varepsilon })=\min_{v\in W_{0}^{1,p}(Q;\mathbb{R}%
^{n})}F_{\varepsilon
}(v).\;\;\;\;\;\;\;\;\;\;\;\;\;\;\;\;\;\;\;\;\;\;\;\;\;\;  \label{5.3}
\end{equation}

Our objective here amounts to find, under the assumption that 
\begin{equation}
f(x,\cdot ,\lambda )\in B_{A}^{1}\text{ for all }x\in \overline{Q},\lambda
\in \mathbb{R}^{nN},  \label{5.4}
\end{equation}%
(where $A$ is an algebra wmv) a homogenized functional $\overline{F}$ such
that the sequence of minimizers $u_{\varepsilon }$ converges to a limit $%
\mathbf{u}$, which is precisely the minimizer of $\overline{F}$.

This issue has already been addressed in many papers (see in particular \cite%
{Baia, Barchiesi, Fonseca, NoDEA}). In \cite{NoDEA} the general
deterministic homogenization of (\ref{5.1}) is addressed, but in separable
ergodic algebras wmv using the $\Sigma $-convergence method. Here no
ergodicity assumption is made on the algebra $A$ and moreover, we use the
Young measures theory to solve the problem. This reduces considerably the
length of this section in contrast to what has been done so far; see e.g. 
\cite{NoDEA}. So we mean here to provide by means of Young measures
generated by an algebra wmv, a full study of the functional $F_{\varepsilon
} $ in the general framework of algebras wmv.

\subsection{Homogenization result}

Let $A$ be an algebra wmv. Using (\ref{5.4}) and some well-known results
(see e.g. \cite[Proposition 2.3]{NoDEA}, \cite[Proposition 3.1]{ACAP}) one
can easily define the function $f(\cdot ,\cdot ,\mathbf{w}):(x,y)\rightarrow
f(x,y,\mathbf{w}(x,y))$ (for $w\in L^{p}(Q;(B_{A}^{p})^{nN})$), as an
element of $L^{1}(Q;B_{A}^{1})$ with 
\begin{equation*}
\widehat{f}(x,s,\widehat{\mathbf{w}}(x,s))=\mathcal{G}(f(x,\cdot ,\mathbf{w}%
(x,\cdot )))(s)\text{ a.e. in }(x,s)\in Q\times \Delta (A)\text{.}
\end{equation*}%
Now, let 
\begin{equation*}
\mathbb{F}_{0}^{1,p}=W_{0}^{1,p}(Q;\mathbb{R}^{n})\times L^{p}(Q;(\mathcal{B}%
_{\#A}^{1,p})^{n}).\;\;\;\;\;\;\;\;\;\;
\end{equation*}
$\mathbb{F}_{0}^{1,p}$ is a Banach space under the norm 
\begin{equation*}
\left\Vert \mathbf{u}\right\Vert _{\mathbb{F}_{0}^{1,p}}=\left( \left\Vert
u_{0}\right\Vert _{W_{0}^{1,p}(Q)^{n}}^{p}+\left\Vert u_{1}\right\Vert
_{L^{p}(Q;(\mathcal{B}_{\#A}^{1,p})^{n})}^{p}\right) ^{\frac{1}{p}}\;\;(%
\mathbf{u}=(u_{0},u_{1})\in \mathbb{F}_{0}^{1,p}),
\end{equation*}%
(where $\left\Vert u_{1}\right\Vert _{L^{p}(Q;(\mathcal{B}%
_{\#A}^{1,p})^{n})}=(\overset{n}{\underset{i=1}{\sum }}\overset{N}{\underset{%
j=1}{\sum }}\left\Vert \overline{\partial }u_{1,i}/\partial y_{j}\right\Vert
_{L^{p}(Q;\mathcal{B}_{A}^{p})}^{p})^{1/p}$\ for $u_{1}=(u_{1,i})_{1\leq
i\leq n}$) admitting $F_{0}^{\infty }=\mathcal{D}(Q)^{n}\times \lbrack 
\mathcal{D}(Q)\otimes (\varrho (A))^{n}]$ as a dense subspace.

We can now state and prove the main result of this subsection.

\begin{theorem}
\label{t5.1}Let $A$ be a separable algebra wmv such that \emph{(\ref{5.4})}
holds. For each $\varepsilon >0$, let $u_{\varepsilon }$ be the unique
solution of \emph{(\ref{5.3})}. Then, as $\varepsilon \rightarrow 0$ we have 
\begin{equation}
u_{\varepsilon }\rightarrow u_{0}\text{ in }W_{0}^{1,p}(Q)^{n}\text{-weak\ \
\ \ \ \ \ \ \ \ \ \ \ \ \ }  \label{5.5}
\end{equation}%
\begin{equation}
\frac{\partial u_{\varepsilon }}{\partial x_{i}}\rightarrow \frac{\partial
u_{0}}{\partial x_{i}}+\frac{\overline{\partial }u_{1}}{\partial y_{i}}\text{
in }L^{p}(Q)^{n}\text{ weak }\Sigma \text{ }(1\leq i\leq N)  \label{5.6}
\end{equation}%
where $\mathbf{u}=(u_{0},u_{1})\in \mathbb{F}_{0}^{1,p}$ is the unique
solution of the variational minimization problem 
\begin{equation}
F(\mathbf{u})=\inf_{\mathbf{v}\in \mathbb{F}_{0}^{1,p}}F(\mathbf{v})\ \ \ \
\ \ \ \ \ \ \ \ \ \ \ \ \ \ \ \ \ \ \ \ \ \ \ \ \ \ \ \ \ \ \ \ \ \ \ \ \ \
\ \ \text{ }  \label{5.7}
\end{equation}%
with the functional $F$ defined on $L^{p}(Q;\mathbb{R}^{n})\times L^{p}(Q;(%
\mathcal{B}_{\#A}^{1,p})^{n})$ by 
\begin{equation*}
F(\mathbf{v})=\left\{ 
\begin{array}{l}
\int_{Q}\int_{\Delta (A)}\widehat{f}(x,s,Dv_{0}+\partial \widehat{v}%
_{1})d\beta dx\text{ for }\mathbf{v}\in \mathbb{F}_{0}^{1,p} \\ 
+\infty \text{\ elsewhere}%
\end{array}%
\right.
\end{equation*}%
and $\partial \widehat{v}_{1}=\widehat{\overline{D}_{y}v_{1}}$.
\end{theorem}

\begin{proof}
First of all we see that Eq. (\ref{5.7}) possesses a unique solution since
the function $\widehat{f}:\overline{Q}\times \Delta (A)\times \mathbb{R}%
^{nN}\rightarrow \lbrack 0,\infty )$, $\widehat{f}(x,s,\lambda )=\mathcal{G}%
(f(x,\cdot ,\lambda ))(s)$ is a Carath\'{e}odory's type function which is
strictly convex in $\lambda $.

Now, in view of the growth condition (\ref{5.2}), the sequence $%
(u_{\varepsilon })_{\varepsilon >0}$ is bounded in $W_{0}^{1,p}(Q)^{n}$ and
so the sequence $f^{\varepsilon }(\cdot ,\cdot ,Du_{\varepsilon })$ is
bounded in $L^{1}(Q)$. Thus given an arbitrary fundamental sequence $E$,
there exist a subsequence $E^{\prime }$ from $E$ and a couple $\mathbf{u}%
=(u_{0},u_{1})\in \mathbb{F}_{0}^{1,p}$ such that (\ref{5.5})-(\ref{5.6})
hold whenever $E^{\prime }\ni \varepsilon \rightarrow 0$. If we show that $%
\mathbf{u}$ solves (\ref{5.7}) then thanks to the uniqueness of the solution
to (\ref{5.7}), the convergence results (\ref{5.5})-(\ref{5.6}) will hold
for $\varepsilon \rightarrow 0$. Thus our only concern here is to check that 
$\mathbf{u}$ solves (\ref{5.7}). To this end, let $(\nu _{x,s})_{x\in Q,s\in
\Delta (A)}$ be the Young measure associated with $(Du_{\varepsilon
})_{\varepsilon \in E^{\prime }}$ at length scale $\varepsilon $. Thanks to
[part (i) of] Theorem \ref{t4.2} we have 
\begin{equation*}
\int_{Q}\int_{\Delta (A)}\int_{\mathbb{R}^{nN}}\widehat{f}(x,s,\lambda )d\nu
_{x,s}(\lambda )d\beta (s)dx\leq \ \underset{E^{\prime }\ni \varepsilon
\rightarrow 0}{\lim \inf }\int_{Q}f\left( x,\frac{x}{\varepsilon }%
,Du_{\varepsilon }\right) dx.
\end{equation*}%
But due to Jensen's inequality one has 
\begin{equation*}
\int_{Q}\int_{\Delta (A)}\int_{\mathbb{R}^{nN}}\widehat{f}(x,s,\lambda )d\nu
_{x,s}(\lambda )d\beta (s)dx\geq \int_{Q}\int_{\Delta (A)}\widehat{f}\left(
x,s,\int_{\mathbb{R}^{nN}}\lambda d\nu _{x,s}(\lambda )\right) d\beta (s)dx,
\end{equation*}%
and by Corollary \ref{c4.1}, 
\begin{equation*}
\int_{Q}\int_{\Delta (A)}\widehat{f}(x,s,Du_{0}+\widehat{\overline{D}%
_{y}u_{1}})d\beta dx\leq \ \underset{E^{\prime }\ni \varepsilon \rightarrow 0%
}{\lim \inf }\int_{Q}f\left( x,\frac{x}{\varepsilon },Du_{\varepsilon
}\right) dx
\end{equation*}%
since $\int_{\mathbb{R}^{nN}}\lambda d\nu _{x,s}(\lambda )=Du_{0}(x)+%
\widehat{\overline{D}_{y}u_{1}(x,\cdot )}(s)=\mathcal{G}_{1}(Du_{0}(x)+%
\overline{D}_{y}u_{1}(x,\cdot ))(s)$. So, let $E_{1}^{\prime }$ be a
subsequence from $E^{\prime }$ such that 
\begin{equation*}
\underset{E^{\prime }\ni \varepsilon \rightarrow 0}{\lim \inf }%
\int_{Q}f\left( x,\frac{x}{\varepsilon },Du_{\varepsilon }\right) dx=%
\underset{E_{1}^{\prime }\ni \varepsilon \rightarrow 0}{\lim }%
\int_{Q}f\left( x,\frac{x}{\varepsilon },Du_{\varepsilon }\right) dx.
\end{equation*}%
We then have 
\begin{equation}
\int_{Q}\int_{\Delta (A)}\widehat{f}(x,s,Du_{0}+\widehat{\overline{D}%
_{y}u_{1}})d\beta dx\leq \ \underset{E_{1}^{\prime }\ni \varepsilon
\rightarrow 0}{\lim }\int_{Q}f\left( x,\frac{x}{\varepsilon }%
,Du_{\varepsilon }\right) dx.  \label{5.8}
\end{equation}%
Let us establish an upper bound for $\lim_{E_{1}^{\prime }\ni \varepsilon
\rightarrow 0}\int_{Q}f\left( x,x/\varepsilon ,Du_{\varepsilon }\right) dx$.
To do that, let $\Phi =(\psi _{0},\varrho ^{n}(\psi _{1}))\in F_{0}^{\infty
} $ with $\psi _{0}\in \mathcal{D}(Q)^{n}$, $\psi _{1}=(\psi _{1,i})_{1\leq
i\leq n}\in \lbrack \mathcal{D}(Q)\otimes (\varrho (A^{\infty }))^{n}]$, $%
\varrho ^{n}(\psi _{1})=(\varrho (\psi _{1,i}))_{1\leq i\leq n}$. We define $%
\Phi _{\varepsilon }$ as follows: $\Phi _{\varepsilon }=\psi
_{0}+\varepsilon \psi _{1}^{\varepsilon }$, that is, $\Phi _{\varepsilon
}(x)=\psi _{0}(x)+\varepsilon \psi _{1}(x,x/\varepsilon )$\ for $x\in Q$.
Then $\Phi _{\varepsilon }\in W_{0}^{1,p}(Q)^{n}$, and, since $%
u_{\varepsilon }$ is the minimizer, one has 
\begin{equation*}
\int_{Q}f\left( x,\frac{x}{\varepsilon },Du_{\varepsilon }\right) dx\leq
\int_{Q}f\left( x,\frac{x}{\varepsilon },D\Phi _{\varepsilon }(x)\right) dx.
\end{equation*}%
Set $v_{\varepsilon }(x)=f\left( x,x/\varepsilon ,Du_{\varepsilon }\right) $
($x\in Q$), $\varepsilon \in E_{1}^{\prime }$. Then $(v_{\varepsilon
})_{\varepsilon \in E_{1}^{\prime }}$ is uniformly integrable; indeed let $%
\varphi (t)=t^{2}$\ ($t\geq 0$); then $\varphi $ is inf-compact, $\varphi
(t)/t\rightarrow +\infty $ as $t\rightarrow +\infty $, and further 
\begin{eqnarray*}
\int_{Q}\varphi (v_{\varepsilon }(x))dx &\leq
&c_{2}^{2}\int_{Q}(1+\left\vert D\Phi _{\varepsilon }\right\vert ^{p})^{2}dx
\\
&\leq &c_{2}^{2}\left\vert Q\right\vert (1+\left\Vert D\psi _{0}\right\Vert
_{\infty }+\left\Vert D\psi _{1}\right\Vert _{\infty }+\left\Vert D_{y}\psi
_{1}\right\Vert _{\infty })^{2p} \\
&<&\infty
\end{eqnarray*}%
where $\left\vert Q\right\vert $ denote the Lebesgue measure of $Q$. The
sequence $(D\Phi _{\varepsilon })_{\varepsilon \in E_{1}^{\prime }}$ being
bounded in $L^{p}(Q;\mathbb{R}^{nN})$ let $(\mu _{x,s})_{x\in Q,s\in \Delta
(A)}$ be the Young measure associated with $(D\Phi _{\varepsilon
})_{\varepsilon \in E_{1}^{\prime }}$ at length scale $\varepsilon $. Since $%
f$ is Carath\'{e}odory and $v_{\varepsilon }=f(\cdot ,\cdot /\varepsilon
,Du_{\varepsilon })$ is uniformly integrable, we deduce by [part (ii) of]
Theorem \ref{t4.2} that, as $E_{1}^{\prime }\ni \varepsilon \rightarrow 0$, 
\begin{equation*}
\int_{Q}f\left( x,\frac{x}{\varepsilon },D\Phi _{\varepsilon }(x)\right)
dx\rightarrow \int_{Q}\int_{\Delta (A)}\int_{\mathbb{R}^{nN}}\widehat{f}%
(x,s,\lambda )d\mu _{x,s}(\lambda )d\beta (s)dx.
\end{equation*}%
But $D\Phi _{\varepsilon }-(D\psi _{0}+(D_{y}\psi _{1})^{\varepsilon
})=\varepsilon (D\psi _{1})^{\varepsilon }$ and so 
\begin{equation*}
\left\Vert D\Phi _{\varepsilon }-(D\psi _{0}+(D_{y}\psi _{1})^{\varepsilon
})\right\Vert _{L^{1}(Q)}\rightarrow 0\text{ as }E_{1}^{\prime }\ni
\varepsilon \rightarrow 0.
\end{equation*}%
This yields by Proposition \ref{p4.2} that $\mu _{x,s}=\delta _{D\psi _{0}+%
\widehat{D_{y}\psi }_{1}}$, in such a way that 
\begin{equation*}
\underset{E_{1}^{\prime }\ni \varepsilon \rightarrow 0}{\lim }%
\int_{Q}f\left( x,\frac{x}{\varepsilon },D\Phi _{\varepsilon }(x)\right)
dx=\int_{Q}\int_{\Delta (A)}\widehat{f}(x,s,D\psi _{0}+\widehat{D_{y}\psi }%
_{1})d\beta dx.
\end{equation*}%
Thus 
\begin{equation*}
\underset{E_{1}^{\prime }\ni \varepsilon \rightarrow 0}{\lim }%
\int_{Q}f\left( x,\frac{x}{\varepsilon },Du_{\varepsilon }(x)\right) dx\leq
\int_{Q}\int_{\Delta (A)}\widehat{f}(x,s,D\psi _{0}+\widehat{D_{y}\psi }%
_{1})d\beta dx
\end{equation*}%
for any $\Phi =(\psi _{0},\varrho ^{n}(\psi _{1}))\in F_{0}^{\infty }$, and
by a density argument, for all $\Phi \in \mathbb{F}_{0}^{1,p}$. Whence 
\begin{equation}
\underset{E_{1}^{\prime }\ni \varepsilon \rightarrow 0}{\lim }%
\int_{Q}f\left( x,\frac{x}{\varepsilon },Du_{\varepsilon }(x)\right) dx\leq
\inf_{\mathbf{v}\in \mathbb{F}_{0}^{1,p}}\int_{Q}\int_{\Delta (A)}\widehat{f}%
(x,s,Dv_{0}+\widehat{\overline{D}_{y}v}_{1})d\beta dx.  \label{5.9}
\end{equation}%
The inequalities (\ref{5.8}) and (\ref{5.9}) yield (\ref{5.7}). This
completes the proof.
\end{proof}

\subsection{Some applications of Theorem \protect\ref{t5.1}\label{subsect5.3}%
}

We can consider the homogenization problem for (\ref{5.3}) under a variety
of assumptions as in the following examples.

\begin{example}[\textit{Homogenization in ergodic algebras}]
\label{e1}\emph{We assume that the algebra }$A$\emph{\ is ergodic. This
allows us to solve the following deterministic homogenization problems:}

\begin{itemize}
\item[(P)$_{1}$] \emph{The function }$f$\emph{\ is periodic in }$y$\emph{;}

\item[(P)$_{2}$] \emph{The function }$f$\emph{\ is almost periodic in }$y$%
\emph{\ \cite{7, 8};}

\item[(P)$_{3}$] 
\begin{equation*}
f(x,\cdot ,\lambda )\in L_{\infty ,AP}^{1}(\mathbb{R}^{N})\text{\ \emph{for
all} }x\in \overline{Q}\text{\ \emph{and all} }\lambda \in \mathbb{R}%
^{nN}\;\;\;\;
\end{equation*}%
\emph{where }$L_{\infty ,AP}^{1}(\mathbb{R}^{N})$\emph{\ denotes the closure
with respect to the seminorm }$\left\| \cdot \right\| _{1}$\emph{\ (defined
in Section 2) of the space of finite sums} 
\begin{equation*}
\sum_{\text{finite}}\varphi _{i}u_{i}\text{\ \ \emph{with }}\varphi _{i}\in 
\mathcal{B}_{\infty }(\mathbb{R}^{N})\text{, }u_{i}\in AP(\mathbb{R}^{N}),
\end{equation*}%
$AP(\mathbb{R}^{N})$ \emph{being the space of all continuous real-valued
almost periodic functions on }$\mathbb{R}^{N}$ \emph{and }$\mathcal{B}%
_{\infty }(\mathbb{R}^{N})$\emph{\ the space of continuous real-valued
functions on }$\mathbb{R}^{N}$\emph{\ that converge at infinity.}
\end{itemize}
\end{example}

\begin{example}[\textit{Homogenization in non ergodic algebra}]
\label{e2}\emph{We assume here that }$N=1$\emph{. Let }$A$\emph{\ be the
algebra generated by the function }$f(z)=\cos \sqrt[3]{z}$\emph{\ (}$z\in 
\mathbb{R}$\emph{) and all its translates }$f(\cdot +a)$\emph{, }$a\in 
\mathbb{R}$\emph{. It is known that }$A$\emph{\ is an algebra with mean
value which is not ergodic; see \cite[p. 243]{20} for details. Since }$A$%
\emph{\ satisfies all the requirements of Theorem \ref{t4.1}, the conclusion
of Theorem \ref{t5.1} holds under the hypothesis }

\begin{itemize}
\item[(H)] $f(x,\cdot ,\lambda )\in B_{A}^{1}$\emph{\ for all }$(x,\lambda
)\in \overline{Q}\times \mathbb{R}^{n}$.
\end{itemize}

\noindent \emph{The homogenization problem solved here is new. One can also
consider other homogenization problems in the present setting of non-ergodic
algebras.}
\end{example}

\section{Homogenization of a stochastic Ladyzhenskaya model for
incompressible viscous flow}

We assume in this section that all vector spaces are real vector spaces, and
all scalar functions are real-valued. Let $A_{y}$ and $A_{\tau }$ be two
algebras wmv on $\mathbb{R}_{y}^{N}$ and $\mathbb{R}_{\tau }$, respectively.
We set $A=A_{y}\odot A_{\tau }$, the product algebra wmv defined as in \cite%
{26, CMP, CPAA}. Obviously, no ergodicity assumption is required neither on $%
A_{y}$, nor on $A_{\tau }$. We begin this section by giving some fundamental
notions about the probability theory and the stochastic calculus which will
be necessary throughout the rest of the section. Although these concepts are
well-known to specialists we find necessary to recall them in order to
facilitate the reading of the paper.

\subsection{Some tools from Stochastic Calculus and Abstract Probability
Theory}

Let $T$ be a positive real number and let $(\Omega ,\mathcal{F},\mathbb{P})$
be a probability space. On $(\Omega ,\mathcal{F},\mathbb{P})$ is defined a
prescribed $m$-dimensional standard Wiener process $W$. We equip $(\Omega ,%
\mathcal{F},\mathbb{P})$ with the natural filtration $(\mathcal{F}%
^{t})_{0\leq t\leq T}$ of $W$. The mathematical expectation on $(\Omega ,%
\mathcal{F},\mathbb{P})$ is denoted by $\mathbb{E}$. Assume that $X$ is a
Hilbert space. We denote by $X^{\otimes m}$ the Cartesian product $X\times
...\times X$, $m$ times. Let $u$ be a $X$-valued process such that

\begin{itemize}
\item[(a)] $u(t)$ is $\mathcal{F}^{t}$-measurable for each $t$,

\item[(b)] $\mathbb{E}\int_{0}^{T}\left\Vert u(t)\right\Vert _{X^{\otimes
m}}^{2}dt<\infty $.
\end{itemize}

\noindent For such a process we can define the \textit{stochastic integral}
(or \textit{It\^{o}'s integral}) 
\begin{equation*}
I(T)u:=\int_{0}^{T}u(t)dW(t)
\end{equation*}%
as the limit in probability of the sums 
\begin{equation*}
\sum_{k=0}^{n-1}u(t_{k})\cdot (W(t_{k+1})-W(t_{k})),
\end{equation*}%
as $|\Delta ^{n}|\rightarrow 0$. Here $\left\vert \Delta ^{n}\right\vert
=\max_{0\leq k\leq n-1}\left\vert t_{k}-t_{k-1}\right\vert $ is the mesh (or
modulus) of the partition $\Delta ^{n}=\{0=t_{0}<t_{1}<...<t_{n}=T\}$ of $%
I=[0,T]$. We state an important property of the stochastic integral. See,
for example, \cite{DaPrato} for its proof.

\begin{theorem}
\label{t6.0}For a process $u(t)$ satisfying \emph{(a)} and \emph{(b)} above,
the stochastic process $\int_{0}^{t}u(s)dW(s)$ is an $X$-valued continuous
martingale. Moreover, we have 
\begin{equation*}
\mathbb{E}I(t)u=0\text{,$\;0\leq t\leq T$}
\end{equation*}%
and 
\begin{equation*}
\mathbb{E}\left\vert I(t)u\right\vert _{X}^{2}=\mathbb{E}\int_{0}^{t}\left%
\Vert u(s)\right\Vert _{X}^{2}ds\text{,$\;0\leq t\leq T$.}
\end{equation*}
\end{theorem}

We note that for those $\mathcal{F}^{t}$-adapted stochastic process $u(t)$
such that 
\begin{equation*}
\int_{0}^{T}\left\Vert u(t)\right\Vert _{X}^{2}dt<\infty \text{\ almost
surely,}
\end{equation*}%
$\left\Vert I(t)u\right\Vert _{X}$ is no longer a martingale but a
continuous local martingale.

Now let $Z(t)$ be a $\mathbb{R}$-valued It\^{o}'s integral with respect to a
standard Brownian motion in $\mathbb{R}^{m}$ defined by 
\begin{equation*}
Z(t)=\sum_{j=1}^{m}\int_{0}^{t}g_{j}(s)dW^{j}(s),
\end{equation*}%
where $(g_{j}(s))_{1\leq j\leq m}$ are $\mathcal{F}^{t}$-adapted processes
such that 
\begin{equation*}
\int_{0}^{T}g_{j}^{2}(s)ds<\infty \text{\ almost surely, for }1\leq j\leq m%
\text{.}
\end{equation*}%
The corresponding It\^{o} integrals exist and they are local martingales. We
have the following result known as \textit{It\^{o}'s formula} (see for
example \cite{revuz}).

\begin{theorem}
\label{FOITO}Let $u(t)$ be a stochastic process given by 
\begin{equation*}
u(t)=\int_{0}^{t}b(s)ds+Z(t),
\end{equation*}%
where $b(s)$ is an adapted integrable process over $[0,T]$ in $\mathbb{R}$.
Suppose that $\phi :\mathbb{R}\times \lbrack 0,T]\rightarrow \mathbb{R}$ is
a continuous function such that $\phi (x,t)$ is continuously differentiable
twice in $x$ and once in $t$. Then, the following holds 
\begin{equation*}
\begin{split}
\phi (u(t),t)& =\phi (u(0),0)+\int_{0}^{t}\frac{\partial \phi (u(s),s)}{%
\partial s}ds+\int_{0}^{t}\frac{\partial \phi (u(s),s)}{\partial x}b(s)ds \\
& +\sum_{j=1}^{m}\int_{0}^{t}\frac{\partial \phi (u(s),s)}{\partial x}%
g_{j}(s)dW^{j}(s) \\
& +\frac{1}{2}\sum_{j=1}^{m}\int_{0}^{t}\frac{\partial ^{2}\phi (u(s),s)}{%
\partial x^{2}}g_{j}^{2}(s)ds.
\end{split}%
\end{equation*}
\end{theorem}

As a special case we have the following It\^{o}'s formula for $\left\Vert
u(t)\right\Vert _{X}^{2}$; we refer to \cite{pardoux} for its proof. We
recall that $W=(W^{1},...,W^{m})$ is a standard Brownian motion in $\mathbb{R%
}^{m}$.

\begin{theorem}
\label{t6.0'}Assume that $\zeta $ is $\mathcal{F}^{0}$-random variable in $X$
with $\left\Vert \mathbb{\zeta }\right\Vert _{X}^{2}<\infty $. Suppose that $%
b(t)$, and $g_{j}(t)$ are $\mathcal{F}^{t}$-adapted $X$-valued processes
such that $b(t)$ is integrable over $[0,T]$ and 
\begin{equation*}
\mathbb{E}\left( \int_{0}^{T}\left\Vert b(t)\right\Vert
_{X}^{2}dt+\sum_{j=1}^{m}\int_{0}^{T}\left\Vert g_{j}(t)\right\Vert
_{X}^{2}dt\right) <\infty .
\end{equation*}%
Let $u(t)$ be a $X$-valued process given by 
\begin{equation*}
u(t)=\zeta +\int_{0}^{t}b(s)ds+\sum_{j=1}^{m}\int_{0}^{t}g_{j}(s)dW^{j}(s).
\end{equation*}%
Then we have 
\begin{equation}
\begin{split}
\left\Vert u(t)\right\Vert _{X}^{2}& =\left\Vert \zeta \right\Vert
_{X}^{2}+2\int_{0}^{t}((b(s),u(s)))_{X}ds+2\sum_{j=1}^{m}%
\int_{0}^{t}((g_{j}(s),u(s)))_{X}dW^{j} \\
& +\int_{0}^{t}\sum_{j=1}^{m}\left\Vert g_{j}(s)\right\Vert _{X}^{2}ds\text{,%
}
\end{split}
\label{6.0}
\end{equation}%
where $((\cdot ,\cdot ))_{X}$ denotes the inner product in $X$.
\end{theorem}

In what follows we quote the famous Burkh\"{o}lder-Davis-Gundy inequality
(see e.g., \cite{Peszat} or \cite{revuz}).

\begin{theorem}
\label{BDGIN}Let $Z_{t}$ be the It\^{o}'s integral in $X$ given by 
\begin{equation*}
Z_{t}=\int_{0}^{t}X(s)dW(s).
\end{equation*}%
Then for any $p>0$ there exists a constant $K_{p}$ ($K_{1}=3$) such that 
\begin{equation*}
\mathbb{E}\sup_{0\leq t\leq T}\left\| \int_{0}^{t}X(s)dW(s)\right\|
_{X}^{p}\leq K_{p}\mathbb{E}\left( \int_{0}^{T}\left\| X(s)\right\|
_{X}^{2}ds\right) ^{\frac{p}{2}}\text{,}
\end{equation*}%
provided that 
\begin{equation*}
\mathbb{E}\left( \int_{0}^{T}\left\| X(s)\right\| _{X}^{2}ds\right) ^{\frac{p%
}{2}}<\infty \text{.}
\end{equation*}
\end{theorem}

Now we turn our attention to the weak convergence topology in the space of
Borel probability measures on topological spaces (see e.g. \cite{KALLIANPUR}%
). For a topological space $\mathfrak{X}$ we denote by $\mathcal{P}(%
\mathfrak{X})$ the space of Borel probability measures on $(\mathfrak{X},%
\mathcal{B}(\mathfrak{X}))$, $\mathcal{B}(\mathfrak{X})$ being the Borel $%
\sigma $-field of $\mathfrak{X}.$

\begin{definition}
\label{d6.0}Let $\mathfrak{X}$ be a topological space.

\begin{enumerate}
\item[(i)] A family $\mathfrak{P}_{k}$ of probability measures on $(%
\mathfrak{X},\mathcal{B}(\mathfrak{X}))$ is relatively compact if every
sequence of elements of $\mathfrak{P}_{k}$ contains a subsequence $\mathfrak{%
P}_{k_{j}}$ which converges weakly to a probability measure $\mathfrak{P}$,
that is, for any bounded and continuous function $\phi $ on $\mathfrak{X}$, 
\begin{equation*}
\lim_{k_{j}\rightarrow \infty }\int_{\mathfrak{X}}\phi (x)d\mathfrak{P}%
_{k_{j}}(x)=\int_{\mathfrak{X}}\phi (x)d\mathfrak{P}(x).
\end{equation*}

\item[(ii)] The family $\mathfrak{P}_{k}$ is said to be tight if for any $%
\varepsilon >0$, there exists a compact set $K_{\varepsilon }\subset 
\mathfrak{X}$ such that $\mathbb{P}(K_{\varepsilon })\geq 1-\varepsilon $\
for every $\mathbb{P}\in \mathfrak{P}_{k}$.
\end{enumerate}
\end{definition}

For a Polish space $\mathfrak{X}$ (that is, a separable and complete metric
space), the following theorem due to Prokhorov gives a sufficient and
necessary condition for a sequence of probability measures on $\mathfrak{X}$
to be weakly (or relatively) compact. We refer to \cite{DaPrato} (see also 
\cite{KALLIANPUR}) for its proof.

\begin{theorem}[Prokhorov]
\label{prok}The family $\mathfrak{P}_{k}$ is relatively compact if and only
if it is tight.
\end{theorem}

Next we present the relationship between convergence in distribution and
convergence almost surely of random variables (see once again \cite{DaPrato}
or \cite{KALLIANPUR}).

\begin{theorem}[Skorokhod]
\label{sko}For any sequence of probability measures $\mathfrak{P}_{k}$ on $%
(\Omega ,\mathcal{F},\mathbb{P})$ which converges to a probability measure $%
\mathfrak{P}$, there exist a probability space $(\Omega ^{\prime },\mathcal{F%
}^{\prime },\mathbb{P}^{\prime })$ and random variables $X_{k}$, $X$ defined
on $\Omega ^{\prime }$ such that the probability law of $X_{k}$ (resp., $X$)
is $\mathfrak{P}_{k}$(resp., $\mathfrak{P}$) and $\lim_{k\rightarrow \infty
}X_{k}=X$\ $\mathbb{P}^{\prime }$-almost surely.
\end{theorem}

The following criteria for convergence in probability whose proof can be
found in \cite{GYONGY} need to be highlighted.

\begin{lemma}
\label{l6.0}Let $X$ be a Polish space. A sequence of a $X$-valued random
variables $\{x_{n}:n\geq 0\}$ converges in probability if and only if for
every subsequence of joint probability laws, $\{\nu _{n_{k},m_{k}}:k\geq 0\}$%
, there exists a further subsequence which converges weakly to a probability
measure $\nu $ such that 
\begin{equation*}
\nu \left( \{(x,y)\in X\times X:x=y\}\right) =1.
\end{equation*}
\end{lemma}

\subsection{Statement of the problem and a priori estimates}

Let $(\Omega ,\mathcal{F},\mathbb{P})$ be a probability space. On $(\Omega ,%
\mathcal{F},\mathbb{P})$ we define a prescribed $m$-dimensional standard
Wiener process $W$. We equip $(\Omega ,\mathcal{F},\mathbb{P})$ with the
natural filtration $(\mathcal{F}^{t})$ of $W$ as in the preceding
subsection. We therefore aim at studying the asymptotics of the following
stochastic generalized Navier-Stokes type equations 
\begin{equation}
\left\{ 
\begin{array}{l}
d\mathbf{u}_{\varepsilon }+\left( -\Div\left[ a\left( \frac{x}{\varepsilon },%
\frac{t}{\varepsilon }\right) D\mathbf{u}_{\varepsilon }+b\left( \frac{x}{%
\varepsilon },\frac{t}{\varepsilon }\right) \left\vert D\mathbf{u}%
_{\varepsilon }\right\vert ^{p-2}D\mathbf{u}_{\varepsilon }\right] +\mathbf{u%
}_{\varepsilon }\cdot D\mathbf{u}_{\varepsilon }+Dq_{\varepsilon }\right) dt
\\ 
\ \ \ \ \ \ =\mathbf{f}dt+g\left( \frac{x}{\varepsilon },\frac{t}{%
\varepsilon },\mathbf{u}_{\varepsilon }\right) dW\text{\ in }Q_{T} \\ 
\Div\mathbf{u}_{\varepsilon }=0\text{\ in }Q_{T} \\ 
\mathbf{u}_{\varepsilon }=0\text{\ on }\partial Q\times (0,T) \\ 
\mathbf{u}_{\varepsilon }(x,0)=\mathbf{u}^{0}(x)\text{ in }Q\text{.}%
\end{array}%
\right.   \label{6.1}
\end{equation}%
In order that (\ref{6.1}) becomes meaningful, we need to precise the data.
Let $Q$ be a smooth bounded open set in $\mathbb{R}_{x}^{N}$ ($N=2$ or $3$),
and let $T$ be a positive real number. In $Q_{T}=Q\times (0,T)$ we consider
the partial differential operator (where $D$ and $\Div$ denote respectively
the gradient operator and divergence operator in $Q$) 
\begin{equation*}
P^{\varepsilon }=-\Div\left[ a\left( \frac{x}{\varepsilon },\frac{t}{%
\varepsilon }\right) D\cdot \right] :=-\overset{N}{\underset{i,j=1}{\sum }}%
\frac{\partial }{\partial x_{i}}\left( a_{ij}\left( \frac{x}{\varepsilon },%
\frac{t}{\varepsilon }\right) \frac{\partial }{\partial x_{j}}\right) 
\end{equation*}%
where the function $a=(a_{ij})_{1\leq i,j\leq N}\in L^{\infty }(\mathbb{R}%
_{y}^{N}\times \mathbb{R}_{\tau })^{N\times N}$ satisfies the following
assumptions: 
\begin{equation}
a_{ij}=a_{ji}  \label{6.2}
\end{equation}%
and there exists a constant $\nu _{0}>0$ such that 
\begin{equation}
\overset{N}{\underset{i,j=1}{\sum }}a_{ij}(y,\tau )\lambda _{i}\lambda
_{j}\geq \nu _{0}\left\vert \lambda \right\vert ^{2}\text{ for all }\lambda
=(\lambda _{i})\in \mathbb{R}^{N}\text{ and a.e. }(y,\tau )\in \mathbb{R}%
_{y,\tau }^{N+1}\text{.}  \label{6.3}
\end{equation}%
The operator $P^{\varepsilon }$ above defined is assumed to act on vector
functions as follows: for $\mathbf{u}=(u^{i})_{1\leq i\leq N}\in
(W^{1,p}(Q))^{N}$ we have $P^{\varepsilon }\mathbf{u}=(P^{\varepsilon
}u^{i})_{1\leq i\leq N}$. The function $b\in L^{\infty }(\mathbb{R}_{y,\tau
}^{N+1})$ and verifies $c_{1}\leq b(y,\tau )\leq c_{1}^{-1}$ a.e. $(y,\tau
)\in \mathbb{R}^{N}\times \mathbb{R}$ where $c_{1}$ is a positive constant.
So, putting $b(y,\tau ,\lambda )=b(y,\tau )\left\vert \lambda \right\vert
^{p-2}\lambda $, the function $b:(y,\tau ,\lambda )\mapsto b(y,\tau ,\lambda
)$, from $\mathbb{R}^{N}\times \mathbb{R}\times \mathbb{R}^{N\times N}$ into 
$\mathbb{R}^{N\times N}$ satisfies: 
\begin{equation}
\text{For each fixed }\lambda \in \mathbb{R}^{N}\text{, }b(\cdot ,\cdot
,\lambda )\text{ is measurable;}  \label{6.4}
\end{equation}%
\begin{equation}
b(y,\tau ,0)=0\text{ a.e. }(y,\tau )\in \mathbb{R}^{N}\times \mathbb{R}\text{%
;}  \label{6.5}
\end{equation}%
\begin{equation}
\begin{array}{l}
\text{There are two positive constants }\nu _{1}\text{ and }\nu _{2}\text{
such that, for a.e. }(y,\tau )\in \mathbb{R}^{N}\times \mathbb{R}\text{,} \\ 
\text{(i) }\left( b(y,\tau ,\lambda )-b(y,\tau ,\mu )\right) \cdot (\lambda
-\mu )\geq \nu _{1}\left\vert \lambda -\mu \right\vert ^{p} \\ 
\text{(ii) }\left\vert b(y,\tau ,\lambda )-b(y,\tau ,\mu )\right\vert \leq
\nu _{2}\left( \left\vert \lambda \right\vert +\left\vert \mu \right\vert
\right) ^{p-2}\left\vert \lambda -\mu \right\vert  \\ 
\text{for all }\lambda ,\mu \in \mathbb{R}^{N\times N}\text{,}%
\end{array}
\label{6.6}
\end{equation}%
where $p\geq 3$ is a real number, the dot denotes the usual Euclidean inner
product in $\mathbb{R}^{N\times N}$, and $\left\vert \cdot \right\vert $ the
associated norm. Next, the mapping $(y,\tau ,u)\mapsto g(y,\tau ,u)$ from $%
\mathbb{R}^{N}\times \mathbb{R}\times \mathbb{R}^{N}$ into $\mathbb{R}^{m}$
(integer $m\geq 1)$ satisfies the assumption that there exist positive
constants $c_{0}$ and $c_{1}$ such that 
\begin{equation}
\begin{array}{l}
\text{(i) }g(\cdot ,\cdot ,u)\text{ is measurable for any }u\in \mathbb{R}%
^{N}; \\ 
\text{(ii) }\left\vert g(y,\tau ,u)\right\vert \leq c_{0}(1+\left\vert
u\right\vert ); \\ 
\text{(iii) }\left\vert g(y,\tau ,u_{1})-g(y,\tau ,u_{2})\right\vert \leq
c_{1}\left\vert u_{1}-u_{2}\right\vert  \\ 
\text{for all }u,u_{1},u_{2}\in \mathbb{R}^{N}\text{ and for a.e. }(y,\tau
)\in \mathbb{R}^{N}\times \mathbb{R}.\ \ \ \ \ \ \ \ \ \ \ \ \ \ \ \ \ \ \ \
\ \ \ \ \ \ \;\;\;%
\end{array}
\label{6.7}
\end{equation}%
The first issue to be discussed is related to the existence and uniqueness
of the solution of (\ref{6.1}). Prior to that, we introduce the following
spaces \cite{Lions, Temam} 
\begin{equation*}
\begin{array}{l}
\mathcal{V}=\{\mathbf{\varphi }\in \mathcal{C}_{0}^{\infty }(Q)^{N}:\Div%
\mathbf{\varphi }=0\}; \\ 
V=\text{ closure of }\mathcal{V}\text{ in }W^{1,p}(Q)^{N}=\{\mathbf{u}\in
W_{0}^{1,p}(Q)^{N}:\Div\mathbf{u}=0\}; \\ 
H=\text{ closure of }\mathcal{V}\text{ in }L^{2}(Q)^{N}\text{.}%
\end{array}%
\end{equation*}%
We endow $V$ with the $W_{0}^{1,p}(Q)^{N}$-norm (the gradient norm), which
gives a reflexive Banach space. The space $H$ is equipped with the $%
L^{2}(Q)^{N}$-norm which makes it a Hilbert space. For $\mathbf{u}\in
L^{p}(0,T;V)$ the question for the existence of the trace function $%
(x,t)\mapsto b(x/\varepsilon ,t/\varepsilon ,D\mathbf{u}(x,t))$ can be
discussed in the same way as in \cite{NgWou1}. Also the function $%
(x,t)\mapsto a(x/\varepsilon ,t/\varepsilon )$ is well defined. With this in
mind, let 
\begin{equation*}
a^{\varepsilon }\left( x,t\right) =a\left( \frac{x}{\varepsilon },\frac{t}{%
\varepsilon }\right) 
\end{equation*}%
and 
\begin{equation*}
b^{\varepsilon }(\cdot ,\cdot ,D\mathbf{u})(x,t)=b\left( \frac{x}{%
\varepsilon },\frac{t}{\varepsilon },D\mathbf{u}(x,t)\right) :=b\left( \frac{%
x}{\varepsilon },\frac{t}{\varepsilon }\right) \left\vert D\mathbf{u}%
(x,t)\right\vert ^{p-2}D\mathbf{u}(x,t)
\end{equation*}%
for $(x,t)\in Q_{T}$. We introduce the functionals 
\begin{equation*}
a_{I}(\mathbf{u},\mathbf{v})=\int_{Q}\left( a^{\varepsilon }D\mathbf{u}%
\right) \cdot D\mathbf{v}dx+\int_{Q}b^{\varepsilon }(\cdot ,\cdot ,D\mathbf{u%
})\cdot D\mathbf{v}dx\text{\ }(\mathbf{u},\mathbf{v}\in W_{0}^{1,p}(Q)^{N});
\end{equation*}%
\begin{equation*}
b_{I}(\mathbf{u},\mathbf{v},\mathbf{w})=\sum_{i,k=1}^{N}\int_{Q}u^{i}\frac{%
\partial v^{k}}{\partial x_{i}}w^{k}dx\ \ (\mathbf{u}=(u^{i}),\mathbf{v},%
\mathbf{w}\in W_{0}^{1,p}(Q)^{N}).
\end{equation*}%
Then the following estimates hold: 
\begin{equation}
\left\vert a_{I}(\mathbf{u},\mathbf{v})\right\vert \leq \left\Vert
a\right\Vert _{\infty }\left\Vert D\mathbf{u}\right\Vert
_{L^{2}(Q)}\left\Vert D\mathbf{v}\right\Vert _{L^{2}(Q)}+\nu _{2}\left\Vert D%
\mathbf{u}\right\Vert _{L^{p}(Q)}^{p-1}\left\Vert D\mathbf{v}\right\Vert
_{L^{p}(Q)};  \label{6.8}
\end{equation}%
\begin{equation}
a_{I}(\mathbf{v},\mathbf{v})\geq \nu _{0}\left\Vert D\mathbf{v}\right\Vert
_{L^{2}(Q)}^{2}+\nu _{1}\left\Vert D\mathbf{v}\right\Vert _{L^{p}(Q)}^{p}
\label{6.9}
\end{equation}%
for all $\mathbf{u},\mathbf{v}\in W_{0}^{1,p}(Q)^{N}$. From the above
estimate (\ref{6.8}) we infer by the Riesz representation theorem the
existence of an operator $\mathcal{A}^{\varepsilon }:V\rightarrow V^{\prime }
$ such that 
\begin{equation*}
a_{I}(\mathbf{u},\mathbf{v})=\left\langle P^{\varepsilon }\mathbf{u},\mathbf{%
v}\right\rangle +\left\langle \mathcal{A}^{\varepsilon }\mathbf{u},\mathbf{v}%
\right\rangle \text{\ for all }\mathbf{u},\mathbf{v}\in V\text{.}
\end{equation*}%
It is worth noting that since $p\geq 3$ (hence $p\geq 2$) we have $%
P^{\varepsilon }\mathbf{u}\in V^{\prime }$ for $\mathbf{u}\in V$. Moreover
the operator $\mathcal{A}^{\varepsilon }$ (for fixed $\varepsilon >0$) is
maximal monotone, surjective and hemicontinuous \cite[Chap. 2, Section 2]%
{Lions}. As far as the trilinear form $b_{I}$ is concerned, we have that 
\cite{Lions, Temam} 
\begin{equation*}
b_{I}(\mathbf{u},\mathbf{v},\mathbf{v})=0\text{\ for all }\mathbf{u}\in V%
\text{ and }\mathbf{v}\in W_{0}^{1,p}(Q)^{N};
\end{equation*}%
\begin{equation*}
b_{I}(\mathbf{u},\mathbf{u},\mathbf{v})=b_{I}(\mathbf{u},\mathbf{v},\mathbf{u%
})\text{\ for }\mathbf{u}\in V\text{ and }\mathbf{v}\in W_{0}^{1,p}(Q)^{N}.
\end{equation*}%
Furthermore, since $W^{1,p}(Q)\subset L^{r}(Q)$ \cite{Adams} for any $r>1$
(indeed for $N=2,3$ and $p\geq 3$ we have $\frac{1}{p}-\frac{1}{N}\leq 0$,
so that, by the Sobolev embedding, the above embedding holds true). So
choosing $r>1$ in such a way that $\frac{2}{r}+\frac{1}{p}=1$, we have by H%
\"{o}lder's inequality that 
\begin{equation}
\left\vert b_{I}(\mathbf{u},\mathbf{v},\mathbf{u})\right\vert \leq
c\left\Vert \mathbf{u}\right\Vert _{L^{r}(Q)}^{2}\left\Vert D\mathbf{v}%
\right\Vert _{L^{p}(Q)}\text{\ for all }\mathbf{u},\mathbf{v}\in
W_{0}^{1,p}(Q)^{N}.  \label{6.9'}
\end{equation}%
We therefore infer the existence of an element $B(\mathbf{u})\in V^{\prime }$
such that 
\begin{equation}
\left\langle B(\mathbf{u}),\mathbf{v}\right\rangle =b_{I}(\mathbf{u},\mathbf{%
u},\mathbf{v})\text{\ for all }\mathbf{u},\mathbf{v}\in V.  \label{6.10}
\end{equation}%
Equation (\ref{6.10}) defines a bounded operator $B:V\rightarrow V^{\prime }$
with the further property that if $\mathbf{u}\in L^{p}(0,T;V)$ then $B(%
\mathbf{u})\in L^{p^{\prime }}(0,T;V^{\prime })$. In fact, from (\ref{6.9'}%
)-(\ref{6.10}) we have by H\"{o}lder's inequality (for $\mathbf{u}\in
L^{p}(0,T;V)$), 
\begin{equation*}
\left\Vert B(\mathbf{u})\right\Vert _{L^{p^{\prime }}(0,T;V^{\prime })}\leq
\left( \int_{0}^{T}\left\Vert \mathbf{u}(t)\right\Vert
_{L^{r}(Q)}^{2p^{\prime }}dt\right) ^{1/p^{\prime }}.
\end{equation*}%
But $W^{1,p}(Q)\subset L^{r}(Q)$ (with continuous embedding), hence there is
a positive constant $c$ independent of $\mathbf{u}$, such that 
\begin{equation*}
\left\Vert B(\mathbf{u})\right\Vert _{L^{p^{\prime }}(0,T;V^{\prime })}\leq
c\left( \int_{0}^{T}\left\Vert \mathbf{u}(t)\right\Vert _{V}^{2p^{\prime
}}dt\right) ^{1/p^{\prime }}.
\end{equation*}%
Also, as $p\geq 3$, we have $2p^{\prime }\leq p$, so that using once again H%
\"{o}lder's inequality with exponent $p/2p^{\prime }\geq 1$, we get 
\begin{equation*}
\left( \int_{0}^{T}\left\Vert \mathbf{u}(t)\right\Vert _{V}^{2p^{\prime
}}dt\right) ^{1/p^{\prime }}\leq c\left( \int_{0}^{T}\left\Vert \mathbf{u}%
(t)\right\Vert _{V}^{p}dt\right) ^{2/p}.
\end{equation*}%
We therefore deduce 
\begin{equation}
\left\Vert B(\mathbf{u})\right\Vert _{L^{p^{\prime }}(0,T;V^{\prime })}\leq
c\left( \int_{0}^{T}\left\Vert \mathbf{u}(t)\right\Vert _{V}^{p}dt\right)
^{2/p}.  \label{6.9''}
\end{equation}%
The above inequality will be useful in the sequel. Finally, for the sake of
completeness, we choose\textbf{\ }$\mathbf{f}\in L^{p^{\prime
}}(0,T;V^{\prime })$ and $\mathbf{u}^{0}\in H$. We are now in a position to
state the existence and uniqueness result for (\ref{6.1}). Before we can do
that, however, we need to take the projection of (\ref{6.1}) on $V^{\prime }$%
; we get the following abstract form in $V^{\prime }$: 
\begin{equation}
\left\{ 
\begin{array}{l}
d\mathbf{u}_{\varepsilon }+(P^{\varepsilon }\mathbf{u}_{\varepsilon }+%
\mathcal{A}^{\varepsilon }\mathbf{u}_{\varepsilon }+B(\mathbf{u}%
_{\varepsilon }))dt=\mathbf{f}dt+g^{\varepsilon }(\mathbf{u}_{\varepsilon
})dW,\ 0<t<T \\ 
\mathbf{u}_{\varepsilon }(0)=\mathbf{u}^{0}.%
\end{array}%
\right.   \label{6.11}
\end{equation}%
With all the properties of the operator $\mathcal{A}^{\varepsilon }$ (among
which the strict monotonicity, the maximality and the hemicontinuity) the
existence and uniqueness of a martingale solution (and hence from the
uniqueness, the strong) solution to (\ref{6.11}) follows exactly the way of
proceeding as in \cite{Sritharan}, and we can formulate the following result
without proof.

\begin{theorem}
\label{t6.1}Let the hypotheses be as above. Let $\varepsilon >0$ be freely
fixed and let $1<r<\infty $. There exists an $\mathcal{F}^{t}$-progressively
measurable process $\mathbf{u}_{\varepsilon }\in L^{r}(\Omega ,\mathcal{F},%
\mathbb{P};L^{p}(0,T;V)\cap L^{\infty }(0,T;H))$ such that 
\begin{equation}
\begin{array}{l}
\left( \mathbf{u}_{\varepsilon }(t),\mathbf{v}\right) +\int_{0}^{t}\left(
P^{\varepsilon }\mathbf{u}_{\varepsilon }(s)+\mathcal{A}^{\varepsilon }%
\mathbf{u}_{\varepsilon }(s)+B(\mathbf{u}_{\varepsilon }(s)),\mathbf{v}%
\right) ds=\left( \mathbf{u}^{0},\mathbf{v}\right) \\ 
\ \ \ \ \ \ \ \ \ \ \ \ \ \ \ \ \ \ \ \ \ \ +\int_{0}^{t}\left( \mathbf{f}%
(s),\mathbf{v}\right) ds+\int_{0}^{t}\left( g^{\varepsilon }(\mathbf{u}%
_{\varepsilon }(s),\mathbf{v}\right) dW(s)%
\end{array}
\label{6.12}
\end{equation}%
for all $\mathbf{v}\in V$ and for almost all $(\omega ,t)\in \Omega \times
\lbrack 0,T]$. Moreover $\mathbf{u}_{\varepsilon }\in L^{r}(\Omega ;\mathcal{%
F},\mathbb{P};\mathcal{C}([0,T];H))$ and is unique in the sense that if $%
\mathbf{u}_{\varepsilon }$ and $\overline{\mathbf{u}}_{\varepsilon }$
satisfy \emph{(\ref{6.12})} then $\mathbb{P}(\omega :\mathbf{u}_{\varepsilon
}(t)=\overline{\mathbf{u}}_{\varepsilon }(t)$ in $V^{\prime }$ for all $t\in
\lbrack 0,T])=1$.
\end{theorem}

\begin{remark}
\label{r6.1}\emph{(1) For the existence result in the above theorem, we only
need to have }$p\geq 1+\frac{2N}{N+2}$\emph{, and for the uniqueness, the
more restricted assumption }$p\geq 1+\frac{N}{2}$\emph{\ is required; see 
\cite{Lions}. We have taken }$p\geq 3$\emph{\ only for the sake of
simplicity. We might take }$p\geq 1+\frac{N}{2}$\emph{\ for both the
existence and uniqueness. (2) Since }$f\in L^{p^{\prime }}(0,T;V^{\prime })$ 
\emph{the existence of the pressure }$q_{\varepsilon }$\emph{\ is out of
reach; see e.g., \cite[Proposition 3]{Simon1} (see also \cite{Simon2}). That
is why, in the sequel we are mainly concern with the asymptotics of the
velocity field }$\mathbf{u}_{\varepsilon }$\emph{\ defined in Theorem \ref%
{t6.1}. Accordingly, throughout the remainder of this section, we will only
refer to problem (\ref{6.11})\ instead of (\ref{6.1}).}
\end{remark}

It is very important to note that very few results are available as regards
the homogenization of SPDEs. We may cite \cite{Bensoussan1, Ichihara1,
Ichihara2, Sango, Wang1, Wang2} in that framework. In the just mentioned
work, the homogenization of SPDEs is studied under the periodicity
assumption on the coefficients of the equations considered. In addition, the
convergence method used is either the G-convergence method \cite%
{Bensoussan1, Ichihara1, Ichihara2} or the two-scale convergence method \cite%
{Wang1, Wang2}. In view of the study of the qualitative properties of the
solutions of SPDEs, it is more convenient to use an appropriate method
taking into account both the random and deterministic behaviours of these
solutions; see Subsection 3.2. As for the homogenization of SPDEs in a
general ergodic environment (or beyond) is concerned, no results is
available so far. The present work is therefore the first one in which such
a problem is considered.

Before we can proceed with the a priori estimates, let us set a convention.
The letter $C$ will throughout denote a positive constant whose value may
change according to the occasion. The dependence of constants on the
parameters will be written explicitly only when necessary. With this in
mind, the following a priori estimates hold.

\begin{proposition}
\label{p6.1}For each fixed $\varepsilon >0$, let $\mathbf{u}_{\varepsilon }$
be the unique solution of \emph{(\ref{6.11})}. Then for any $1<r<\infty $ we
have 
\begin{equation}
\mathbb{E}\sup_{0\leq t\leq T}\left\Vert \mathbf{u}_{\varepsilon
}(t)\right\Vert _{L^{2}(Q)}^{r}\leq C;  \label{6.13}
\end{equation}%
\begin{equation}
\mathbb{E}\int_{0}^{T}\left\Vert \mathbf{u}_{\varepsilon }(t)\right\Vert
_{H_{0}^{1}(Q)^{N}}^{2}dt\leq C  \label{6.13'}
\end{equation}%
and 
\begin{equation}
\mathbb{E}\int_{0}^{T}\left\Vert \mathbf{u}_{\varepsilon }(t)\right\Vert
_{V}^{2}dt\leq C  \label{6.14}
\end{equation}%
for any $\varepsilon >0$, where $C$ is a positive constant independent of $%
\varepsilon $.
\end{proposition}

\begin{proof}
Applying It\^{o}'s formula to $\left\Vert \mathbf{u}_{\varepsilon
}(t)\right\Vert _{L^{2}(Q)}^{2}$ gives 
\begin{eqnarray*}
&&\left\Vert \mathbf{u}_{\varepsilon }(t)\right\Vert
_{L^{2}(Q)}^{2}+2\int_{0}^{t}\left\langle P^{\varepsilon }\mathbf{u}%
_{\varepsilon }(s)+A^{\varepsilon }\mathbf{u}_{\varepsilon }(s)+B(\mathbf{u}%
_{\varepsilon }(s)),\mathbf{u}_{\varepsilon }(s)\right\rangle ds \\
&=&\left\Vert \mathbf{u}^{0}\right\Vert
_{L^{2}(Q)}^{2}+2\int_{0}^{t}\left\langle \mathbf{f}(s),\mathbf{u}%
_{\varepsilon }(s)\right\rangle ds+\int_{0}^{t}\left\vert g^{\varepsilon }(%
\mathbf{u}_{\varepsilon }(s))\right\vert ^{2}ds \\
&&+2\int_{0}^{t}\left( g^{\varepsilon }(\mathbf{u}_{\varepsilon }(s)),%
\mathbf{u}_{\varepsilon }(s)\right) dW(s).
\end{eqnarray*}%
By using (\ref{6.3}), [part (i) of] (\ref{6.6}) and (\ref{6.7}) we get 
\begin{equation}
\begin{array}{l}
\left\Vert \mathbf{u}_{\varepsilon }(t)\right\Vert _{L^{2}(Q)}^{2}+2\nu
_{0}\int_{0}^{t}\left\Vert \mathbf{u}_{\varepsilon }(t)\right\Vert
_{H_{0}^{1}(Q)^{N}}^{2}ds+2\nu _{1}\int_{0}^{t}\left\Vert \mathbf{u}%
_{\varepsilon }(t)\right\Vert _{V}^{p}ds \\ 
\leq \left\Vert \mathbf{u}^{0}\right\Vert
_{L^{2}(Q)}^{2}+2\int_{0}^{t}\left\Vert \mathbf{f}(s)\right\Vert _{V^{\prime
}}\left\Vert \mathbf{u}_{\varepsilon }(s)\right\Vert
_{V}ds+C\int_{0}^{t}\left\Vert \mathbf{u}_{\varepsilon }(s)\right\Vert
_{L^{2}(Q)}^{2}ds \\ 
+C+2\int_{0}^{t}\left( g^{\varepsilon }(\mathbf{u}_{\varepsilon }(s)),%
\mathbf{u}_{\varepsilon }(s)\right) dW(s).%
\end{array}
\label{6.15}
\end{equation}%
By Young's inequality applied to the first integral on the right-hand side
of (\ref{6.15}), 
\begin{eqnarray*}
2\int_{0}^{t}\left\Vert \mathbf{f}(s)\right\Vert _{V^{\prime }}\left\Vert 
\mathbf{u}_{\varepsilon }(s)\right\Vert _{V}ds &\leq &C(\nu
_{1})\int_{0}^{t}\left\Vert \mathbf{f}(s)\right\Vert _{V^{\prime
}}^{p^{\prime }}ds+\nu _{1}\int_{0}^{t}\left\Vert \mathbf{u}_{\varepsilon
}(s)\right\Vert _{V}^{p}ds \\
&\leq &C+\nu _{1}\int_{0}^{t}\left\Vert \mathbf{u}_{\varepsilon
}(s)\right\Vert _{V}^{p}ds.
\end{eqnarray*}%
Taking into account the above inequality in (\ref{6.15}) we are led to 
\begin{eqnarray*}
&&\left\Vert \mathbf{u}_{\varepsilon }(t)\right\Vert _{L^{2}(Q)}^{2}+2\nu
_{0}\int_{0}^{t}\left\Vert \mathbf{u}_{\varepsilon }(t)\right\Vert
_{H_{0}^{1}(Q)^{N}}^{2}ds+\nu _{1}\int_{0}^{t}\left\Vert \mathbf{u}%
_{\varepsilon }(t)\right\Vert _{V}^{p}ds \\
&\leq &\left\Vert \mathbf{u}^{0}\right\Vert
_{L^{2}(Q)}^{2}+C+C\int_{0}^{t}\left\Vert \mathbf{u}_{\varepsilon
}(s)\right\Vert _{L^{2}(Q)}^{2}ds+2\int_{0}^{t}\left( g^{\varepsilon }(%
\mathbf{u}_{\varepsilon }(s)),\mathbf{u}_{\varepsilon }(s)\right) dW(s).
\end{eqnarray*}%
Taking first the supremum over $0\leq s\leq t$ (for all $0\leq t\leq T$) and
next the mathematical expectation in the above inequality, 
\begin{equation}
\begin{array}{l}
\mathbb{E}\underset{0\leq s\leq t}{\sup }\left\Vert \mathbf{u}_{\varepsilon
}(s)\right\Vert _{L^{2}(Q)}^{2}+2\nu _{0}\mathbb{E}\int_{0}^{t}\left\Vert 
\mathbf{u}_{\varepsilon }(t)\right\Vert _{H_{0}^{1}(Q)^{N}}^{2}ds+\nu _{1}%
\mathbb{E}\int_{0}^{t}\left\Vert \mathbf{u}_{\varepsilon }(t)\right\Vert
_{V}^{p}ds \\ 
\leq \left\Vert \mathbf{u}^{0}\right\Vert _{L^{2}(Q)}^{2}+C+C\mathbb{E}%
\int_{0}^{t}\left\Vert \mathbf{u}_{\varepsilon }(s)\right\Vert
_{L^{2}(Q)}^{2}ds \\ 
\ \ \ \ +2\mathbb{E}\underset{0\leq s\leq t}{\sup }\left\vert
\int_{0}^{s}\left( g^{\varepsilon }(\mathbf{u}_{\varepsilon }(\tau )),%
\mathbf{u}_{\varepsilon }(\tau )\right) dW(\tau )\right\vert .%
\end{array}
\label{6.16}
\end{equation}%
Making use of the Burkh\"{o}lder-Davis-Gundy's inequality applied to the
last term in the right-hand side of (\ref{6.16}), 
\begin{eqnarray*}
&&2\mathbb{E}\underset{0\leq s\leq t}{\sup }\left\vert \int_{0}^{s}\left(
g^{\varepsilon }(\mathbf{u}_{\varepsilon }(\tau )),\mathbf{u}_{\varepsilon
}(\tau )\right) dW(\tau )\right\vert \\
&\leq &C\mathbb{E}\left( \int_{0}^{t}\left\vert \left( g^{\varepsilon }(%
\mathbf{u}_{\varepsilon }(s)),\mathbf{u}_{\varepsilon }(s)\right)
\right\vert ^{2}ds\right) ^{1/2} \\
&\leq &C\mathbb{E}\left[ \sup_{0\leq s\leq t}\left\Vert \mathbf{u}%
_{\varepsilon }(s)\right\Vert _{L^{2}(Q)}^{2}\left( \int_{0}^{t}\left\Vert
g^{\varepsilon }(\mathbf{u}_{\varepsilon }(s))\right\Vert
_{L^{2}(Q)}^{2}ds\right) ^{1/2}\right] ,
\end{eqnarray*}%
and by Cauchy-Schwartz's inequality, 
\begin{equation*}
\begin{array}{l}
2\mathbb{E}\underset{0\leq s\leq t}{\sup }\left\vert \int_{0}^{s}\left(
g^{\varepsilon }(\mathbf{u}_{\varepsilon }(\tau )),\mathbf{u}_{\varepsilon
}(\tau )\right) dW(\tau )\right\vert \leq \frac{1}{2}\mathbb{E}\underset{%
0\leq s\leq t}{\sup }\left\Vert \mathbf{u}_{\varepsilon }(s)\right\Vert
_{L^{2}(Q)}^{2} \\ 
\ \ \ \ \ \ \ \ +C\mathbb{E}\int_{0}^{t}\left\Vert \mathbf{u}_{\varepsilon
}(s)\right\Vert _{L^{2}(Q)}^{2}ds.%
\end{array}%
\end{equation*}%
Putting this in (\ref{6.16}) we derive 
\begin{equation}
\begin{array}{l}
\frac{1}{2}\mathbb{E}\underset{0\leq s\leq t}{\sup }\left\Vert \mathbf{u}%
_{\varepsilon }(s)\right\Vert _{L^{2}(Q)}^{2}+2\nu _{0}\mathbb{E}%
\int_{0}^{t}\left\Vert \mathbf{u}_{\varepsilon }(s)\right\Vert
_{H_{0}^{1}(Q)^{N}}^{2}ds+\nu _{1}\mathbb{E}\int_{0}^{t}\left\Vert \mathbf{u}%
_{\varepsilon }(s)\right\Vert _{V}^{p}ds \\ 
\leq C+C\mathbb{E}\int_{0}^{t}\left\Vert \mathbf{u}_{\varepsilon
}(s)\right\Vert _{L^{2}(Q)}^{2}ds.%
\end{array}
\label{6.17}
\end{equation}%
It therefore follows from Gronwall's inequality that 
\begin{equation}
\mathbb{E}\underset{0\leq t\leq T}{\sup }\left\Vert \mathbf{u}_{\varepsilon
}(t)\right\Vert _{L^{2}(Q)}^{2}\leq C  \label{6.18}
\end{equation}%
where $C$ is independent of $\varepsilon $. It also follows from (\ref{6.17}%
) and (\ref{6.18}) that 
\begin{equation*}
\mathbb{E}\int_{0}^{T}\left\Vert \mathbf{u}_{\varepsilon }(t)\right\Vert
_{V}^{p}dt\leq C\text{ and }\mathbb{E}\int_{0}^{T}\left\Vert \mathbf{u}%
_{\varepsilon }(t)\right\Vert _{H_{0}^{1}(Q)^{N}}^{2}dt\leq C
\end{equation*}%
where $C$ is also independent of $\varepsilon $, hence (\ref{6.13'}) and (%
\ref{6.14}).

Now, as far as (\ref{6.13}) is concerned, we start again from the It\^{o}'s
formula which reads in this case as: if 
\begin{equation*}
X_{t}=X_{0}+\int_{0}^{t}\phi (s)ds+N_{t}
\end{equation*}%
where $X_{t}=\left\Vert \mathbf{u}_{\varepsilon }(t)\right\Vert
_{L^{2}(Q)}^{2}$, $X_{0}=\left\Vert \mathbf{u}^{0}\right\Vert
_{L^{2}(Q)}^{2} $, $N_{t}=2\int_{0}^{t}\left( g^{\varepsilon }(\mathbf{u}%
_{\varepsilon }(s),\mathbf{u}_{\varepsilon }(s)\right) dW(s)$ and $\phi
(s)=-2\left\langle P^{\varepsilon }\mathbf{u}_{\varepsilon }(s)+\mathcal{A}%
^{\varepsilon }\mathbf{u}_{\varepsilon }(s),\mathbf{u}_{\varepsilon
}(s)\right\rangle +2\left\langle \mathbf{f}(s),\mathbf{u}_{\varepsilon
}(s)\right\rangle +\left\Vert g^{\varepsilon }(\mathbf{u}_{\varepsilon
}(s))\right\Vert _{L^{2}(Q)}^{2}$, then 
\begin{equation*}
X_{t}^{l}=X_{0}^{l}+l\int_{0}^{t}X_{s}^{l-1}\phi
(s)ds+l\int_{0}^{t}X_{s}^{l-1}dN_{s}+\frac{1}{2}l(l-1)%
\int_{0}^{t}X_{s}^{l-2}d\left\langle N_{s}\right\rangle
\end{equation*}%
for any $1\leq l<\infty $. We apply the above formula with $l=r/2$ ($r>2$)
and we get 
\begin{eqnarray*}
&&\left\Vert \mathbf{u}_{\varepsilon }(t)\right\Vert
_{L^{2}(Q)}^{r}+r\int_{0}^{t}\left\Vert \mathbf{u}_{\varepsilon
}(s)\right\Vert _{L^{2}(Q)}^{r-2}[\nu _{0}\left\Vert \mathbf{u}_{\varepsilon
}(s)\right\Vert _{H_{0}^{1}(Q)^{N}}^{2}+\nu _{1}\left\Vert \mathbf{u}%
_{\varepsilon }(s)\right\Vert _{V}^{p}]ds \\
&\leq &\left\Vert \mathbf{u}^{0}\right\Vert _{L^{2}(Q)}^{r}+\frac{r}{2}%
\int_{0}^{t}\left\Vert \mathbf{u}_{\varepsilon }(s)\right\Vert
_{L^{2}(Q)}^{r-2}\left\Vert g^{\varepsilon }(\mathbf{u}_{\varepsilon
}(s))\right\Vert _{L^{2}(Q)}^{2}ds \\
&&+r\left( \frac{r}{2}-1\right) \int_{0}^{t}\left\Vert \mathbf{u}%
_{\varepsilon }(s)\right\Vert _{L^{2}(Q)}^{r-2}\left\Vert g^{\varepsilon }(%
\mathbf{u}_{\varepsilon }(s))\right\Vert _{L^{2}(Q)}^{2}ds \\
&&+\frac{r}{2}\int_{0}^{t}\left\Vert \mathbf{u}_{\varepsilon }(s)\right\Vert
_{L^{2}(Q)}^{r-2}\left( g^{\varepsilon }(\mathbf{u}_{\varepsilon }(s)),%
\mathbf{u}_{\varepsilon }(s)\right) dW(s).
\end{eqnarray*}%
We can therefore follow the same way of reasoning as before (see also \cite%
{Sritharan}) to get (\ref{6.13}). The proof is completed.
\end{proof}

The estimate (\ref{6.13}) (for $r>2$) is concerned with the higher
integrability of the sequence ($\mathbf{u}_{\varepsilon })_{\varepsilon }$
so that we can make use of it together with the Vitali's theorem. However
this will become precise in the next subsection. Now, in order to prove the
tightness property of the sequence of probability laws of $(\mathbf{u}%
_{\varepsilon })_{\varepsilon }$, we will also need the

\begin{proposition}
\label{p6.2}Assuming that the function $t\mapsto \mathbf{u}_{\varepsilon
}(t) $ is extended by zero outside the interval $[0,T]$, there exists a
positive constant $C$ such that 
\begin{equation*}
\mathbb{E}\sup_{\left\vert \theta \right\vert \leq \delta
}\int_{0}^{T}\left\Vert \mathbf{u}_{\varepsilon }(t+\theta )-\mathbf{u}%
_{\varepsilon }(t)\right\Vert _{V^{\prime }}^{p^{\prime }}dt\leq C\delta ^{%
\frac{1}{p-1}}
\end{equation*}%
for each $\varepsilon >0$ and $0<\delta <1$.
\end{proposition}

\begin{proof}
Assume $\theta \geq 0$. The same way of reasoning will apply for $\theta <0$%
. We have 
\begin{eqnarray*}
\mathbf{u}_{\varepsilon }(t+\theta )-\mathbf{u}_{\varepsilon }(t)
&=&\int_{t}^{t+\theta }d\mathbf{u}_{\varepsilon }(s) \\
&=&\int_{t}^{t+\theta }[-P^{\varepsilon }\mathbf{u}_{\varepsilon }(s)-%
\mathcal{A}^{\varepsilon }\mathbf{u}_{\varepsilon }(s)-B(\mathbf{u}%
_{\varepsilon }(s))+\mathbf{f}(s)]ds \\
&&+\int_{t}^{t+\theta }g^{\varepsilon }(\mathbf{u}_{\varepsilon }(s))dW(s),
\end{eqnarray*}%
hence 
\begin{eqnarray*}
\left\Vert \mathbf{u}_{\varepsilon }(t+\theta )-\mathbf{u}_{\varepsilon
}(t)\right\Vert _{V^{\prime }}^{p^{\prime }} &\leq &C\left\Vert
\int_{t}^{t+\theta }P^{\varepsilon }\mathbf{u}_{\varepsilon
}(s)ds\right\Vert _{V^{\prime }}^{p^{\prime }}+C\left\Vert
\int_{t}^{t+\theta }\mathcal{A}^{\varepsilon }\mathbf{u}_{\varepsilon
}(s)ds\right\Vert _{V^{\prime }}^{p^{\prime }} \\
&&+C\left\Vert \int_{t}^{t+\theta }B(\mathbf{u}_{\varepsilon
}(s))ds\right\Vert _{V^{\prime }}^{p^{\prime }}+C\left\Vert
\int_{t}^{t+\theta }\mathbf{f}(s))ds\right\Vert _{V^{\prime }}^{p^{\prime }}
\\
&&+C\left\Vert \int_{t}^{t+\theta }\int_{t}^{t+\theta }g^{\varepsilon }(%
\mathbf{u}_{\varepsilon }(s))dW(s)\right\Vert _{V^{\prime }}^{p^{\prime }} \\
&=&I_{1}+I_{2}+I_{3}+I_{4}+I_{5}.
\end{eqnarray*}%
But 
\begin{equation*}
I_{1}\leq C\left( \int_{t}^{t+\theta }\left\Vert P^{\varepsilon }\mathbf{u}%
_{\varepsilon }(s)\right\Vert _{V^{\prime }}ds\right) ^{p^{\prime }}
\end{equation*}%
and 
\begin{eqnarray*}
\int_{t}^{t+\theta }\left\Vert P^{\varepsilon }\mathbf{u}_{\varepsilon
}(s)\right\Vert _{V^{\prime }}ds &\leq &C\int_{t}^{t+\theta }\left\Vert
P^{\varepsilon }\mathbf{u}_{\varepsilon }(s)\right\Vert _{H^{-1}(Q)^{N}}ds \\
&\leq &C\theta ^{\frac{1}{p}}\left( \int_{t}^{t+\theta }\left\Vert
P^{\varepsilon }\mathbf{u}_{\varepsilon }(s)\right\Vert
_{H^{-1}(Q)^{N}}^{p^{\prime }}ds\right) ^{\frac{1}{p^{\prime }}},
\end{eqnarray*}%
hence 
\begin{eqnarray*}
I_{1} &\leq &C\theta ^{\frac{p^{\prime }}{p}}\int_{t}^{t+\theta }\left\Vert
P^{\varepsilon }\mathbf{u}_{\varepsilon }(s)\right\Vert
_{H^{-1}(Q)^{N}}^{p^{\prime }}ds \\
&\leq &C\theta ^{\frac{p^{\prime }}{p}}\int_{t}^{t+\theta }\left(
1+\left\Vert P^{\varepsilon }\mathbf{u}_{\varepsilon }(s)\right\Vert
_{H^{-1}(Q)^{N}}^{2}\right) ds\text{\ since }p^{\prime }<2 \\
&\leq &C\theta ^{\frac{p^{\prime }}{p}}\int_{t}^{t+\theta }\left\Vert
P^{\varepsilon }\mathbf{u}_{\varepsilon }(s)\right\Vert
_{H^{-1}(Q)^{N}}^{2}ds \\
&\leq &C\delta ^{\frac{p^{\prime }}{p}}\int_{t}^{t+\theta }\left\Vert 
\mathbf{u}_{\varepsilon }(s)\right\Vert _{H_{0}^{1}(Q)^{N}}^{2}ds.
\end{eqnarray*}%
We infer from (\ref{6.13'}) that 
\begin{equation*}
\mathbb{E}\sup_{0\leq \theta \leq \delta }\int_{0}^{T}I_{1}dt\leq C\delta ^{%
\frac{p^{\prime }}{p}}\mathbb{E}\int_{0}^{T}\left( \int_{t}^{t+\delta
}\left\Vert \mathbf{u}_{\varepsilon }(s)\right\Vert
_{H_{0}^{1}(Q)^{N}}^{2}ds\right) dt\leq C\delta ^{\frac{p^{\prime }}{p}}.
\end{equation*}%
As for $I_{2}$ we have, as above, 
\begin{equation*}
I_{2}\leq C\theta ^{\frac{p^{\prime }}{p}}\int_{t}^{t+\theta }\left\Vert 
\mathcal{A}^{\varepsilon }\mathbf{u}_{\varepsilon }(s)\right\Vert
_{V^{\prime }}^{p^{\prime }}ds,
\end{equation*}%
hence 
\begin{eqnarray*}
\mathbb{E}\sup_{0\leq \theta \leq \delta }\int_{0}^{T}I_{2}dt &\leq &C\delta
^{\frac{p^{\prime }}{p}}\mathbb{E}\int_{0}^{T}\left( \int_{t}^{t+\delta
}\left\Vert D\mathbf{u}_{\varepsilon }(s)\right\Vert
_{L^{p}(Q)}^{p}ds\right) dt \\
&\leq &C\delta ^{\frac{p^{\prime }}{p}}\text{ by (\ref{6.14}).}
\end{eqnarray*}%
Now, dealing with $I_{3}$ we have 
\begin{equation*}
\int_{t}^{t+\theta }\left\Vert B(\mathbf{u}_{\varepsilon }(s))\right\Vert
_{V^{\prime }}ds\leq \theta ^{\frac{1}{p}}\int_{t}^{t+\theta }\left\Vert B(%
\mathbf{u}_{\varepsilon }(s))\right\Vert _{V^{\prime }}^{p^{\prime }}ds,
\end{equation*}%
and proceeding as above (taking into account (\ref{6.9''}) and (\ref{6.14}))
we get 
\begin{equation*}
\mathbb{E}\sup_{0\leq \theta \leq \delta }\int_{0}^{T}I_{3}dt\leq C\delta ^{%
\frac{p^{\prime }}{p}}.
\end{equation*}%
We also have $\mathbb{E}\sup_{0\leq \theta \leq \delta
}\int_{0}^{T}I_{4}dt\leq C\delta ^{\frac{p^{\prime }}{p}}$. For the last
integral, by using the Burkh\"{o}lder-Davis-Gundy's inequality we have 
\begin{eqnarray*}
\mathbb{E}\sup_{0\leq \theta \leq \delta }\int_{0}^{T}I_{5}dt &\leq
&C\int_{0}^{T}\mathbb{E}\left( \int_{t}^{t+\theta }\left\Vert g^{\varepsilon
}(\mathbf{u}_{\varepsilon }(s))\right\Vert _{L^{2}(Q)}^{2}ds\right) ^{\frac{%
p^{\prime }}{2}}dt \\
&\leq &C\int_{0}^{T}\mathbb{E}\left( \int_{t}^{t+\theta }\left\Vert \mathbf{u%
}_{\varepsilon }(s)\right\Vert _{L^{2}(Q)}^{2}ds\right) ^{\frac{p^{\prime }}{%
2}}dt\text{\ \ (by (\ref{6.7}))} \\
&\leq &C\delta ^{\frac{p^{\prime }}{2}}\int_{0}^{T}\mathbb{E}\sup_{0\leq
s\leq T}\left\Vert \mathbf{u}_{\varepsilon }(s)\right\Vert
_{L^{2}(Q)}^{p^{\prime }}dt \\
&\leq &C\delta ^{\frac{p^{\prime }}{2}}\text{\ by (\ref{6.13}).}
\end{eqnarray*}%
Combining all the above estimates leads at once at 
\begin{equation*}
E\sup_{0\leq \theta \leq \delta }\int_{0}^{T}\left\Vert \mathbf{u}%
_{\varepsilon }(t+\theta )-\mathbf{u}_{\varepsilon }(t)\right\Vert
_{V^{\prime }}^{p^{\prime }}dt\leq C\delta ^{\frac{p^{\prime }}{p}}
\end{equation*}%
since $\delta ^{\frac{p^{\prime }}{2}}\leq \delta ^{\frac{p^{\prime }}{p}}$
(recall that $p\geq 3$ and $0<\delta <1$). As the same inequality obviously
holds for $\theta <0$, the proof is completed.
\end{proof}

\subsection{Tightness property of the probability laws of $(\mathbf{u}_{%
\protect\varepsilon })_{\protect\varepsilon }$ \label{subsect6.3}}

We are now able to prove the tightness of the law of $(\mathbf{u}%
_{\varepsilon },W)$. We shall for this aim, follow the lead of Bensoussan 
\cite[Proposition 3.1]{bensoussan2} and Debussche et al. \cite{Glatt}.
Before we can proceed any further, we need the following important result.

\begin{lemma}[{\protect\cite[Proposition 3.1]{bensoussan2}}]
\label{l6.1}Let $(\mu _{n})_{n}$ and $(\nu _{n})_{n}$ be two ordinary
sequences of positive real numbers such that $\mu _{n},\nu _{n}\rightarrow 0$
as $n\rightarrow \infty $. For the three positive constants $K$, $L$ and $M$%
, the set 
\begin{equation*}
\begin{array}{l}
Z=\{\mathbf{u}:\int_{0}^{T}\left\Vert \mathbf{u}\right\Vert _{V}^{p}dt\leq L%
\text{,}\ \left\Vert \mathbf{u}(t)\right\Vert _{H}^{2}\leq K\text{ a.e. }t%
\text{,}\  \\ 
\ \ \ \sup_{\left\vert \theta \right\vert \leq \mu
_{n}}\int_{0}^{T}\left\Vert \mathbf{u}(t+\theta )-\mathbf{u}(t)\right\Vert
_{V^{\prime }}^{p^{\prime }}dt\leq \nu _{n}M\text{ for all }n\in \mathbb{N}\}%
\end{array}%
\end{equation*}%
is a compact subset of $L^{2}(0,T;H)$.
\end{lemma}

This being so, set $\mathfrak{S}=L^{2}(0,T;H)\times \mathcal{C}(0,T;\mathbb{R%
}^{m})$, a metric space equipped with its Borel $\sigma $-algebra $\mathcal{B%
}(\mathfrak{S})$. For $0<\varepsilon <1$, let $\Psi _{\varepsilon }$ be the
measurable $\mathfrak{S}$-valued mapping defined on $(\Omega ,\mathcal{F},%
\mathbb{P})$ as 
\begin{equation*}
\Psi _{\varepsilon }(\omega )=(\mathbf{u}_{\varepsilon }(\cdot ,\omega
),W(\cdot ,\omega ))\text{\ }(\omega \in \Omega )\text{.}
\end{equation*}%
We introduce the image of $\mathbb{P}$ under $\Psi _{\varepsilon }$ defined
by 
\begin{equation*}
\pi _{\varepsilon }(S)=\mathbb{P}(\Psi _{\varepsilon }^{-1}(S))\;\;(S\in 
\mathcal{B}(\mathfrak{S})),
\end{equation*}%
which defines a sequence of probability measures on $\mathfrak{S}$. The
following result holds.

\begin{theorem}
\label{t6.2}The sequence $(\pi _{\varepsilon })_{0<\varepsilon <1}$ is tight
in $(\mathfrak{S},\mathcal{B}(\mathfrak{S}))$.
\end{theorem}

\begin{proof}
Let $\delta >0$ and let $L_{\delta }$, $K_{\delta }$, $M_{\delta }$ be
positive constants depending only on $\delta $ (to be fixed later). We have
by Lemma \ref{l6.1} that 
\begin{equation*}
Z_{\delta }=\left\{ \mathbf{u}:\int_{0}^{T}\left\Vert \mathbf{u}\right\Vert
_{V}^{p}dt\leq L_{\delta }\text{,}\ \left\Vert \mathbf{u}(t)\right\Vert
_{H}^{2}\leq K_{\delta }\text{ a.e. }t\text{,}\ \sup_{\left\vert \theta
\right\vert \leq \mu _{n}}\int_{0}^{T}\left\Vert \mathbf{u}(t+\theta )-%
\mathbf{u}(t)\right\Vert _{V^{\prime }}^{p^{\prime }}\leq \nu _{n}M_{\delta
}\right\}
\end{equation*}%
is a compact subset of $L^{p}(0,T;H)$ for any $\delta >0$. Here we choose
the sequence $(\mu _{n})_{n}$ and $(\nu _{n})_{n}$ so that $\sum \frac{1}{%
\nu _{n}}(\mu _{n})^{\frac{1}{p-1}}<\infty $. Then we have 
\begin{equation*}
\begin{split}
\mathbb{P}\left( \mathbf{u}_{\varepsilon }\notin Z_{\delta }\right) & \leq 
\mathbb{P}\left( \int_{0}^{T}\left\Vert \mathbf{u}_{\varepsilon }\right\Vert
_{V}^{p}dt\geq L_{\delta }\right) +\mathbb{P}\left( \sup_{t\in \lbrack
0,T]}\left\Vert \mathbf{u}_{\varepsilon }(t)\right\Vert _{H}^{2}\geq
K_{\delta }\right) \\
& +\mathbb{P}\left( \sup_{\left\vert \theta \right\vert \leq \mu
_{n}}\int_{0}^{T}\left\Vert \mathbf{u}_{\varepsilon }(t+\theta )-\mathbf{u}%
_{\varepsilon }(t)\right\Vert _{V^{\prime }}^{p^{\prime }}dt\geq \nu
_{n}M_{\delta }\right) .
\end{split}%
\end{equation*}%
In view of Tchebychev's inequality we have 
\begin{equation*}
\begin{split}
\mathbb{P}(\mathbf{u}_{\varepsilon }& \notin Z_{\delta })\leq \frac{1}{%
L_{\delta }}\mathbb{E}\int_{0}^{T}\left\Vert \mathbf{u}_{\varepsilon
}(t)\right\Vert _{V}^{p}dt+\frac{1}{K_{\delta }}\mathbb{E}\sup_{t\in \lbrack
0,T]}\left\Vert \mathbf{u}_{\varepsilon }(t)\right\Vert _{H}^{2} \\
& +\sum \frac{1}{\nu _{n}M_{\delta }}\mathbb{E}\sup_{\left\vert \theta
\right\vert \leq \mu _{n}}\int_{0}^{T}\left\Vert \mathbf{u}_{\varepsilon
}(t+\theta )-\mathbf{u}_{\varepsilon }(t)\right\Vert _{V^{\prime
}}^{p^{\prime }}dt.
\end{split}%
\end{equation*}%
From Propositions \ref{p6.1} and \ref{p6.2} it follows that 
\begin{equation*}
\mathbb{P}(\mathbf{u}_{\varepsilon }\in Z_{\delta })\leq \frac{C}{L_{\delta }%
}+\frac{C}{K_{\delta }}+\frac{C}{M_{\delta }}\sum_{n}\frac{1}{\nu _{n}}(\mu
_{n})^{\frac{1}{p-1}}.
\end{equation*}%
So if we choose 
\begin{equation*}
K_{\delta }=L_{\delta }=\frac{6C}{\delta }\text{ and }M_{\delta }=\frac{%
6C\left( \sum_{n}\frac{1}{\nu _{n}}(\mu _{n})^{\frac{1}{p-1}}\right) }{%
\delta },
\end{equation*}%
then we have are led to 
\begin{equation}
\mathbb{P}\left( \mathbf{u}_{\varepsilon }\notin Z_{\delta }\right) \leq 
\frac{\delta }{2}.  \label{6.19}
\end{equation}%
Next, considering the sequence of probability measures $\pi
_{2}^{\varepsilon }(A):=\mathbb{P}(W\in A)$ ($A\in \mathcal{B}(\mathcal{C}%
(0,T;\mathbb{R}^{m}))$), it consists of only one element, hence it is weakly
compact. As $\mathcal{C}(0,T;\mathbb{R}^{m})$ is a Polish space, any weakly
compact sequence of probability measure is tight, so that, given $\delta >0$
there is a compact subset $C_{\delta }$ of $\mathcal{C}(0,T;\mathbb{R}^{m})$
such that $\mathbb{P}(W\in C_{\delta })\geq 1-\delta /2$. We infer from this
together with (\ref{6.18}) that 
\begin{equation*}
\mathbb{P}\left( (\mathbf{u}_{\varepsilon },W)\in Z_{\delta }\times
C_{\delta }\right) \geq 1-\delta \text{.}
\end{equation*}%
So we have just checked that for any $\delta >0$ there is a compact $%
Z_{\delta }\times C_{\delta }\subset \mathfrak{S}$ such that 
\begin{equation*}
\pi _{\varepsilon }(Z_{\delta }\times C_{\delta })\geq 1-\delta \text{,}
\end{equation*}%
by this proving the tightness of the family $\pi _{\varepsilon }$ in $%
\mathfrak{S}=L^{p}(0,T;H)\times \mathcal{C}(0,T;\mathbb{R}^{m})$.
\end{proof}

It follows from Theorem \ref{t6.2} and Prokhorov's theorem that there is a
subsequence $(\pi _{\varepsilon _{n}})_{n}$ of $(\pi _{\varepsilon
})_{0<\varepsilon <1}$ converging weakly (in the sense of measure) to a
probability measure $\Pi $. It emerges from Skorokhod's theorem that we can
find a new probability space $(\bar{\Omega},\bar{\mathcal{F}},\bar{\mathbb{P}%
})$ and random variables $(\mathbf{u}_{\varepsilon _{n}},W^{\varepsilon
_{n}})$, $(\mathbf{u}_{0},\bar{W})$ defined on this new probability space
and taking values in $\mathfrak{S}=L^{p}(0,T;H)\times \mathcal{C}(0,T;%
\mathbb{R}^{m})$ such that:

\begin{itemize}
\item[(i)] The probability law of $(\mathbf{u}_{\varepsilon
_{n}},W^{\varepsilon _{n}})$ is $\pi _{\varepsilon _{n}}$;

\item[(ii)] The probability law of $(\mathbf{u}_{0},\bar{W})$ is $\Pi $;

\item[(iii)] As $n\rightarrow \infty $, 
\begin{equation}
W^{\varepsilon _{n}}\rightarrow \bar{W}\text{ in }\mathcal{C}(0,T;\mathbb{R}%
^{m})\text{\ }\bar{\mathbb{P}}\text{-a.s.\ \ \ \ \ \ \ \ \ \ }  \label{6.20}
\end{equation}%
and

\item[(iv)] As $n\rightarrow \infty $, 
\begin{equation}
\mathbf{u}_{\varepsilon _{n}}\rightarrow \mathbf{u}_{0}\text{ in }%
L^{2}(0,T;H)\text{\ }\bar{\mathbb{P}}\text{-a.s.\ \ \ \ \ \ \ \ \ \ \ \ \ \
\ \ \ \ \ }  \label{6.21}
\end{equation}
\end{itemize}

We can see that $\left\{ W^{\varepsilon _{n}}\right\} $ is a sequence of $m$%
-dimensional standard Brownian motions. Let $\bar{\mathcal{F}}^{t}$ be the $%
\sigma $-algebra generated by $(\bar{W}(s),\mathbf{u}_{0}(s))$ ($0\leq s\leq
t$) and the null sets of $\bar{\mathcal{F}}$. Arguing as in \cite[Proof of
Theorem 1.1]{bensoussan2} we can show that $\bar{W}$ is an $\bar{\mathcal{F}}%
^{t}$-adapted standard $\mathbb{R}^{m}$-valued Wiener process. Also by the
same argument as in \cite[pp. 281-283]{bensoussan} we can show that, for all 
$\mathbf{v}\in V$ and for almost every $(\omega ,t)\in \bar{\Omega}\times
\lbrack 0,T]$ the following holds true 
\begin{equation}
\begin{array}{l}
\left( \mathbf{u}_{\varepsilon _{n}}(t),\mathbf{v}\right)
+\int_{0}^{t}\left( P^{\varepsilon }\mathbf{u}_{\varepsilon _{n}}(s)+%
\mathcal{A}^{\varepsilon }\mathbf{u}_{\varepsilon _{n}}(s)+B(\mathbf{u}%
_{\varepsilon _{n}}(s)),\mathbf{v}\right) ds=\left( \mathbf{u}^{0},\mathbf{v}%
\right)  \\ 
\ \ \ \ \ \ \ \ \ \ \ \ \ \ \ \ \ \ \ \ \ \ +\int_{0}^{t}\left( \mathbf{f}%
(s),\mathbf{v}\right) ds+\int_{0}^{t}\left( g^{\varepsilon }(\mathbf{u}%
_{\varepsilon _{n}}(s),\mathbf{v}\right) dW^{\varepsilon _{n}}(s)\text{.}%
\end{array}
\label{6.22}
\end{equation}

\subsection{Homogenization results}

\subsubsection{\textbf{Abstract formulation of the problem and preliminary
results}}

We begin this subsection by stating some important preliminary results
necessary to the homogenization process. The notations are those of the
preceding sections. It is worth noting that property (\ref{3.17}) in
Definition \ref{d3.4} still valid for $f\in B(\Omega ;\mathcal{C}(\overline{Q%
}_{T};B_{A}^{p^{\prime },\infty }))$ where $B_{A}^{p^{\prime },\infty
}=B_{A}^{p^{\prime }}\cap L^{\infty }(\mathbb{R}_{y,\tau }^{N+1})$ and as
usual, $p^{\prime }=p/(p-1)$.

Bearing this in mind, the question of homogenization of (\ref{6.11}) will
naturally arise from the following important assumption: 
\begin{equation}
\left\{ 
\begin{array}{l}
b\in B_{A}^{\infty }\text{ and }a_{ij},\;g_{k}(\cdot ,\cdot ,\mu )\in
B_{A}^{2},1\leq i,j\leq N,1\leq k\leq m \\ 
\text{for any }\mu \in \mathbb{R}^{N}%
\end{array}%
\right.  \label{6.25}
\end{equation}%
where $g=(g_{k})_{1\leq k\leq m}$.\medskip

The above hypothesis, which depends on the algebra wmv $A$, is crucial in
homogenization theory. It gives the structure of the coefficients of the
operator under consideration, and therefore allows one to pass to the limit.
Without such a hypothesis, one cannot perform the homogenization since the
convergence process relies heavily on the latter. The commonly assumption
used is the periodicity (obtained by taking the algebra to be the continuous
periodic functions). Hypothesis (\ref{6.25}) includes a variety of
behaviours, ranging from the periodicity to the weak almost periodicity (as
far as the ergodic algebras are concerned), and also encompassing all the
non ergodic algebras wmv. In this regard, this is a true advance in the
homogenization theory, regardless the applications either to PDEs or to
SPDEs.

Let $\Psi \in B(\Omega ;\mathcal{C}(\overline{Q}_{T};(A)^{N\times N}))$.
Suppose that (\ref{6.25}) is satisfied. It can be shown (as in \cite[%
Proposition 4.5]{NA}) that the function $(x,t,y,\tau ,\omega )\mapsto
b(y,\tau ,\Psi (x,t,y,\tau ,\omega ))$, denoted below by $b(\cdot ,\Psi )$,
belongs to $B(\Omega ;\mathcal{C}(\overline{Q}_{T};B_{A}^{p^{\prime },\infty
}))^{N\times N}$; assumption (\ref{6.25}) is crucially used in order to
obtain the above result. Likewise, the function $(x,t,y,\tau ,\omega
)\mapsto g_{k}(y,\tau ,\psi _{0}(x,t,\omega ))$ (for $\psi _{0}\in B(\Omega ;%
\mathcal{C}(\overline{Q}_{T};(A)^{N}))$) denoted by $g_{k}(\cdot ,\psi _{0})$%
, is an element of $B(\Omega ;\mathcal{C}(\overline{Q}_{T};B_{A}^{2,\infty
}))$. We may then define their traces 
\begin{equation*}
(x,t,\omega )\mapsto b\left( \frac{x}{\varepsilon },\frac{t}{\varepsilon }%
,\Psi \left( x,t,\frac{x}{\varepsilon },\frac{t}{\varepsilon },\omega
\right) \right) 
\end{equation*}%
and 
\begin{equation*}
(x,t,\omega )\mapsto g_{k}\left( \frac{x}{\varepsilon },\frac{t}{\varepsilon 
},\psi _{0}\left( x,t,\omega \right) \right) 
\end{equation*}%
from $Q_{T}\times \Omega $\ into $\mathbb{R}$, denoted respectively by $%
b^{\varepsilon }(\cdot ,\Psi ^{\varepsilon })$ and $g_{k}^{\varepsilon
}(\cdot ,\psi _{0})$, as elements of $L^{\infty }(Q_{T}\times \Omega )$. The
following result and its corollary can be proven exactly as its homologue in 
\cite{NgWou1} (see especially Proposition 3.1 therein).

\begin{proposition}
\label{p6.3}Let $3\leq p<\infty $. Suppose \emph{(\ref{6.25})} holds. For $%
\Psi \in B(\Omega ;\mathcal{C}(\overline{Q}_{T};(A)^{N\times N}))$ we have 
\begin{equation*}
b^{\varepsilon }(\cdot ,\Psi ^{\varepsilon })\rightarrow b(\cdot ,\Psi )%
\text{ in }L^{p^{\prime }}(Q_{T}\times \Omega )^{N\times N}\text{-weak }%
\Sigma \text{ as }\varepsilon \rightarrow 0.
\end{equation*}%
The mapping $\Psi \mapsto b(\cdot ,\Psi )$ of $B(\Omega ;\mathcal{C}(%
\overline{Q}_{T};(A)^{N\times N}))$ into $L^{p^{\prime }}(Q_{T}\times \Omega
;B_{A}^{p^{\prime }})^{N\times N}$ extends by continuity to a unique mapping
still denoted by $b$, of $L^{p}(Q_{T}\times \Omega ;(B_{A}^{p})^{N\times N})$
into $L^{p^{\prime }}(Q_{T}\times \Omega ;B_{A}^{p^{\prime }})^{N\times N}$
such that 
\begin{equation*}
(b(\cdot ,\mathbf{v})-b(\cdot ,\mathbf{w}))\cdot (\mathbf{v}-\mathbf{w})\geq
\nu _{1}\left\vert \mathbf{v}-\mathbf{w}\right\vert ^{p}\text{\ a.e. in }%
Q_{T}\times \Omega \times \mathbb{R}_{y}^{N}\times \mathbb{R}_{\tau }
\end{equation*}%
\begin{equation*}
\begin{array}{l}
\left\Vert b(\cdot ,\mathbf{v})-b(\cdot ,\mathbf{w})\right\Vert
_{L^{p^{\prime }}(Q_{T}\times \Omega ;B_{A}^{p^{\prime }})^{N\times N}} \\ 
\;\;\;\;\;\leq \nu _{2}\left\Vert \left\vert \mathbf{v}\right\vert
+\left\vert \mathbf{w}\right\vert \right\Vert _{L^{p}(Q_{T}\times \Omega
;B_{A}^{p})}^{p-2}\left\Vert \mathbf{v}-\mathbf{w}\right\Vert
_{L^{p}(Q_{T}\times \Omega ;(B_{A}^{p})^{N\times N})}%
\end{array}%
\end{equation*}%
\begin{equation*}
b(\cdot ,0)=0\;\;\text{a.e. in }\mathbb{R}_{y}^{N}\times \mathbb{R}_{\tau
}\;\;\;\;\;\;\;\;\;\;\;\;\;\;\;\;\;\;\;\;\;\;\;\;\;\;
\end{equation*}%
for all $\mathbf{v},\mathbf{w}\in L^{p}(Q_{T}\times \Omega
;(B_{A}^{p})^{N\times N})$.
\end{proposition}

\begin{corollary}
\label{c6.1}Let $\psi _{0}\in (B(\Omega )\otimes \mathcal{C}_{0}^{\infty
}(Q_{T}))^{N}$ and $\psi _{1}\in (B(\Omega )\otimes \mathcal{C}_{0}^{\infty
}(Q_{T})\otimes A^{\infty })^{N}$. For $\varepsilon >0$, let 
\begin{equation}
\Phi _{\varepsilon }=\psi _{0}+\varepsilon \psi _{1}^{\varepsilon
},\;\;\;\;\;\;\;\;\;\;\;\;\;\;  \label{6.26}
\end{equation}%
i.e., $\Phi _{\varepsilon }(x,t,\omega )=\psi _{0}(x,t,\omega )+\varepsilon
\psi _{1}(x,t,x/\varepsilon ,t/\varepsilon ,\omega )$\ for $(x,t,\omega )\in
Q_{T}\times \Omega $. Let $(v_{\varepsilon })_{\varepsilon \in E}$ is a
sequence in $L^{p}(Q_{T}\times \Omega )^{N\times N}$ such that $%
v_{\varepsilon }\rightarrow v_{0}$ in $L^{p}(Q_{T}\times \Omega )^{N\times
N} $-weak $\Sigma $ as $E\ni \varepsilon \rightarrow 0$ where $\mathbf{v}%
_{0}\in L^{p}(Q_{T}\times \Omega ;\mathcal{B}_{A}^{p})^{N\times N}$, then,
as $E\ni \varepsilon \rightarrow 0$, 
\begin{equation*}
\int_{Q_{T}\times \Omega }b^{\varepsilon }(\cdot ,D\Phi _{\varepsilon
})\cdot v_{\varepsilon }dxdtd\mathbb{P}\rightarrow \iint_{Q_{T}\times \Omega
\times \Delta (A)}\widehat{b}(\cdot ,D\psi _{0}+\partial \widehat{\psi }%
_{1})\cdot \widehat{v}_{0}dxdtd\mathbb{P}d\beta \text{.}
\end{equation*}
\end{corollary}

Also we will need the following important result in order to pass to the
limit in the stochastic term.

\begin{lemma}
\label{l5.1}Let $(\mathbf{u}_{\varepsilon })_{\varepsilon }$ be a sequence
in $L^{2}(Q_{T}\times \Omega )^{N}$ such that $\mathbf{u}_{\varepsilon
}\rightarrow \mathbf{u}_{0}$ in $L^{2}(Q_{T}\times \Omega )^{N}$ as $%
\varepsilon \rightarrow 0$ where $\mathbf{u}_{0}\in L^{2}(Q_{T}\times \Omega
)^{N}$. Then for each $1\leq k\leq m$ we have, 
\begin{equation*}
g_{k}^{\varepsilon }(\cdot ,\mathbf{u}_{\varepsilon })\rightarrow
g_{k}(\cdot ,\mathbf{u}_{0})\text{ in }L^{2}(Q_{T}\times \Omega )\text{-weak 
}\Sigma \text{ as }\varepsilon \rightarrow 0\text{.}
\end{equation*}
\end{lemma}

\begin{proof}
First of all, let $\mathbf{u}\in B(\Omega ;\mathcal{C}(\overline{Q}%
_{T}))^{N} $; as seen above, the function $(x,t,y,\tau ,\omega )\mapsto
g_{k}(y,\tau ,\mathbf{u}(x,t,\omega ))$ lies in $B(\Omega ;\mathcal{C}(%
\overline{Q}_{T};B_{A}^{2,\infty }))$, so that we have $g_{k}^{\varepsilon
}(\cdot ,\mathbf{u})\rightarrow g_{k}(\cdot ,\mathbf{u})$ in $%
L^{2}(Q_{T}\times \Omega )$-weak $\Sigma $ as $\varepsilon \rightarrow 0$.
Next, since $B(\Omega ;\mathcal{C}(\overline{Q}_{T}))$ is dense in $%
L^{2}(Q_{T}\times \Omega )$, it can be easily shown that 
\begin{equation}
g_{k}^{\varepsilon }(\cdot ,\mathbf{u}_{0})\rightarrow g_{k}(\cdot ,\mathbf{u%
}_{0})\text{ in }L^{2}(Q_{T}\times \Omega )\text{-weak }\Sigma \text{ as }%
\varepsilon \rightarrow 0.  \label{Eq2}
\end{equation}%
Now, let $f\in L^{2}(\Omega ;L^{2}(Q_{T};A))$; then 
\begin{eqnarray*}
&&\int_{Q_{T}\times \Omega }g_{k}^{\varepsilon }(\cdot ,\mathbf{u}%
_{\varepsilon })f^{\varepsilon }dxdtd\mathbb{P}-\iint_{Q_{T}\times \Omega
\times \Delta (A)}\widehat{g}_{k}(\cdot ,\mathbf{u}_{0})\widehat{f}dxdtd%
\mathbb{P}d\beta \\
&=&\int_{Q_{T}\times \Omega }(g_{k}^{\varepsilon }(\cdot ,\mathbf{u}%
_{\varepsilon })-g_{k}^{\varepsilon }(\cdot ,\mathbf{u}_{0}))f^{\varepsilon
}dxdtd\mathbb{P}+\int_{Q_{T}\times \Omega }g_{k}^{\varepsilon }(\cdot ,%
\mathbf{u}_{0})f^{\varepsilon }dxdtd\mathbb{P} \\
&&-\iint_{Q_{T}\times \Omega \times \Delta (A)}\widehat{g}_{k}(\cdot ,%
\mathbf{u}_{0})\widehat{f}dxdtd\mathbb{P}d\beta .
\end{eqnarray*}%
Using the inequality 
\begin{equation*}
\left| \int_{Q_{T}\times \Omega }(g_{k}^{\varepsilon }(\cdot ,\mathbf{u}%
_{\varepsilon })-g_{k}^{\varepsilon }(\cdot ,\mathbf{u}_{0}))f^{\varepsilon
}dxdtd\mathbb{P}\right| \leq C\left\| \mathbf{u}_{\varepsilon }-\mathbf{u}%
_{0}\right\| _{L^{2}(Q_{T}\times \Omega )^{N}}\left\| f^{\varepsilon
}\right\| _{L^{2}(Q_{T}\times \Omega )}
\end{equation*}%
in conjunction with (\ref{Eq2}) leads at once to the result.
\end{proof}

\begin{remark}
\label{r6.2}\emph{In view of the Lipschitz property on the function }$g_{k}$%
\emph{\ we may get more information on the limit of the sequence }$%
g_{k}^{\varepsilon }(\cdot ,\mathbf{u}_{\varepsilon })$\emph{. Indeed, since 
}$\left\vert g_{k}^{\varepsilon }(\cdot ,\mathbf{u}_{\varepsilon
})-g_{k}^{\varepsilon }(\cdot ,\mathbf{u}_{0})\right\vert \leq C\left\vert 
\mathbf{u}_{\varepsilon }-\mathbf{u}_{0}\right\vert $\emph{, we deduce the
following convergence result: }%
\begin{equation*}
g_{k}^{\varepsilon }(\cdot ,\mathbf{u}_{\varepsilon })\rightarrow \widetilde{%
g}_{k}(\mathbf{u}_{0})\text{\ in }L^{2}(Q_{T}\times \Omega )\text{ as }%
\varepsilon \rightarrow 0
\end{equation*}%
\emph{where }$\widetilde{g}_{k}(\mathbf{u}_{0})(x,t,\omega )=\int_{\Delta
(A)}\widehat{g}_{k}(s,s_{0},\mathbf{u}_{0}(x,t,\omega ))d\beta $\emph{, so
that we can derive the existence of a subsequence of }$g_{k}^{\varepsilon
}(\cdot ,\mathbf{u}_{\varepsilon })$\emph{\ that converges a.e. in }$%
Q_{T}\times \Omega $\emph{\ to }$\widetilde{g}_{k}(\mathbf{u}_{0})$\emph{.
For the sequel we shall need the following function: }$\widetilde{g}(\mathbf{%
u}_{0})=(\widetilde{g}_{k}(\mathbf{u}_{0}))_{1\leq k\leq m}$\emph{.}
\end{remark}

We end this subsection by collecting here below some function spaces that we
will make use in the sequel. We begin by noting that the space $B(\bar{\Omega%
})\otimes \mathcal{C}_{0}^{\infty }(0,T)\otimes \mathcal{V}$ is dense in $%
L^{p}(\bar{\Omega};L^{p}(0,T;V))$. Next, let the space 
\begin{equation*}
\mathcal{B}_{\Div}^{1,p}=\{\mathbf{u}\in (\mathcal{B}_{\#A_{y}}^{1,p})^{N}:~%
\overline{\Div}_{y}\mathbf{u}=0\}
\end{equation*}%
where $\overline{\Div}_{y}\mathbf{u}=\sum_{i=1}^{N}\overline{\partial }%
u^{i}/\partial y_{i}$, and its smooth counterpart 
\begin{equation*}
A_{y,\Div}^{\infty }=\{\mathbf{u}\in (\mathcal{D}_{A_{y}}(\mathbb{R}%
^{N})/I_{A_{y}}^{p})^{N}:~\overline{\Div}_{y}\mathbf{u}=0\}.
\end{equation*}%
The following result holds.

\begin{lemma}
\label{l6.2}The space $\varrho _{y}^{N}(A_{y,\Div}^{\infty })$ is dense in $%
\mathcal{B}_{\Div}^{1,p}$ where, for $\mathbf{u}=(u^{i})_{1\leq i\leq N}\in
(A_{y}^{\infty })^{N}$ we have $\varrho _{y}^{N}(\mathbf{u})=(\varrho
_{y}(u^{i}))_{1\leq i\leq N}$, $\varrho _{y}$ being the canonical mapping of 
$B_{A_{y}}^{p}$ into its separated completion $\mathcal{B}_{A_{y}}^{p}$.
\end{lemma}

\begin{proof}
This follows exactly in a same way as the proof of \cite[Lemma 2.3]{Wright1}.
\end{proof}

Now, let 
\begin{equation*}
\mathbb{F}_{0}^{1,p}=L^{p}(\bar{\Omega}\times (0,T);V)\times
L^{p}(Q_{T}\times \bar{\Omega};\mathcal{B}_{A_{\tau }}^{p}(\mathbb{R}_{\tau
};\mathcal{B}_{\Div}^{1,p}))
\end{equation*}%
and 
\begin{equation*}
\mathcal{F}_{0}^{\infty }=[B(\bar{\Omega})\otimes \mathcal{C}_{0}^{\infty
}(0,T)\otimes \mathcal{V}]\times \left[ B(\bar{\Omega})\otimes \mathcal{C}%
_{0}^{\infty }(Q_{T})\otimes \left( \mathcal{D}_{A_{\tau }}(\mathbb{R}_{\tau
})\otimes \varrho _{y}^{N}(A_{y,\Div}^{\infty })\right) \right] .
\end{equation*}%
Thanks to Lemma \ref{l6.2} we have the density of $\mathcal{F}_{0}^{\infty }$
in $\mathbb{F}_{0}^{1,p}$.

\subsubsection{\textbf{Homogenized problem}}

Let $(\mathbf{u}_{\varepsilon _{n}})_{n}$ be the sequence determined in the
Subsection \ref{subsect6.3} and satisfying Eq. (\ref{6.22}). Because of (\ref%
{6.22}) the sequence $(\mathbf{u}_{\varepsilon _{n}})_{n}$ also satisfies
the a priori estimates (\ref{6.13}), (\ref{6.13'}) and (\ref{6.14}).
Therefore, owing to the estimate (\ref{6.13}) (which yields the uniform
integrability of the sequence $(\mathbf{u}_{\varepsilon _{n}})_{n}$ with
respect to $\omega $) and the Vitali's theorem, we deduce from (\ref{6.21})
that, as $n\rightarrow \infty $, 
\begin{equation*}
\mathbf{u}_{\varepsilon _{n}}\rightarrow \mathbf{u}_{0}\text{ in }L^{2}(\bar{%
\Omega};L^{2}(0,T;H))
\end{equation*}%
and hence 
\begin{equation}
\mathbf{u}_{\varepsilon _{n}}\rightarrow \mathbf{u}_{0}\text{ in }%
L^{2}(Q_{T}\times \bar{\Omega})^{N}\text{ as }n\rightarrow \infty \text{.}
\label{6.23}
\end{equation}%
In view of (\ref{6.14}) and by the diagonal process, one can find a
subsequence of $(\mathbf{u}_{\varepsilon _{n}})_{n}$ (not relabeled) which
weakly converges in $L^{p}(\bar{\Omega};L^{p}(0,T;V))$ to the function $%
\mathbf{u}_{0}$ (this means that $\mathbf{u}_{0}\in L^{p}(\bar{\Omega}%
;L^{p}(0,T;V))$). From Theorem \ref{t3.8}, we infer the existence of a
function $\mathbf{u}_{1}=(u_{1}^{k})_{1\leq k\leq N}\in L^{p}(Q_{T}\times 
\bar{\Omega};\mathcal{B}_{A_{\tau }}^{p}(\mathbb{R}_{\tau };\mathcal{B}%
_{\#A_{y}}^{1,p})^{N})$ such that the convergence result 
\begin{equation}
\frac{\partial \mathbf{u}_{\varepsilon _{n}}}{\partial x_{i}}\rightarrow 
\frac{\partial \mathbf{u}_{0}}{\partial x_{i}}+\frac{\overline{\partial }%
\mathbf{u}_{1}}{\partial y_{i}}\text{ in }L^{p}(Q_{T}\times \bar{\Omega})^{N}%
\text{-weak }\Sigma \text{ }(1\leq i\leq N)  \label{6.24}
\end{equation}%
holds when $\varepsilon _{n}\rightarrow 0$. We recall that $\frac{\partial 
\mathbf{u}_{0}}{\partial x_{i}}=\left( \frac{\partial u_{0}^{k}}{\partial
x_{i}}\right) _{1\leq k\leq N}$ ($\mathbf{u}_{0}=(u_{0}^{k})_{1\leq k\leq N}$%
) and $\frac{\overline{\partial }\mathbf{u}_{1}}{\partial y_{i}}=\left( 
\frac{\overline{\partial }u_{1}^{k}}{\partial y_{i}}\right) _{1\leq k\leq N}$%
. Now, let us consider the following functionals: 
\begin{eqnarray*}
\widehat{a}_{I}(\mathbf{u},\mathbf{v})
&=&\sum_{i,j,k=1}^{N}\iint_{Q_{T}\times \bar{\Omega}\times \Delta (A)}%
\widehat{a}_{ij}(s,s_{0})\mathbb{D}_{j}u^{k}\mathbb{D}_{i}v^{k}dxdtd\bar{%
\mathbb{P}}d\beta \\
&&\ \ \ +\iint_{Q_{T}\times \bar{\Omega}\times \Delta (A)}\widehat{b}%
(s,s_{0},\mathbb{D}\mathbf{u})\cdot \mathbb{D}\mathbf{v}dxdtd\bar{\mathbb{P}}%
d\beta
\end{eqnarray*}%
where $\mathbb{D}_{j}u^{k}=\frac{\partial u_{0}^{k}}{\partial x_{j}}%
+\partial _{j}\widehat{u}_{1}^{k}$ ($\partial _{j}\widehat{u}_{1}^{k}=%
\mathcal{G}_{1}\left( \frac{\overline{\partial }u_{1}^{k}}{\partial y_{j}}%
\right) $, and the same definition for $\mathbb{D}_{i}v^{k}$) and $\mathbb{D}%
\mathbf{u}=(\mathbb{D}_{j}\mathbf{u})_{1\leq j\leq N}$ with $\mathbb{D}_{j}%
\mathbf{u}=(\mathbb{D}_{j}u^{k})_{1\leq k\leq N}$; 
\begin{equation*}
\widehat{b}_{I}(\mathbf{u},\mathbf{v},\mathbf{w})=\sum_{i,k=1}^{N}%
\iint_{Q_{T}\times \bar{\Omega}}u_{0}^{i}\frac{\partial v_{0}^{k}}{\partial
x_{i}}w_{0}^{k}dxdtd\bar{\mathbb{P}}
\end{equation*}%
for $\mathbf{u}=(\mathbf{u}_{0},\mathbf{u}_{1}),\mathbf{v}=(\mathbf{v}_{0},%
\mathbf{v}_{1}),\mathbf{w}=(\mathbf{w}_{0},\mathbf{w}_{1})\in \mathbb{F}%
_{0}^{1,p}$. The functionals $\widehat{a}_{I}$ and $\widehat{b}_{I}$ are
well-defined. Next, associated to these functionals is the variational
problem 
\begin{equation}
\left\{ 
\begin{array}{l}
\mathbf{u}=(\mathbf{u}_{0},\mathbf{u}_{1})\in \mathbb{F}_{0}^{1,p}: \\ 
-\int_{Q_{T}\times \bar{\Omega}}\mathbf{u}_{0}\cdot \mathbf{\psi }%
_{0}^{\prime }dxdtd\bar{\mathbb{P}}+\widehat{a}_{I}(\mathbf{u},\mathbf{\Phi }%
)+\widehat{b}_{I}(\mathbf{u},\mathbf{u},\mathbf{\Phi }) \\ 
\ \ \ \ =\int_{\bar{\Omega}}\int_{0}^{T}\left( \mathbf{f}(t),\mathbf{\psi }%
_{0}(t,\omega )\right) dtd\bar{\mathbb{P}}+\int_{\bar{\Omega}%
}\int_{0}^{T}\left( \widetilde{g}(\mathbf{u}_{0}),\mathbf{\psi }_{0}\right) d%
\bar{W}d\bar{\mathbb{P}} \\ 
\text{for all }\Phi =(\mathbf{\psi }_{0},\mathbf{\psi }_{1})\in \mathcal{F}%
_{0}^{\infty }\text{.}%
\end{array}%
\right.  \label{6.27}
\end{equation}

The following \textit{global} homogenization result holds.

\begin{theorem}
\label{t6.3}The couple $(\mathbf{u}_{0},\mathbf{u}_{1})$ determined by \emph{%
(\ref{6.23})-(\ref{6.24})} solves problem \emph{(\ref{6.27})}.
\end{theorem}

\begin{proof}
In what follows, we drop the index $n$ from the sequence $\varepsilon _{n}$.
So we will merely write $\varepsilon $ for $\varepsilon _{n}$. Now, from the
equality $\Div\mathbf{u}_{\varepsilon }=0$ we easily obtain that $\Div%
\mathbf{u}_{0}=0$ and $\overline{\Div}_{y}\mathbf{u}_{1}=0$, hence $\mathbf{u%
}=(\mathbf{u}_{0},\mathbf{u}_{1})\in \mathbb{F}_{0}^{1,p}$. It remains to
show that $\mathbf{u}$ solves (\ref{6.27}). For that, let $\mathbf{\Phi }=(%
\mathbf{\psi }_{0},\varrho ^{N}(\mathbf{\psi }_{1}))\in \mathcal{F}%
_{0}^{\infty }$; define $\mathbf{\Phi }_{\varepsilon }$ as in Corollary \ref%
{c6.1} (see (\ref{6.26}) therein), that is, as follows: 
\begin{equation*}
\mathbf{\Phi }_{\varepsilon }(x,t,\omega )=\mathbf{\psi }_{0}(x,t,\omega
)+\varepsilon \mathbf{\psi }_{1}\left( x,t,\frac{x}{\varepsilon },\frac{t}{%
\varepsilon },\omega \right) \text{ for }(x,t,\omega )\in Q_{T}\times \bar{%
\Omega}\text{.}
\end{equation*}%
Then we have $\mathbf{\Phi }_{\varepsilon }\in (B(\bar{\Omega})\otimes 
\mathcal{C}_{0}^{\infty }(Q_{T}))^{N}$ and, using $\mathbf{\Phi }%
_{\varepsilon }$ as a test function in the variational formulation of (\ref%
{6.22}) we get 
\begin{eqnarray}
&&-\int_{Q_{T}\times \bar{\Omega}}\mathbf{u}_{\varepsilon }\cdot \frac{%
\partial \mathbf{\Phi }_{\varepsilon }}{\partial t}dxdtd\bar{\mathbb{P}}%
+\int_{Q_{T}\times \bar{\Omega}}a^{\varepsilon }D\mathbf{u}_{\varepsilon
}\cdot D\mathbf{\Phi }_{\varepsilon }dxdtd\bar{\mathbb{P}}  \label{6.28} \\
&&+\int_{Q_{T}\times \bar{\Omega}}b^{\varepsilon }(\cdot ,D\mathbf{u}%
_{\varepsilon })\cdot D\mathbf{\Phi }_{\varepsilon }dxdtd\bar{\mathbb{P}}%
+\int_{\bar{\Omega}}\int_{0}^{T}b_{I}(\mathbf{u}_{\varepsilon },\mathbf{u}%
_{\varepsilon },\mathbf{\Phi }_{\varepsilon })dtd\bar{\mathbb{P}}  \notag \\
&=&\int_{0}^{T}\int_{\bar{\Omega}}\left( \mathbf{f}(t),\mathbf{\Phi }%
_{\varepsilon }\right) dtd\bar{\mathbb{P}}+\int_{0}^{T}\int_{\bar{\Omega}%
}\left( g^{\varepsilon }(\cdot ,\mathbf{u}_{\varepsilon }),\mathbf{\Phi }%
_{\varepsilon }\right) dW^{\varepsilon }d\bar{\mathbb{P}}.  \notag
\end{eqnarray}%
We pass to the limit in (\ref{6.28}) by considering each term separately.
First we have 
\begin{eqnarray*}
\int_{Q_{T}\times \bar{\Omega}}\mathbf{u}_{\varepsilon }\cdot \frac{\partial 
\mathbf{\Phi }_{\varepsilon }}{\partial t}dxdtd\bar{\mathbb{P}}
&=&\int_{Q_{T}\times \bar{\Omega}}\mathbf{u}_{\varepsilon }\cdot \frac{%
\partial \mathbf{\psi }_{0}}{\partial t}dxdtd\bar{\mathbb{P}}+\varepsilon
\int_{Q_{T}\times \bar{\Omega}}\mathbf{u}_{\varepsilon }\cdot \left( \frac{%
\partial \mathbf{\psi }_{1}}{\partial t}\right) ^{\varepsilon }dxdtd\bar{%
\mathbb{P}} \\
&&+\int_{Q_{T}\times \bar{\Omega}}\mathbf{u}_{\varepsilon }\cdot \left( 
\frac{\partial \mathbf{\psi }_{1}}{\partial \tau }\right) ^{\varepsilon
}dxdtd\bar{\mathbb{P}}.
\end{eqnarray*}%
But in view of (\ref{6.23}) coupling with the convergence result $(\partial 
\mathbf{\psi }_{1}/\partial \tau )^{\varepsilon }\rightarrow M(\partial 
\mathbf{\psi }_{1}/\partial \tau )=0$ in $L^{2}(Q_{T}\times \bar{\Omega})^{N}
$-weak, we obtain 
\begin{equation*}
\int_{Q_{T}\times \bar{\Omega}}\mathbf{u}_{\varepsilon }\cdot \frac{\partial 
\mathbf{\Phi }_{\varepsilon }}{\partial t}dxdtd\bar{\mathbb{P}}\rightarrow
\int_{Q_{T}\times \bar{\Omega}}\mathbf{u}_{0}\cdot \frac{\partial \mathbf{%
\psi }_{0}}{\partial t}dxdtd\bar{\mathbb{P}}.
\end{equation*}

Next, it is an usual well known fact that, using the convergence result (\ref%
{6.23}) together with the weak $\Sigma $-convergence of the sequence $(D%
\mathbf{\Phi }_{\varepsilon })_{\varepsilon }$ to $\mathbb{D}\mathbf{\Phi }$%
, we get 
\begin{equation*}
\int_{Q_{T}\times \bar{\Omega}}a^{\varepsilon }D\mathbf{u}_{\varepsilon
}\cdot D\mathbf{\Phi }_{\varepsilon }dxdtd\bar{\mathbb{P}}\rightarrow 
\widehat{a}_{I}(\mathbf{u},\Phi ).
\end{equation*}%
Considering the next term, we use the monotonicity property to have 
\begin{equation}
\int_{Q_{T}\times \bar{\Omega}}(b^{\varepsilon }(\cdot ,D\mathbf{u}%
_{\varepsilon })-b^{\varepsilon }(\cdot ,D\mathbf{\Phi }_{\varepsilon
}))\cdot (D\mathbf{u}_{\varepsilon }-D\mathbf{\Phi }_{\varepsilon })dxdtd%
\bar{\mathbb{P}}\geq 0\text{.}  \label{6.29}
\end{equation}%
Owing to the estimate (\ref{6.14}) (denoting by $\overline{\mathbb{E}}$ the
mathematical expectation on $(\bar{\Omega},\bar{\mathcal{F}},\bar{\mathbb{P}}%
)$) we infer that 
\begin{equation*}
\sup_{\varepsilon >0}\overline{\mathbb{E}}\left\| b^{\varepsilon }(\cdot ,D%
\mathbf{u}_{\varepsilon })\right\| _{L^{p^{\prime }}(Q_{T})^{N\times
N}}^{p^{\prime }}<\infty ,
\end{equation*}%
so that, from Theorem \ref{t3.7}, there exist a function $\chi \in
L^{p^{\prime }}(Q_{T}\times \bar{\Omega};\mathcal{B}_{A}^{p^{\prime
}})^{N\times N}$ and a subsequence of $\varepsilon $ not relabeled, such
that $b^{\varepsilon }(\cdot ,D\mathbf{u}_{\varepsilon })\rightarrow \chi $
in $L^{p^{\prime }}(Q_{T}\times \bar{\Omega})^{N\times N}$-weak $\Sigma $ as 
$\varepsilon \rightarrow 0$. We therefore pass to the limit in (\ref{6.29})
(as $\varepsilon \rightarrow 0$) using Corollary \ref{c6.1} to get 
\begin{equation}
\iint_{Q_{T}\times \bar{\Omega}\times \Delta (A)}(\widehat{\chi }-\widehat{b}%
(\cdot ,\mathbb{D}\Phi ))\cdot (\mathbb{D}\mathbf{u}-\mathbb{D}\mathbf{\Phi }%
)dxdtd\bar{\mathbb{P}}d\beta \geq 0  \label{6.30}
\end{equation}%
for any $\Phi \in \mathcal{F}_{0}^{\infty }$ where, as above, $\mathbb{D}%
\mathbf{u}=D\mathbf{u}_{0}+\partial \widehat{\mathbf{u}}_{1}$ ($\mathbf{u}=(%
\mathbf{u}_{0},\mathbf{u}_{1})$) and $\mathbb{D}\mathbf{\Phi }=D\mathbf{\psi 
}_{0}+\partial \widehat{\mathbf{\psi }}_{1}$. By the density of $\mathcal{F}%
_{0}^{\infty }$ in $\mathbb{F}_{0}^{1,p}$ and by a continuity argument, (\ref%
{6.30}) still holds for $\mathbf{\Phi }\in \mathbb{F}_{0}^{1,p}$. Hence by
taking $\mathbf{\Phi }=\mathbf{u}+\lambda \mathbf{v}$ for $\mathbf{v}=(%
\mathbf{v}_{0},\mathbf{v}_{1})\in \mathbb{F}_{0}^{1,p}$ and $\lambda >0$
arbitrarily fixed, we get 
\begin{equation*}
\lambda \iint_{Q_{T}\times \bar{\Omega}\times \Delta (A)}(\widehat{\chi }-%
\widehat{b}(\cdot ,\mathbb{D}\mathbf{u}+\lambda \mathbb{D}\mathbf{v}))\cdot 
\mathbb{D}\mathbf{v}dxdtd\bar{\mathbb{P}}d\beta \geq 0\;\text{for all }%
\mathbf{v}\in \mathbb{F}_{0}^{1,p}.
\end{equation*}%
Therefore by a mere routine, we deduce that $\chi =b(\cdot ,D\mathbf{u}_{0}+%
\overline{D}_{y}\mathbf{u}_{1})$.

The next point to check is to compute the $\lim_{\varepsilon \rightarrow
0}\int_{\bar{\Omega}}\int_{0}^{T}b_{I}(\mathbf{u}_{\varepsilon },\mathbf{u}%
_{\varepsilon },\mathbf{\Phi }_{\varepsilon })dtd\bar{\mathbb{P}}$. We claim
that, as $\varepsilon \rightarrow 0$, 
\begin{equation*}
\int_{\bar{\Omega}}\int_{0}^{T}b_{I}(\mathbf{u}_{\varepsilon },\mathbf{u}%
_{\varepsilon },\mathbf{\Phi }_{\varepsilon })dtd\bar{\mathbb{P}}\rightarrow 
\widehat{b}_{I}(\mathbf{u},\mathbf{u},\mathbf{\Phi }).
\end{equation*}%
Indeed, by the strong convergence result (\ref{6.23}) in conjunction with
the Theorem \ref{t3.3}, our claim is justified.

We obviously have that 
\begin{equation*}
\int_{0}^{T}\int_{\bar{\Omega}}\left( \mathbf{f}(t),\mathbf{\Phi }%
_{\varepsilon }(t,\omega )\right) dtd\bar{\mathbb{P}}\rightarrow
\int_{0}^{T}\int_{\bar{\Omega}}\left( \mathbf{f}(t),\mathbf{\psi }%
_{0}(t,\omega )\right) dtd\bar{\mathbb{P}}.
\end{equation*}%
The last point is concerned with the stochastic part $\int_{0}^{T}\int_{\bar{%
\Omega}}\left( g^{\varepsilon }(\cdot ,\mathbf{u}_{\varepsilon }),\mathbf{%
\Phi }_{\varepsilon }\right) dW^{\varepsilon }d\bar{\mathbb{P}}$.

But thanks to Remark \ref{r6.2} we get at once 
\begin{equation*}
\int_{0}^{T}\int_{\bar{\Omega}}\left( g^{\varepsilon }(\cdot ,\mathbf{u}%
_{\varepsilon }),\mathbf{\Phi }_{\varepsilon }\right) dW^{\varepsilon }d\bar{%
\mathbb{P}}\rightarrow \int_{\bar{\Omega}}\int_{0}^{T}\left( \widetilde{g}(%
\mathbf{u}_{0}),\mathbf{\psi }_{0}\right) d\bar{W}d\bar{\mathbb{P}}.
\end{equation*}%
It emerges from the above study that $\mathbf{u}=(\mathbf{u}_{0},\mathbf{u}%
_{1})$ satisfies (\ref{6.27}).
\end{proof}

In order to derive the homogenized problem, we need to deal with an
equivalent expression of problem (\ref{6.27}). As we can see, this problem
is equivalent to the following system (\ref{6.36})-(\ref{6.37}) reading as 
\begin{equation}
\left\{ 
\begin{array}{l}
\iint_{Q_{T}\times \bar{\Omega}\times \Delta (A)}\widehat{a}\mathbb{D}%
\mathbf{u}\cdot \partial \widehat{\mathbf{\psi }}_{1}dxdtd\bar{\mathbb{P}}%
d\beta +\iint_{Q_{T}\times \bar{\Omega}\times \Delta (A)}\widehat{b}(\cdot ,%
\mathbb{D}\mathbf{u})\cdot \partial \widehat{\mathbf{\psi }}_{1}dxdtd\bar{%
\mathbb{P}}d\beta =0 \\ 
\text{for all }\mathbf{\psi }_{1}\in B(\bar{\Omega})\otimes \mathcal{C}%
_{0}^{\infty }(Q_{T})\otimes \lbrack \mathcal{D}_{A_{\tau }}(\mathbb{R}%
_{\tau })\otimes \varrho _{y}^{N}(A_{y,\text{div}}^{\infty })]\text{;}%
\end{array}%
\right.   \label{6.36}
\end{equation}%
\begin{equation}
\left\{ 
\begin{array}{l}
-\int_{Q_{T}\times \bar{\Omega}}\mathbf{u}_{0}\cdot \mathbf{\psi }%
_{0}^{\prime }dxdtd\bar{\mathbb{P}}+\widehat{a}_{I}(\mathbf{u},(\mathbf{\psi 
}_{0},0))+\widehat{b}_{I}(\mathbf{u},\mathbf{u},(\mathbf{\psi }_{0},0)) \\ 
\ \ \ \ =\int_{\bar{\Omega}}\int_{0}^{T}\left( \mathbf{f}(t),\mathbf{\psi }%
_{0}(t,\omega )\right) dtd\bar{\mathbb{P}}+\int_{\bar{\Omega}%
}\int_{0}^{T}\left( \widetilde{g}(\mathbf{u}_{0}),\mathbf{\psi }_{0}\right) d%
\bar{W}d\bar{\mathbb{P}} \\ 
\text{for all }\mathbf{\psi }_{0}\in B(\bar{\Omega})\otimes \mathcal{C}%
_{0}^{\infty }(0,T)\otimes \mathcal{V}\text{.}%
\end{array}%
\right.   \label{6.37}
\end{equation}%
It is an easy matter to deal with (\ref{6.36}). In fact, fix $\xi \in 
\mathbb{R}^{N\times N}$ and consider the following cell problem: 
\begin{equation}
\left\{ 
\begin{array}{l}
\pi (\xi )\in \mathcal{B}_{A_{\tau }}^{p}(\mathbb{R}_{\tau };\mathcal{B}_{%
\text{div}}^{1,p}): \\ 
\int_{\Delta (A)}\widehat{a}(\xi +\partial \widehat{\pi }(\xi ))\cdot
\partial \widehat{w}d\beta +\int_{\Delta (A)}\widehat{b}(\cdot ,\xi
+\partial \widehat{\pi }(\xi ))\cdot \partial \widehat{w}d\beta =0 \\ 
\text{for all }w\in \mathcal{B}_{A_{\tau }}^{p}(\mathbb{R}_{\tau };\mathcal{B%
}_{\text{div}}^{1,p})\text{.}%
\end{array}%
\right.   \label{6.38}
\end{equation}%
Due to the properties of the functions $a$ and $b$, Eq. (\ref{6.38}) admits
at least a solution (see e.g., \cite[Chap. 2]{Lions}). But if $\pi _{1}=\pi
_{1}(\xi )$ and $\pi _{2}=\pi _{2}(\xi )$ are two solutions of (\ref{6.38})
then, setting $\pi =\pi _{1}-\pi _{2}$, 
\begin{equation*}
\left\{ 
\begin{array}{l}
\int_{\Delta (A)}\widehat{a}\partial \widehat{\pi }\cdot \partial \widehat{w}%
d\beta +\int_{\Delta (A)}(\widehat{b}(\cdot ,\xi +\partial \widehat{\pi }%
_{1})-\widehat{b}(\cdot ,\xi +\partial \widehat{\pi }_{2}))\cdot \partial 
\widehat{w}d\beta =0 \\ 
\text{for all }w\in \mathcal{B}_{A_{\tau }}^{p}(\mathbb{R}_{\tau };\mathcal{B%
}_{\text{div}}^{1,p})\text{.}%
\end{array}%
\right. 
\end{equation*}%
Taking the particular test function $w=\pi $, we are led to 
\begin{equation*}
\int_{\Delta (A)}\widehat{a}\partial \widehat{\pi }\cdot \partial \widehat{%
\pi }d\beta +\int_{\Delta (A)}(\widehat{b}(\cdot ,\xi +\partial \widehat{\pi 
}_{1})-\widehat{b}(\cdot ,\xi +\partial \widehat{\pi }_{2}))\cdot \partial 
\widehat{\pi }d\beta =0,
\end{equation*}%
and using once again the properties of $a$ and $b$ (see in particular the
Proposition \ref{p6.3}), we get 
\begin{equation*}
\nu _{0}\int_{\Delta (A)}\left\vert \partial \widehat{\pi }\right\vert
^{2}d\beta +\nu _{1}\int_{\Delta (A)}\left\vert \partial \widehat{\pi }%
\right\vert ^{p}d\beta =0,
\end{equation*}%
which gives $\partial \widehat{\pi }=0$, or equivalently, $\overline{D}%
_{y}\pi =0$. It then follows that $\pi =0$ since it belong to $\mathcal{B}%
_{A_{\tau }}^{p}(\mathbb{R}_{\tau };(\mathcal{B}_{\#A_{y}}^{1,p})^{N})$.

Now, choosing $\mathbf{\psi }_{1}=\phi \otimes \varphi \otimes w$ in (\ref%
{6.36}) with $\phi \in B(\bar{\Omega})$, $\varphi \in \mathcal{C}%
_{0}^{\infty }(Q_{T})$ and $w\in \lbrack \mathcal{D}_{A_{\tau }}(\mathbb{R}%
_{\tau })\otimes \varrho _{y}^{N}(A_{y,\text{div}}^{\infty })]$, we obtain
by disintegration the following equation: 
\begin{equation}
\left\{ 
\begin{array}{l}
\int_{\Delta (A)}\widehat{a}\mathbb{D}\mathbf{u}\cdot \partial \widehat{w}%
d\beta +\int_{\Delta (A)}\widehat{b}(\cdot ,\mathbb{D}\mathbf{u})\cdot
\partial \widehat{w}d\beta =0 \\ 
\text{for all }w\in \mathcal{D}_{A_{\tau }}(\mathbb{R}_{\tau })\otimes
\varrho _{y}^{N}(A_{y,\text{div}}^{\infty })\text{.}%
\end{array}%
\right.  \label{6.39}
\end{equation}%
Coming back to (\ref{6.38}) we choose there $\xi =D\mathbf{u}_{0}(x,t,\omega
)$ (for arbitrarily fixed $(x,t,\omega )\in Q_{T}\times \bar{\Omega}$).
Comparing the resulting equation with (\ref{6.39}) and using the density of $%
\mathcal{D}_{A_{\tau }}(\mathbb{R}_{\tau })\otimes \varrho _{y}^{N}(A_{y,%
\text{div}}^{\infty })$ in $\mathcal{B}_{A_{\tau }}^{p}(\mathbb{R}_{\tau };%
\mathcal{B}_{\text{div}}^{1,p})$ we get by the uniqueness of the solution of
(\ref{6.38}) that $\mathbf{u}_{1}=\pi (D\mathbf{u}_{0})$, where $\pi (D%
\mathbf{u}_{0})$ stands for the function $(x,t,\omega )\mapsto \pi (D\mathbf{%
u}_{0}(x,t,\omega ))$ defined from $Q_{T}\times \bar{\Omega}$ into $\mathcal{%
B}_{A_{\tau }}^{p}(\mathbb{R}_{\tau };\mathcal{B}_{\text{div}}^{1,p})$. This
shows the uniqueness of the solution of (\ref{6.36}).

As for (\ref{6.37}), let, for $\xi \in \mathbb{R}^{N\times N}$, 
\begin{equation*}
M(\xi )=\int_{\Delta (A)}\widehat{b}(\cdot ,\xi +\partial \widehat{\pi }(\xi
))d\beta
\end{equation*}%
and 
\begin{equation*}
\mathsf{m}\xi =\int_{\Delta (A)}\widehat{a}(\xi +\partial \widehat{\pi }(\xi
))d\beta .
\end{equation*}%
Then substituting $\mathbf{u}_{1}=\pi (D\mathbf{u}_{0})$ in (\ref{6.37}) and
choosing there the special test function $\mathbf{\psi }_{0}(x,t,\omega
)=\phi (\omega )\chi (t)\varphi (x)$ with $\phi \in B(\bar{\Omega})$, $\chi
\in \mathcal{C}_{0}^{\infty }(0,T)$ and $\varphi \in \mathcal{V}$, we
quickly obtain by It\^{o}'s formula, the macroscopic homogenized problem
(which holds as an equality in $V^{\prime }$)%
\begin{equation}
\left\{ 
\begin{array}{l}
d\mathbf{u}_{0}+(-\Div(\mathsf{m}D\mathbf{u}_{0})-\Div M(D\mathbf{u}_{0})+B(%
\mathbf{u}_{0}))dt=\mathbf{f}dt+\widetilde{g}(\mathbf{u}_{0})d\bar{W} \\ 
\mathbf{u}_{0}(0)=\mathbf{u}^{0}.%
\end{array}%
\right.  \label{6.40}
\end{equation}%
Since the above problem is of the same type as (\ref{6.11}) the existence
and the uniqueness of its solution is ensured by the same arguments as (\ref%
{6.11}). We therefore have the following

\begin{theorem}
\label{t6.4}Assume that \emph{(\ref{6.2})-(\ref{6.7})} hold. Moreover
suppose that \emph{(\ref{6.25})} holds true. Let $3\leq p<\infty $. For each 
$\varepsilon >0$ let $\mathbf{u}_{\varepsilon }$ be the unique solution of 
\emph{(\ref{6.11})} on a given stochastic system $(\Omega ,\mathcal{F},%
\mathbb{P}),\mathcal{F}^{t},W$ defined as in Section \emph{4}. Then as $%
\varepsilon \rightarrow 0$, the whole sequence $\mathbf{u}_{\varepsilon }$
converges in probability to $\mathbf{u}_{0}$ in $L^{2}(Q_{T})^{N}$ (i.e., $||%
\mathbf{u}_{\varepsilon }-\mathbf{u}_{0}||_{L^{2}(Q_{T})^{N}}$ converges to
zero in probability) where $\mathbf{u}_{0}$ is the unique strong
probabilistic solution of \emph{(\ref{6.40})} with $\bar{W}$ replaced by $W$.
\end{theorem}

\begin{proof}
The proof of this Theorem is copied on that of \cite[Theorem 8]{WoukengArxiv}%
. Note that it relies on Lemma \ref{l6.0}.
\end{proof}

\begin{remark}
\label{r6.3}\emph{As regard the deterministic setting (i.e. when }$g\equiv 0$%
\emph{), this is the first time that the homogenization of (\ref{6.11}) (and
hence (\ref{6.1})) is considered. So even in the deterministic framework,
our results are new, and we can formulated the deterministic counterpart of
Theorem \ref{t6.4} as follows.}
\end{remark}

\begin{theorem}
\label{t6.5}Assume that \emph{(\ref{6.2})-(\ref{6.7})} hold. Moreover
suppose that \emph{(\ref{6.25})} holds true. Let $3\leq p<\infty $. For each 
$\varepsilon >0$ let $\mathbf{u}_{\varepsilon }$ be the unique solution of 
\emph{(\ref{6.11})}. Then as $\varepsilon \rightarrow 0$, 
\begin{equation*}
\mathbf{u}_{\varepsilon }\rightarrow \mathbf{u}_{0}\text{\ in }%
L^{2}(Q_{T})^{N}
\end{equation*}%
where $\mathbf{u}_{0}$ is the unique solution of the following problem
(viewed as an equality of functions with values in $V^{\prime }$)%
\begin{equation*}
\left\{ 
\begin{array}{l}
\frac{\partial \mathbf{u}_{0}}{\partial t}-\Div(\mathsf{m}D\mathbf{u}_{0})-%
\Div M(D\mathbf{u}_{0})+B(\mathbf{u}_{0})=\mathbf{f} \\ 
\mathbf{u}_{0}(0)=\mathbf{u}^{0}.%
\end{array}%
\right.
\end{equation*}
\end{theorem}

Of course the conclusion of the above theorem can be achieved using less
technicalities than as in the present study. We do not prove it. For the
reader's convenience, its proof does not use the famous Prokhorov and
Skorohod theorems, but a rather less complicated Aubin-Lions' theorem.

\subsection{Some concrete applications of the results of the previous
subsection}

A look at the previous subsection reveals that the homogenization process
has been made possible because of the assumption (\ref{6.25}) which was
fundamental in the said subsection. This assumption is formulated in a
general fashion encompassing a variety of concrete behaviours as regard the
coefficients of the operator involved in (\ref{6.1}). We aim at providing in
this subsection some natural situations leading to the homogenization of (%
\ref{6.1}).

\begin{example}
\label{e6.1}\emph{The homogenization of (\ref{6.1}) can be achieved under
the periodicity assumption }

\begin{itemize}
\item[(\protect\ref{6.25})$_{1}$] \emph{The functions }$b_{i}(\cdot ,\cdot
,\lambda )$\emph{, }$a_{ij}$\emph{\ and }$g_{k}(\cdot ,\cdot ,\mu )$\emph{\
are both periodic of period }$1$\emph{\ in each scalar coordinate.}

\noindent \emph{This leads to (\ref{6.25}) with }$A=\mathcal{C}_{\text{\emph{%
per}}}(Y\times Z)=\mathcal{C}_{\text{\emph{per}}}(Y)\odot \mathcal{C}_{\text{%
\emph{per}}}(Z)$\emph{\ (the product algebra, with }$Y=(0,1)^{N}$\emph{\ and 
}$Z=(0,1)$\emph{), and hence }$B_{A}^{r}=L_{\text{\emph{per}}}^{r}(Y\times
Z) $\emph{\ for }$1\leq r\leq \infty $\emph{.}
\end{itemize}
\end{example}

\begin{example}
\label{e6.2}\emph{The above functions in (\ref{6.25})}$_{1}$\emph{\ are both
Besicovitch almost periodic in }$(y,\tau )$\emph{. This amounts to (\ref%
{6.25}) with }$A=AP(\mathbb{R}_{y,\tau }^{N+1})=AP(\mathbb{R}_{y}^{N})\odot
AP(\mathbb{R}_{\tau })$\emph{\ (}$AP(\mathbb{R}_{y}^{N})$\emph{\ the Bohr
almost periodic functions on }$R_{y}^{N}$\emph{).}
\end{example}

\begin{example}
\label{e6.3}\emph{The homogenization problem for (\ref{6.1}) can be may be
considered under the assumption }

\begin{itemize}
\item[(\protect\ref{6.25})$_{2}$] $b_{i}(\cdot ,\cdot ,\lambda )$\emph{\ is
weakly almost periodic while the functions }$a_{ij}$\emph{\ and }$%
g_{k}(\cdot ,\cdot ,\mu )$\emph{\ are almost periodic in the Besicovitch
sense. This yields (\ref{6.25}) with }$A=WAP(\mathbb{R}_{y}^{N})\odot WAP(%
\mathbb{R}_{\tau })$\emph{\ (}$WAP(\mathbb{R}_{y}^{N})$\emph{, the algebra
of continuous weakly almost periodic functions on }$\mathbb{R}_{y}^{N}$\emph{%
; see e.g., \cite{17}).}
\end{itemize}
\end{example}

\begin{example}
\label{e6.4}\emph{Let }$A_{\tau }$\emph{\ be the algebra of Example \ref{e2}
(see Subsection \ref{subsect5.3}). It is known that }$A_{\tau }$\emph{\ is
not ergodic \cite[p. 243]{20}. We may however study the homogenization
problem for (\ref{6.1}) under the assumption that }

\begin{itemize}
\item[(\protect\ref{6.25})$_{3}$] $b_{i}(y,\cdot ,\lambda )\in B_{A_{\tau
}}^{p^{\prime }}$\emph{\ and }$b_{i}(\cdot ,\tau ,\lambda )$\emph{\ is
weakly almost periodic; }$a_{ij}$\emph{\ is periodic and }$g_{k}(\cdot
,\cdot ,\mu )$\emph{\ is almost periodic.}
\end{itemize}

\emph{This assumption is more involved. In fact, let }$A_{1,\tau }$\emph{\
be the algebra generated by }$AP(\mathbb{R}_{\tau })\cup A_{\tau }$\emph{.
It is a fact that }$A_{1,\tau }$\emph{\ is an algebra wmv on }$\mathbb{R}%
_{\tau }$\emph{\ which is not ergodic. Next, let }$A=WAP(\mathbb{R}%
_{y}^{N})\odot A_{1,\tau }$\emph{. Then, also }$A$\emph{\ is not ergodic,
and it can be easily shown that (\ref{6.25}) is satisfied with the above }$A$%
\emph{.}
\end{example}

Many other examples can be considered. We may also consider an example
involving only non ergodic algebras by taking for example $A$ to be $N+1$
copies of the $A_{\tau }$'s above: $A=A_{\tau }\odot \ldots \odot A_{\tau }$%
, $N+1$ times, which gives a non ergodic algebra on $\mathbb{R}^{N+1}$.

\end{document}